\documentclass{amsart}

\usepackage[utf8]{inputenc}
\usepackage[T1]{fontenc}
\usepackage{mathtools,amssymb,amsthm} 
\usepackage[all]{xy}
\usepackage{color}
\usepackage{url}
\usepackage[pagebackref=true,colorlinks=true, pdfstartview=FitV, linkcolor=blue,citecolor=red, urlcolor=blue]{hyperref}

\newtheorem{thm}{Theorem}[section]
\newtheorem{cor}[thm]{Corollary}
\newtheorem{lem}[thm]{Lemma}
\newtheorem{prop}[thm]{Proposition}
\theoremstyle{definition}

\theoremstyle{definition}
\newtheorem{defi}[thm]{Definition}
\theoremstyle{remark}
\newtheorem{rem}[thm]{Remark}

\newcommand{\bburl}[1]{\textcolor{blue}{\url{#1}}}

\numberwithin{equation}{section}

\newcommand\hyph{\nobreak\hskip0pt-\nobreak\hskip0pt\relax}

\DeclareMathOperator{\Sp}{Sp}
\DeclareMathOperator{\SL}{SL}
\DeclareMathOperator{\GL}{GL}

\DeclareMathOperator{\PGL}{PGL}

\DeclareMathOperator{\Hom}{Hom}
\DeclareMathOperator{\Gal}{Gal}
\DeclareMathOperator{\Ext}{Ext}

\newcommand{\FF}{F} %% p-adic field F
\newcommand{\EE}{E} %% quadratic extension of F
\newcommand{\cont}[1]{#1^{\vee}} %% contragredient
\newcommand{\sdual}[1]{#1^{\sigmaEF}} %% contragredient
\newcommand{\sigmaEF}{\sigma} %% non-trivial F-automorphism of E
\newcommand{\omegaEF}{\omega_{\EE/\FF}} %% non-trivial F-automorphism of E

\newcommand{\Fl}{\overline{\mathbf{F}}_{\ell}}
\newcommand{\Ql}{\overline{\mathbf{Q}}_{\ell}}
\newcommand{\Zl}{\overline{\mathbf{Z}}_{\ell}}

\DeclareMathOperator{\cInd}{c-Ind}
\DeclareMathOperator{\Ind}{Ind}

\DeclareMathOperator{\Nm}{Nm}

\DeclareMathOperator{\Irr}{Irr}
\DeclareMathOperator{\Nilp}{Nilp}

\DeclareMathOperator{\Scusp}{Scusp}
\DeclareMathOperator{\MScusp}{MScusp}
\DeclareMathOperator{\Mod}{Mod}
\DeclareMathOperator{\tr}{tr}
\DeclareMathOperator{\im}{Im}

\DeclareMathOperator{\WDRep}{WDRep}

\newcommand{\tomegaEF}{\widetilde{\omega}_{\EE/\FF}} %% non-trivial F-automorphism of E

\DeclareMathOperator{\St}{St} %Steinberg
\DeclareMathOperator{\Spe}{Sp} %Special

\DeclareMathOperator{\Rep}{Rep}

\title{Modulo $\ell$ distinction problems}

\author{Peiyi Cui}
\address{Peiyi Cui, Morningside Center of Mathematics, Chinese Academy of Sciences, Zhongguancun East Road, Haidian District, Beijing 100190, P.R.C.}
\email{\textcolor{blue}{\href{peiyi.cuimath@gmail.com}{peiyi.cuimath@gmail.com}}}

\author{Thomas Lanard}
\address{Thomas Lanard, CNRS, Laboratoire de mathématiques de Versailles, Université Paris-Saclay, UVSQ, 78000, Versailles, France}
\email{\textcolor{blue}{\href{thomas.lanard@uvsq.fr}{thomas.lanard@uvsq.fr}}}

\author{Hengfei Lu}
\address{Hengfei Lu, School of Mathematical Sciences, Beihang University, 9 Nansan Street, Shahe Higher Education Park, Changping, Beijing 102206, P.R.C.}
\email{\textcolor{blue}{\href{luhengfei@buaa.edu.cn}{luhengfei@buaa.edu.cn}}}

\makeatletter
\@namedef{subjclassname@2020}{%
  \textup{2020} Mathematics Subject Classification}
\makeatother



\date{\today}

\begin{document}

\begin{abstract}
    Let $\FF$ be a non-archimedean local field of characteristic different from 2 and residual characteristic $p$. This paper concerns the $\ell$-modular representations of a connected reductive group $G$ distinguished by a Galois involution, with $\ell$ an odd prime different from $p$. We start by proving a general theorem allowing to lift supercuspidal $\Fl$-representations of $\GL_n(F)$ distinguished by an arbitrary closed subgroup $H$ to a distinguished supercuspidal $\Ql$-representation. Given a quadratic field extension $E/F$ and an irreducible $\Fl$-representation $\pi$ of $\GL_n(E)$, we verify the Jacquet conjecture in the modular setting that if the Langlands parameter $\phi_\pi$ is irreducible and conjugate-self-dual, then $\pi$ is either $\GL_n(F)$-distinguished or $(\GL_{n}(F),\omega_{E/F})$-distinguished (where $\omegaEF$ is the quadratic character of $F^\times$ associated to the quadratic field extension $E/F$ by the local class field theory), but not both, which extends one result of S\'echerre to the case $p=2$. We give another application of our lifting theorem for supercuspidal representations distinguished by a unitary involution, extending one result of Zou to $p=2$.
    After that, we give a complete classification of the $\GL_2(F)$-distinguished representations of $\GL_2(E)$. Using this classification we discuss a modular version of the Prasad conjecture for $\PGL_2$. We show that the ``classical'' Prasad conjecture fails in the modular setting. We propose a solution using non-nilpotent Weil-Deligne representations. Finally, we apply the restriction method of Anandavardhanan and Prasad to classify the $\SL_2(F)$-distinguished {modular} representations of $\SL_2(E)$.
\end{abstract}

\maketitle

\tableofcontents

\section{Introduction}

Let $\FF$ be a non-archimedean local field of characteristic different from 2 and residual characteristic $p$ and $G$ the $\FF$-points of a reductive group defined over $\FF$. Let $R$ be an algebraically closed field of characteristic different from $p$. A representation $\pi$ of a group $G$, with coefficients in $R$, is said to be \emph{distinguished} with respect to a subgroup $H$ of $G$ if it admits a non-trivial $H$-invariant linear form. More generally, if $\chi$ is a character of $H$, we will say that $\pi$ is $(H,\chi)$-distinguished if there exists a non-trivial linear functional on the space of $\pi$ on which $H$ acts via $\chi$, i.e. $f:\pi \to R$ such that
\[ f(\pi(h)v)=\chi(h)f( v)\]
for all $h \in H$ and $v \in \pi$.

Distinguished representations are central objects in the study of the \emph{relative Langlands program}. The distinction problem is closely related to the Langlands functorial conjectures and it should be possible to characterize the $(H,\chi)$\hyph distinguished representations as the images with respect to a functorial transfer to $G$ from a third group $G'$ in many cases. There exists a rich literature, such as \cite{anandavardhanan2003distinguished,S-V,lu2018pacific,lu2020ANT}, trying to classify all $H$-distinguished complex representations of $G$. Furthermore, when $\theta$ is the Galois involution of order $2$ and $H=G^{\theta},$ Dipendra Prasad \cite{prasad2015arelative} constructs $\chi_H$ a quadratic character of $H$ and gives a precise conjecture for the multiplicity $\dim_\mathbb{C}\Hom_H(\pi,\chi_H)$ in terms of the enhanced Langlands parameter of $\pi$.

\bigskip

More recently, mathematicians have been interested in modular representations of $p$-adic groups, which are smooth representations with coefficients in $\Fl$, with $\ell \neq p$. The theory of $\ell$-modular representations has been initiated by Vignéras in \cite{vigneras}. These works are motivated by a modular local Langlands program and studying congruences between automorphic forms (which have been used to prove many remarkable theorems of arithmetic geometry).

\bigskip

For modular representations, much remains to be done for the distinction problems.  Sécherre and Venketasubramanian have examined the pair $(G,H) = (\GL_n(F), \GL_{n-1}(F))$ in \cite{SecherreVenketasubramanian}. Sécherre \cite{vincent2019supercusp} also investigated the pair $(\GL_{n}(E),\GL_{n}(F))$ for supercuspidal representations with $p$ odd, where $E$ is a quadratic extension of $F$. Ongoing work by Kurinczuk, Matringe and Sécherre aims at extending the results of the pair $(\GL_{n}(E), \GL_{n}(F))$ to all representations (see \cite{KMS}).

\bigskip

In this paper, we consider a quadratic field extension $\EE/\FF$ of locally compact non-archimedean local fields of characteristic different from 2 and residual characteristic $p$. Denote by $q_F$ (resp. $q_E$)
the cardinality of the residue field of $\FF$ (resp. $\EE$). Let $\ell$ be an odd prime different from $p$. We are interested in the distinction problems for the pairs $(\GL_{n}(E),\GL_{n}(F))$ and $(\SL_{2}(E),\SL_{2}(F))$. We also give a modular version of the Prasad conjecture for $\PGL_2$.

\bigskip

We would like to point out that the ongoing work \cite{KMS} of Kurinczuk, Matringe and Sécherre also gives a classification of all irreducible $\GL_2(F)$-distinguished representations of $\GL_{2}(E)$ (at least for $p\neq 2$). However, our methods and theirs are completely different. They use the gamma factors and the epsilon factors whereas in this article we use Mackey theory.

\subsection{Distinguished supercuspidal representations of $\GL_n$}

A $\Ql$-representation $(\tilde{\pi},\tilde{V})$ of a $p$-adic group $G$ is said to be \emph{$\ell$-integral} if it admits a free $\Zl$-submodule $L$ of $\tilde{V}$, stable by $G$, such that $L \otimes_{\Zl} \Ql = \tilde{V}$. If $(\pi,V)$ is an irreducible $\ell$-integral $\Ql$-representation and $L$ is a $\Zl[G]$-lattice of $V$ of finite type, then the semi-simplification of $L \otimes_{\Zl} \Fl$ doesn't depend on the choice of $L$ and is called the \emph{reduction modulo $\ell$} of $\pi$ (we refer to \cite{vigneras} for more details). If $\pi$ is an $\Fl$-representation of $G$, an $\ell$-integral irreducible $\Ql$-representation $\tilde{\pi}$ is called a $\Ql$-lift of $\pi$ if its reduction modulo $\ell$ is $\pi$.

Our first result, discussed in Section \ref{secSupercuspGLn}, is the following.

\begin{thm}[Theorem \ref{thmLiftDistinguished}]
    Let $H$ be a closed subgroup of $\GL_n(F)$. Let $\pi$ be a supercuspidal $\Fl$-representation of $\GL_n(F)$ which is $H$-distinguished. Then there exists  a $\Ql$-lift $\tilde{\pi}$ of $\pi$ such that it is supercuspidal and distinguished by $H$.\label{thm1.1}
\end{thm}

\begin{rem}
    Theorem \ref{thm1.1}  is valid even for $F$ of characteristic 2.
\end{rem}

The proof depends on the type theory and the existence of a projective envelope for supercuspidal types. This is a very general result allowing us to transfer the problem from modular representations to complex representations. As an immediate consequence, we get:

\begin{cor} [Corollary \ref{symplecticvanishing}]
    There is no supercuspidal $\Fl$-representation of $\GL_{2n}(F)$ distinguished by $\Sp_{2n}(F)$.
\end{cor}

We also give two other applications. The first one, is to study supercuspidal representations of $\GL_{n}(E)$ distinguished by a Galois involution. For $\ell \neq 2$, we prove the \emph{Dichotomy Theorem} and the \emph{Disjunction Theorem}. Let $\omegaEF$ be the quadratic character of $\FF^{\times}$ associated to the field extension $\EE/\FF$ by the local class field theory and let $\sigmaEF$ be the non-trivial $\FF$-automorphism of $\EE$. A representation $\pi$ of $\GL_n(\EE)$ is said $\sigmaEF$-selfdual if $\cont{\pi} \simeq \sdual{\pi}$, where $\cont{\pi}$ is the contragredient representation of $\pi$.

% Another application, is to extend the results of \cite{vincent2019supercusp} for supercuspidal representations of $\GL_{n}(E)$ distinguished by a Galois involution to the case $p=2$. For $\ell \neq 2$, we prove the \emph{Dichotomy Theorem} and the \emph{Disjunction Theorem}. Let $\omegaEF$ be the class field character of $\EE/\FF$ and let $\sigmaEF$ be the non-trivial $\FF$-automorphism of $\EE$.

\begin{thm}[Theorem \ref{thmDichotomy}]
    Let $\pi$ be a supercuspidal $\Fl$-representation of $\GL_n(\EE)$. Then $\pi$ is $\sigmaEF$-selfdual if and only if it is distinguished or $\omegaEF$-distinguished by $\GL_n(F)$. Moreover, $\pi$ cannot be both distinguished and $\omegaEF$-distinguished by $\GL_n(F)$.
\end{thm}
For the complex representations, it is called the Jacquet conjecture proved by Kable in \cite{kable}. We first show that one can lift $\sigmaEF$-selfdual supercuspidal $\Fl$-representations to $\sigmaEF$-selfdual supercuspidal  $\Ql$-representations. This follows from the fact that since $\ell$ is odd, the number of inertial classes of such lifts is odd, hence has a fixed point under $\sigmaEF$-duality. Then we prove that all these supercuspidal lifts share the same sign due to Schur's Lemma. Combining with Theorem \ref{thm1.1}, this allows us to deduce the Jacquet conjecture for modular representations from complex representations.

\begin{rem}
    When $p\neq 2$, the distinguished supercuspidal representations were studied by Sécherre in \cite{vincent2019supercusp} (with no restriction on $\ell$). Our method is different from the one used in \cite{vincent2019supercusp} and works for $\ell \neq 2$ but for any $p$. Therefore, combining the results in this paper and in \cite{vincent2019supercusp} gives the complete result for distinguished supercuspidal representations for all $\ell$ and $p$ with $\ell \neq p$.
\end{rem}

We also have a characterization using the local Langlands correspondence of Vignéras.

\begin{prop}[Proposition \ref{proDistConj}]
    Let $\pi$ be a $\sigmaEF$-selfdual supercuspidal $\Fl$\hyph representation of $\GL_n(\EE)$ and $\varphi$ its Langlands parameter. Then $\pi$ is distinguished (resp. $\omegaEF$-distinguished) by $\GL_n(F)$ if and only if $\varphi_\pi$ is conjugate-orthogonal (resp. conjugate-symplectic).
\end{prop}

The last application, is for representations distinguished by a unitary involution. Let $\ell\neq 2$ and $\varsigma$ be a \emph{unitary involution} of $\GL_n(\EE)$, that is $\varsigma(g)=\varepsilon \sigma({}^{t}g^{-1}) \varepsilon^{-1}$ where $\varepsilon$ is a hermitian matrix in $M_n(\EE)$.

\begin{thm}[Theorem \ref{thmDistUnitary}]
    Let $\pi$ be a supercuspidal $\Fl$-representation of $\GL_n(\EE)$ and $\varsigma$ a unitary involution of $G$. Then $\pi$ is distinguished by $\GL_n(\EE)^{\varsigma}$ if and only if $\pi^{\sigma}\simeq \pi$.
\end{thm}

This extends the result of \cite{Zou} to the case $p=2$ when $F$ has characteristic zero.

\subsection{Modular distinguished representations for $(\GL_{2}(E),\GL_{2}(F))$}

In section \ref{secGL2}, we complete the classification of all representations of $\GL_{2}(E)$ distinguished by $\GL_{2}(F)$.

Let $\chi_1,\chi_2$ be two characters of $\EE^{\times}$. We use Mackey theory to show that the principal series representation $\pi(\chi_1,\chi_2)$ is distinguished by $\GL_2(F)$ if and only if either $\chi_1\chi_2^\sigma=\mathbf{1}$ or $\chi_1|_{F^\times}=\chi_2|_{F^\times}=\mathbf{1}$ with $\chi_1\neq\chi_2$. (See Lemma \ref{lemPrincSer} for more details.)

To get a complete classification, we are just missing irreducible subquotients of non-irreducible principal series. Denote by $\nu^{1/2}$ the unramified character of $\EE^{\times}$ sending a uniformizer to $q_\EE^{1/2}$  where the square root $q_\EE^{1/2}$ of $q_\EE$ is fixed. Let $\chi$ be a character of $\EE^{\times}$. When $q_{\EE} \not\equiv -1 \pmod{\ell}$, we denote by $\St_{\chi}$ the twisted Steinberg representation, that is the unique generic irreducible subquotient of $\pi(\chi \nu^{-1/2},\chi \nu^{1/2})$. In the case $q_{\EE} \equiv -1 \pmod{\ell}$, the generic irreducible subquotient of $\pi(\chi \nu^{-1/2},\chi \nu^{1/2})$ is the special representation, denoted by $\Spe_{\chi}$.

\begin{thm}[Theorem \ref{thmCuspDist}]
    Let $\chi$ be a character of $\EE^{\times}$.
    \label{thm1.8}
    \begin{enumerate}
        \item If $\ell \nmid q_{\EE}^2-1$, then $\St_\chi$ is $\GL_2(F)$-distinguished if and only if $\chi_{|{F^\times}}=\omegaEF$.
        \item If $\ell \mid q_{\EE}+1$, then $\Spe_\chi$ is $\GL_2(F)$-distinguished if and only if $\ell \mid q_{\FF}+1$ and $\chi_{|{F^\times}}=\omegaEF$ or $\nu^{1/2}_{|{F^\times}}$.
        \item If $\ell \mid q_{\EE}-1$, then $\St_\chi$ is $\GL_2(F)$-distinguished if and only if $\chi_{|{F^\times}}=\omegaEF$ or $\chi_{|{F^\times}}=\mathbf{1}$ with $\ell \mid q_{\FF}-1$.
    \end{enumerate}
\end{thm}

\begin{rem}
    From Theorem \ref{thm1.8}, we would like to highlight the fact that we have in the modular case new phenomena that do not appear in the complex setting.
    \begin{itemize}
        \item When $q_{\FF}\equiv 1 \pmod{\ell}$, the Steinberg representation $\St$ is both $\GL_2(F)$-distinguished and $(\GL_2(F),\omega_{E/F})$-distinguished.
        \item When $q_{\EE}\equiv -1 \pmod{\ell}$ and $\EE/\FF$ is unramified (that is $q_{\FF}^2\equiv -1 \pmod{\ell}$ and so $q_{\FF}\not\equiv -1 \pmod{\ell}$) the special representation $\Spe$ is neither $\GL_2(F)$-distinguished nor $(\GL_2(F),\omega_{E/F})$\hyph distinguished. This has been mentioned by Vincent Sécherre in \cite[Rem. 2.8]{vincent2019supercusp}.
        \item When $q_{\EE} \equiv -1 \pmod{\ell}$ and $\EE/\FF$ is ramified (that is $q_{\FF} \equiv -1 \pmod{\ell}$) the special representation $\Spe$ is $(\GL_2(F),\omega_{E/F})$-distinguished and is also $(\GL_2(F),\nu^{1/2}_{|{F^\times}})$-distinguished.
    \end{itemize}
\end{rem}

\subsection{The modulo $\ell$ Prasad conjecture}

In section \ref{secPrasad}, we discuss the Prasad conjecture for modular representations. In \cite{prasad2015arelative}, Dipendra Prasad proposed a conjecture for the multiplicity $\dim_\mathbb{C} \Hom_{G(F)}(\pi,\chi_G) $ under the local Langlands conjecture, where $G$ is a quasi-split reductive group defined over $F$, $\pi$ is an irreducible smooth representation of $G(E)$ lying in a generic $L$-packet and $\chi_G$ is a quadratic character depending on $G$ and the quadratic extension $E/F$. Since the local Langlands correspondence for the $\ell$-modular representations of $G(F)$ has not been set up in general, except for $G=\GL_n$, we are concerned only with the simplest case where $G=\PGL_2$.

\bigskip

The modulo $\ell$ local Langlands correspondence for $\GL_2$ has been defined by Vignéras in \cite{VignerasLanglands}, and is a map $V : \Irr_R(\GL_2(\EE)) \to \Nilp_{R}(W_{\EE},\GL_2)$, where $\Nilp_{R}(W_{\EE},\GL_2)$ denotes the set of isomorphism classes of nilpotent semisimple Weil-Deligne representations. Let $\pi$ be an irreducible representation of $\GL_2(\EE)$ with central character $\omega_{\pi}$. Since $\det V(\pi)=\omega_{\pi}$ as characters of $W_E$, it induces a map
\[
    PV : \Irr_R(\PGL_2(\EE)) \to \Nilp_{R}(W_{\EE},\SL_2).
\]
Using this correspondence, the Prasad conjecture is not valid for $\ell$-modular representations  (at least in the non-banal case). Actually, it does not work for any map $L : \Irr_R(\PGL_2(\EE)) \to \Nilp_{R}(W_{\EE},\SL_2)$ having the same semisimple part as $PV$ (see Section \ref{secIssuePrasad}).

\bigskip

In order to propose a solution for $\ell$-modular representations, we will consider non-nilpotent Weil-Deligne representations as in \cite{MatringeKurinczuk}. Let $\WDRep_{R}(W_{\EE},\GL_2)$ denote the set of isomorphism classes of semisimple Weil-Deligne representation of $\GL_2$. Kurinczuk and Matringe define in \cite{MatringeKurinczuk} an equivalence relation $\sim$ on $\WDRep_{R}(W_{\EE},\GL_2)$, and we denote by $[\WDRep_{R}(W_{\EE},\GL_2)]$ the quotient set.

Then we define an injection from $\Nilp_{\Fl}(W_{\EE},\SL_2)$ to $[\WDRep_{\Fl}(W_{\EE},\SL_2)]$ (which is not the trivial one in the non-banal case) in the following way. Let $\chi$ be a quadratic character of $\EE^\times$. If $\ell \mid q_\EE-1$, we denote by $\Psi_{\St,\chi} \in \Nilp_{\Fl}(W_{\EE},\SL_2)$, the Weil-Deligne representation $\Psi_{\St,\chi} = (\chi\nu^{-1/2} \oplus \chi\nu^{-1/2},N)$ with $N = \begin{pmatrix}
        0 & 1 \\
        0 & 0
    \end{pmatrix}$. If $\ell \mid q_\EE+1$, let $\Psi_{\Spe,\chi} := (\chi\nu^{-1/2} \oplus \chi\nu^{1/2},0)$. We define an injection

\[P:  \Nilp_{\Fl}(W_{\EE},\SL_2) \hookrightarrow [\WDRep_{\Fl}(W_{\EE},\SL_2)]\]
by
\[
    P(\Psi) = \left\{
    \begin{array}{ll}

        {\left[\chi\nu^{-1/2} \oplus \chi\nu^{-1/2},\begin{pmatrix}
                                                                0 & 1 \\
                                                                1 & 0
                                                            \end{pmatrix}\right]} & \mbox{if } \ell \mid q_\EE-1 \mbox{ and } \Psi = \Psi_{         \St,\chi} \\
        {\left[\chi\nu^{-1/2} \oplus \chi\nu^{1/2},\begin{pmatrix}
                                                           0 & 1 \\
                                                           1 & 0
                                                       \end{pmatrix}\right]}  & \mbox{if } \ell \mid q_\EE+1 \mbox{ and } \Psi = \Psi_{\Spe,\chi}         \\

        [\Psi]                                                      & \mbox{otherwise.}
    \end{array}
    \right.
\]

Then we prove a modulo $\ell$  version of the Prasad conjecture, using our modified injection $P$:

\begin{thm}[Theorem \ref{thmPrasadModular}]
    Let $\pi$ be an irreducible generic $\Fl$-representation of $\PGL_2(\EE)$. Then $\pi$ is $\omegaEF$-distinguished by $\PGL_2(F)$ if and only if there exists $\Psi_{F} \in \WDRep_{\Fl}(W_{\FF},\SL_2)$ such that $\Psi_{F|W_{\EE}} \sim P\circ PV(\pi)$.
\end{thm}

\begin{rem}
    \begin{enumerate}
        \item In the banal case i.e., $\ell \nmid q_\EE^2-1$, $[\WDRep_{\Fl}(W_{\EE},\SL_2)]=\Nilp_{\Fl}(W_{\EE},\SL_2)$ and $P$ is the trivial injection, so we found the ``classical'' Prasad conjecture.
        \item When $\ell \mid q_\EE+1$, $\ell \mid q_\FF + 1$ and $\pi = \Spe_\chi$ for  a quadratic character $\chi$ of $\EE^{\times}$  such that $\chi_{|\FF^\times}=\omegaEF \nu_{|\FF^\times}^{1/2}$, then the representation $\Psi_F$ is not the Langlands parameter of any representation of $\PGL_2(\FF)$ (nor it is in the image of $P$).
    \end{enumerate}
\end{rem}

\subsection{Modular distinguished representations for $(\SL_{2}(E),\SL_{2}(F))$}

In the last part, Section \ref{secSL2}, we use our classification of $\GL_2(F)$-distinguished representation of $\GL_2(\EE)$ and the restriction method of \cite{anandavardhanan2003distinguished} to classify $\SL_2(F)$-distinguished representations for $\SL_{2}(E)$ in the modular setting.

\bigskip

Let $\pi$ be an irreducible $\Fl$-representation of $\GL_2(E)$. Denote by $\lg(\pi)$ the length of $\pi\vert_{\SL_2(E)}$ and by $\lg_{+}(\pi)$ the length of $\pi\vert_{\GL_2^{+}(E)}$, where $\GL_2^{+}(E)$ is the subgroup of $\GL_2(E)$, consisting {of} elements whose determinants belong to $F^{\times}E^{\times2}$. Let $\mathfrak{p}_{E}$ be a uniformizer of $E$, and $\mathfrak{o}_E$ be the ring of integers of $E$. We fix an $\Fl$-character $\psi_0$ of $E$, which is non-trivial on $\mathfrak{o}_E$ and is trivial on both $\mathfrak{p}_E$ and $F$.

\bigskip

For supercuspidal representations, we will adapt the methods of \cite{anandavardhanan2003distinguished} to prove the following theorems.

\begin{thm}[Theorem \ref{thmDistGeneric}]
    Let $\pi$ be an irreducible supercuspidal $\Fl$\hyph representation of $\GL_2(E)$ distinguished by $\SL_2(F)$, and $\pi^+$ the unique irreducible component of $\pi\vert_{\GL_2^{+}(E)}$ that is $\psi_0$-generic. Then $\pi^+$ is distinguished by $\SL_2(F)$. Furthermore, let $\tau$ be an irreducible component of $\pi\vert_{\SL_2(E)}$, distinguished by $\SL_2(F)$. Then $\tau$ is an irreducible component of $\pi^+\vert_{\SL_2(E)}$.
\end{thm}

\begin{thm}[Theorem \ref{thmDistSupercuspSL}]
    Let $\pi$ be an irreducible supercuspidal $\Fl$\hyph representation of $\GL_2(E)$, and $\tau$  an irreducible component of $\pi\vert_{\SL_2(E)}$. Suppose $\tau$ is distinguished by $\SL_2(F)$. Then,
    \[
        \mathrm{dim}\mathrm{Hom}{_{\SL_2(F)}}(\tau,\mathbf{1})= \left\{
        \begin{array}{ll}

            1, & \mbox{if } \pi\vert_{\SL_2(E)}\cong\tau;           \\
            1, & \mbox{if } \lg_{+}(\pi)=2 \mbox{ and } \lg(\pi)=4; \\
            2, & \mbox{if } \lg_{+}(\pi)=\lg(\pi)=2;                \\
            4, & \mbox{if } \lg_{+}(\pi)=\lg(\pi)=4.
        \end{array}
        \right.
    \]

    The {first case and the} last case arise only when $p=2$.
\end{thm}

We also have a criterion for principal series representations. See Theorem \ref{prin} for more details.

\bigskip

Along the way, we prove the existence of a good lift for supercuspidal representations of $\SL_2(E)$. Let $\tau$ be an irreducible cuspidal $\Fl$-representation of $\mathrm{SL}_2(E)$, and $\tilde{\tau}$ an irreducible cuspidal $\Ql$-representation of $\mathrm{SL}_2(E)$, which is $\ell$-integral. We say that $\tilde{\tau}$ is a \emph{good $\Ql$-lift} of $\tau$, if the reduction modulo $\ell$ of $\tilde{\tau}$ is irreducible and isomorphic to $\tau$. We prove in Proposition \ref{propQoodLifting} that any supercuspidal representation of $\SL_2(E)$ always admits a good $\Ql$-lift.

\section*{Acknowledgements}

We warmly thank Vincent Sécherre for his remarks and comments which made it possible, in particular, to generalize the results from $\GL_2$ to $\GL_n$. We also thank Nadir Matringe and Alberto M\'inguez for useful discussions. We thank an anonymous referee for his useful comments and corrections. The main results were obtained and the paper was written when we were working together at the University of Vienna as postdocs. We would like to thank the University of Vienna for providing such nice working environment. The work of T. L. was partially supported by the ERC, grant agreement 101001159.
H. L. was partially supported by NSFC 12301031.

\section{Notation}

Let $\EE/\FF$ be a quadratic extension of locally compact non-archimedean local fields of characteristic different from 2 and residual characteristic $p$. Denote by $q_F$ (resp. $q_E$) the cardinality of the residue field of $\FF$ (resp. $\EE$). Let $W_F$ be the Weil group of $\FF$.
We denote by $\sigmaEF$ the non-trivial $\FF$-automorphism of $\EE$.

Denote by $\omegaEF$ the quadratic character of $F^\times$ associated to the quadratic field extension $E/F$ by the local class field theory. We may identify the characters of $W_{\FF}$ and the characters of $\FF^\times$ by the local class field theory. Let $\Nm_{E/F}$ be the norm from $E$ to $F$.

For $n \geq 1$, we denote by $\GL_n$ the general linear group, and by $\SL_n$ the special linear group.
Denote the standard Borel subgroup of $\GL_n$ by $B_n$. Let $\pi_i$ (i=1,2) be an irreducible representation of $\GL_{n_i}(F)$. Denote by $\pi_1\boxtimes\pi_2$ the tensor product representation of $\GL_{n_1}(F)\times\GL_{n_2}(F)$. For a character $\chi$ of $F^{\times}$, we will also denote by $\chi$ the character $\chi\circ\det$ of $\GL_n(F)$. The trivial character will be denoted by $\mathbf{1}$. Let $\nu_n$ be the character $g\mapsto|\det(g)|$ of $\GL_n(F)$.

Let $H$ be a subgroup of $G$. Let $\pi$ be an irreducible representation of $G$.
We say that $\pi$ is $H$-distinguished if $\Hom_H(\pi,\mathbf{1})\neq0$. In case that the subgroup $H$ is clear, we say that $\pi$ is distinguished sometimes.

In this article, $\ell$ is an odd prime number different from $p$ (except for Section \ref{secLiftDist}, where $\ell=2$ is also allowed).

\begin{rem}
    This paper is written for $\FF$ of characteristic different from 2. However, our main theorem, Theorem \ref{thmLiftDistinguished}, which allows lifting $\ell$-modular supercuspidal distinguished representations to $\ell$-adic distinguished representations works without any assumption on the characteristic. Therefore, in Section \ref{secLiftDist} only, no assumption on the characteristic of $\FF$ will be made. After that, we always assume the characteristic to be different from 2 because we use results from the theory of complex representations which are only written with this assumption. For instance, the Jacquet conjecture follows from \cite{kable}, written in characteristic zero, but which also works for positive odd characteristic by \cite[Appendix A]{AKMSS} (or \cite{Jo}).
\end{rem}

\section{Distinguished supercuspidal representations of \texorpdfstring{$\GL_n$}{GLn}}

\label{secSupercuspGLn}

In this section, we study the distinguished supercuspidal representation of $\GL_n$. Firstly, we prove a lifting theorem for supercuspidal representations distinguished by an arbitrary closed subgroup $H$. This theorem allows us to study the supercuspidal representations of $\GL_n(E)$ distinguished by a Galois involution, and extend the result of \cite{vincent2019supercusp} to the case $p=2$. We also give an application to supercuspidal representations of $\GL_n(E)$ distinguished by unitary groups.
%it is also called the Galois involution.

\subsection{Lifting of modular distinguished representation}
\label{secLiftDist}

In this section, $G=\GL_n(\FF)$ and $H$ is an arbitrary closed subgroup of $G$. In this part only, $\FF$ could be of any characteristic and we do not require the prime $\ell$ to be odd, just $\ell \neq p$. We want to prove that a supercuspidal $H$-distinguished $\Fl$-representation can be lifted to a supercuspidal $H$-distinguished $\Ql$-representation. This allows bringing the modular distinction problems to the complex setting. To do that, we {use} type theory and the existence of a projective envelope.

\bigskip

We start with a few lemmas. Let $K$ be a finite field extension of $\FF$. We denote by $\varpi_K$ (resp. $\varpi_F$) a uniformizer of $K$ (resp. $\FF$). Let $\mathcal{O}_K$ be the ring of integers of $K$, and for an integer $i\geq 1$, we denote by $U_K^i$ the subgroup of $\mathcal{O}_K^{\times}$ defined by $U_K^i:=\{1+\varpi_K^{i} u, u \in \mathcal{O}_K\}$.

\begin{lem}
    \label{lemUi}
    Let $i\geq 1$ and $x\in U_K^i$. Then $x^p \in U_K^{i+1}$.
\end{lem}

\begin{proof}
    Let us write $x = 1 + \varpi_K^{i} u$, with $u \in \mathcal{O}_K$. The binomial expansion gives us $x^p=1 + p \varpi_K^{i} u + \varpi_K^{2i} u'$, with $u' \in \mathcal{O}_K$. But $\varpi_K$ divides $p$ in $\mathcal{O}_K$ so $x^p \in U_K^{i+1}$.
\end{proof}

\begin{lem}
    \label{lemVarpiN}
    Let $N \in \mathbb{N}^{*}$. Then there exist two integers $m\geq 1$ and $s\geq 1$ such that $\varpi_K^m \in \varpi_\FF^s U_K^N$.
\end{lem}

\begin{proof}
    Let $e$ be the ramification index of $K$ over $\FF$. Then $\varpi_K^{e} = \varpi_\FF u$, with $u \in \mathcal{O}_K^{\times}$. Now $\mathcal{O}_K^{\times} / U_K^1$ is isomorphic to $k_K^{\times}$ the residue field of $K$ of cardinal $q_K-1$. Therefore $x:=u^{q_K-1} \in U_K^1$ and from Lemma \ref{lemUi}, $x^{p^{N}} \in U_K^N$. We get the result with $m=e(q_K-1)p^N$ and $s=(q_K-1)p^N$.
\end{proof}

Let $G:=\GL_n(\FF)$. Let $\pi$ be a supercuspidal $\Fl$-representation of $G$. Let us write $\pi = \cInd_{J}^{G}(\Lambda)$ (compact induction), for an extended maximal simple type $(J,\Lambda)$ (for the definition see Section 3.1 \cite{MinguezSecherre} in which it is called ``un type simple maximal \'etendu''). Let $J^{0}$ be the unique maximal compact subgroup of $J$, $J^1$ its maximal normal pro-$p$-subgroup. There exists a field extension $K$ of $\FF$, such that $J = K^{\times}J^0=\langle \varpi_K ,J^0 \rangle$. The restricted representation $\lambda:=\Lambda\vert_{J^0}$ is irreducible, and the pair $(J^0,\lambda)$ is a maximal simple type. In particular, we have an isomorphism $\lambda\cong\kappa\otimes\rho$, where $\kappa$ is an irreducible representation of $J^0$ with some technical conditions and $\rho$ is inflated from an irreducible supercuspidal representation of ${J^0\slash J^1}$.

\bigskip

Since $F^{\times}$ is central in $J$, $\varpi_F$ acts on $\Lambda$ as a scalar $\bar{\alpha} \in \Fl^\times$. For $R=\Fl$ or $\Zl$ and $\beta \in R^\times$, we denote by $\Rep_{R}^{\beta}(J)$ the category\footnote{The category $\Rep_{R}^\beta(J)$ is the subcategory of the category of smooth representations of $J$, which contains a full subcategory composed of smooth representations of $J$ with a fixed central character on $F^\times$.}  of smooth $R$-representations of $J$ such that $\varpi_F$ acts as the scalar $\beta$.

\begin{prop}
    \label{propProjEnvType}
    \begin{enumerate}
        \item The representation $\Lambda$ admits a projective envelope $\overline{\mathcal{P}}$ in $\Rep_{\Fl}^{\bar{\alpha}}(J)$.
        \item Let $\alpha \in \Zl$ be a lift of $\bar{\alpha}$. Then there exists a unique (up to isomorphism) projective object $\mathcal{P}$ in $\Rep_{\Zl}^{\alpha}(J)$ such that $\mathcal{P} \otimes \Fl = \overline{\mathcal{P}}$.
    \end{enumerate}

    %The representation $\Lambda$ admits a projective envelope in the category of finite length $\Fl$-representations of $J$.
\end{prop}

\begin{proof}
    \begin{enumerate}
        \item We apply a similar strategy as in \cite[Section 4]{SecherreDrevon}. Let $\eta$ be a character of $J$, trivial on $J^0$ such that $\eta(\varpi_F)=\bar{\alpha}$. Let $\Sigma:=\Lambda \eta^{-1}$ such that $\varpi_F$ acts trivially on $\Sigma$.

              Since $\Sigma$ is smooth and $J^0$ is compact, there exists an open subgroup $H^2$ of $J^1$ such that $H^2 \subset \ker(\Sigma_{|J^0})$ and $J^0/H^2$ is finite. Since $H^2$ is open in $J^0$, there exists an integer $N\geq 1$ such that $U_K^N \subset H^2$. By Lemma \ref{lemVarpiN}, there exist $m,n\geq 1$ such that $\varpi_K^m \in \varpi_\FF^s U_K^N$. We define a group $Q$ by
              \[
                  Q=J/\langle H^2, \varpi_F \rangle.
              \]
              The group $Q$ is finite.  Moreover, since $\varpi_K^m \in \varpi_\FF^s H^2$, we have that $\langle H^2, \varpi_K^m \rangle=\langle H^2, \varpi_F^s \rangle$. Therefore $\Sigma$ is trivial on $\langle H^2, \varpi_F \rangle$ and we can consider $\Sigma$ as a representation of $Q$.

              \bigskip

              Since $Q$ is a finite group, by \cite[Prop. 41]{Serre} we can consider the projective envelope of $\Sigma$ in the category of $\Fl[Q]$-modules, denoted by $\overline{\mathcal{P}}$. We can regard $\overline{\mathcal{P}}$ as a representation of $J$ by inflation. Let
              \[
                  Q_1= J/ \langle \varpi_F\rangle,
              \]
              which is compact and $Q$ is a quotient of $Q_1$. Since $H^2$ is pro-$p$, a projective object of $Q$ remains projective as a representation of $Q_1$. Now we show that $\overline{\mathcal{P}}$ is projective in $\Rep_{\Fl}^{1}(J)$. It is due to the fact that $\Rep_{\Fl}^{1}(J)$ is equivalent to the category of smooth $\Fl$-representations of $Q_1$. For arbitrary $\bar{\alpha}$, it follows from the fact that $\Rep_{\Fl}^{1}(J)$ and $\Rep_{\Fl}^{\bar{\alpha}}(J)$ are equivalent via twisting $\eta$ with each object in $\Rep_{\Fl}^{1}(J)$.

        \item By \cite[Prop. 42]{Serre}, there is a unique projective $\Zl[Q]$-module $\mathcal{P}$ such that $\mathcal{P} \otimes \Fl = \overline{\mathcal{P}}$. Like in the first part, $\mathcal{P}$ can be lifted to a projective object in $\Rep_{\Zl}^{\alpha}(J)$. If $\mathcal{P}'$ is another projective in $\Rep_{\Zl}^{\alpha}(J)$ with the same properties, then, by projectivity, any isomorphism from $\overline{\mathcal{P}}$ to $\mathcal{P}' \otimes \Fl$ can be lifted to an isomorphism from $\mathcal{P}$ to $\mathcal{P}'$.

              %Since $H^2$ is a pro-$p$-group and $\varpi_F^s$ is central in $J$, the proof of \cite[Lem. 4.5]{SecherreDrevon} works with $\langle H^2, \varpi_F^s \rangle$. In particular, \cite[Lem. 4.6]{SecherreDrevon} shows that $\mathcal{P}$ is the projective envelope of $\Sigma$ in the category of finite length $\Fl$-representations of $J$. By \cite[Lem. 4.7]{SecherreDrevon} $\mathcal{P}\eta$ is the projective envelope of $\Sigma\eta = \Lambda$ and we get the result.

    \end{enumerate}
\end{proof}

Using this projective envelope, we can prove the desired theorem.

\begin{thm}
    \label{thmLiftDistinguished}
    Let $\FF$ be a non-archimedean local field. Let $H$ be a closed subgroup of $\GL_n(F)$. Let $\pi$ be a supercuspidal $\Fl$-representation of $\GL_n(F)$ which is $H$-distinguished. Then there exists a $\Ql$-lift  $\tilde{\pi}$ of $\pi$ such that it is supercuspidal and distinguished by $H$.
\end{thm}

\begin{proof}
    We write $\pi = \cInd_{J}^{G}(\Lambda)$, for an extended maximal simple type $(J,\Lambda)$. We have $\Lambda\in \Rep_{\Fl}^{\bar{\alpha}}(J)$ for $\bar{\alpha}\in\Fl^{\times}$. Let $\bar{\chi}$ be an $\Fl$-character of $F^{\times}$ such that $\bar{\chi}$ is trivial on $\mathcal{O}_F^{\times}$ and $\bar{\chi}(\varpi_F)=\bar{\alpha}$. Frobenius reciprocity and Mackey formula give us
    \[
        0\neq \Hom_{H}(\pi,\mathbf{1})\simeq \prod_{g\in J\backslash G/H} \Hom_{J^{g}\cap H}(\Lambda^{g},\mathbf{1}).
    \]
    Thus there exists a $g\in G$ such that $\Hom_{J^{g}\cap H}(\Lambda^{g},\mathbf{1}) \neq 0$. Let $H':=H^{g^{-1}}$ such that $J^g \cap H = (J\cap H')^g$. Therefore

    \[
        0\neq \Hom_{J^{g}\cap H}(\Lambda^{g},\mathbf{1}) \simeq \Hom_{J\cap H'}(\Lambda,\mathbf{1}).
    \]
    Define $I=\langle J\cap H',\varpi_F\rangle$ and so $\langle \varpi_F \rangle\subset I\subset J.$  Applying Frobenius reciprocity, we have
    $$\Hom_{J\cap H'}(\Lambda,\mathbf{1})\cong\Hom_{I}(\Lambda,\Ind_{J\cap H'}^{I}\mathbf{1}).$$
    % $$\Hom_{J\cap H'}(\Lambda,\mathbf{1})\cong\Hom_{J}(\Lambda,\Ind_{J\cap H'}^{J}\mathbf{1})\cong\Hom_{I}(\Lambda,\Ind_{J\cap H'}^{I}\mathbf{1}).$$
    Let $\phi$ be a non-trivial element in $\Hom_{I}(\Lambda,\Ind_{J\cap H'}^{I}\mathbf{1})$. Let $V$ be the representation space of $\Lambda$. Taking the restriction to $\langle \varpi_F\rangle$, we have
    $$0\neq \phi\in \Hom_{\langle \varpi_F\rangle}(\Lambda,\Ind_{J\cap H'}^{I}\mathbf{1}).$$
    Since $\langle\varpi_F\rangle$ acts as the character $\bar{\chi}$ on $\Lambda$, $\phi(V)$ is contained in the $\bar{\chi}$-socle of $\mathrm{res}_{\langle\varpi_F\rangle}^{I}\Ind_{J\cap H'}^{I}\mathbf{1}$, which is the maximal subrepresentation where $\langle \varpi_F\rangle$ acts as a multiple of $\bar{\chi}$. Meanwhile,
    $$\Ind_{J\cap H'}^{I}\mathbf{1}\subset\Ind_{\{1\}}^{\langle \varpi_F\rangle}\mathbf{1}.$$
    By Frobenius reciprocity, $\dim\Hom_{\langle \varpi_F\rangle}(\bar{\chi},\Ind_{\{1\}}^{\langle \varpi_F\rangle}\mathbf{1})=1$. On the other hand, $J\cap H'$ acts trivially on $\Ind_{J\cap H'}^{I}\mathbf{1}$. Since $\Lambda$ is $J\cap H'$-distinguished, one can inflate $\bar{\chi}$ to a character of $I$, still denoted by $\bar{\chi}$. Then the $\bar{\chi}$-socle of $\Ind_{J\cap H'}^{I}\mathbf{1}$ is equivalent to $\bar{\chi}$. Hence
    $$\phi\in \Hom_{I}(\Lambda,\bar{\chi})\neq 0.$$
    By Frobenius reciprocity and the fact that  $J/ I$ is compact, we have
    \begin{equation}
        \label{nonvan:chi}
        0 \neq \Hom_{I}(\Lambda,\bar{\chi})\cong \Hom_{J}(\Lambda,\cInd_{I}^{J}\bar{\chi}),\end{equation}
    where $\cInd_{I}^{J}\bar{\chi}\in \Rep_{\Fl}^{\bar{\alpha}}(J)$ is admissible.

    Let $\alpha$ be a root of unity in $\Zl$ whose image (after modulo $\ell$) in $\Fl$  is equal to $\bar{\alpha}$, and denote by $\chi$ a $\Zl$-lift of $\bar{\chi}$ such that $\chi$ maps $\varpi_F$ to $\alpha$. Set $\tilde{\chi}=\chi\otimes_{\Zl}\Ql$. From Proposition \ref{propProjEnvType}, let $\overline{\mathcal{P}}$ be the projective envelope of $\Lambda$ in $\Rep_{\Fl}^{\bar{\alpha}}(J)$ and $\mathcal{P}$ a projective object in $\Rep_{\Zl}^{\alpha}(J)$ such that $\mathcal{P} \otimes \Fl = \overline{\mathcal{P}}$. Set $\Lambda|_{J^0}=\kappa\otimes\rho$. Since $\rho$ corresponds to a supercuspidal representation of $J^0/J^1$, a similar argument to \cite[Chapitre III, 2.9]{vigneras} gives us the structure of the projective envelope. Let us give some more details. From the structure of the projective envelope for finite groups, we see that if $\delta$ is an irreducible integral $\Ql$-represenation of $J$ then the multiplicity of $\delta$ in $\mathcal{P} \otimes \Ql$ is equal to the multiplicity of $\Lambda$ in $r_\ell(\delta)$ (see \cite[Prop. 4.9]{SecherreDrevon}). Moreover, if $\Lambda'$ is a subquotient of $\mathcal{P} \otimes \Fl$,  then \cite[Lem. 4.13]{SecherreDrevon} says that there exists $\delta$ an irreducible integral $\Ql$-represenation of $J$ such that its modulo $\ell$ reduction contains $\Lambda$ and $\Lambda'$. We show that $\Lambda\cong\Lambda'$ by proving that $\delta$ is an extended maximal simple type and that $\Lambda=r_\ell(\delta)$. In fact, there is an irreducible component $\delta_0$ of $\delta\vert_{J^0}$, such that $\kappa\otimes\rho$ is a subquotient of $r_{\ell}(\delta_0)$. Let $\tilde{\kappa}$ be a $\Ql$-lift of $\kappa$ as in \cite[Chapitre III,4.20]{vigneras}. Denote by $\tilde{\eta}$ the restriction $\tilde{\kappa}\vert_{J^1}$. Since the $\tilde{\eta}$-isotypic part of $\delta_0$ is non-trivial, by taking $\tilde{\kappa}$-restriction as defined in \cite[Definition 8.1 (c)]{Vig01b}, there is an irreducible $\Ql$-representation $\tilde{\rho}$ of $J^0\slash J^1$ such that $\delta_0\cong\tilde{\kappa}\otimes\tilde{\rho}$. Then $\rho$ is a subquotient of $r_{\ell}(\tilde{\rho})$. Since $\rho$ is supercuspidal, by \cite[Prop. 5.7]{Helm} we know $\tilde{\rho}$ is supercuspidal and $\rho=r_{\ell}(\tilde{\rho})$, which implies that $\delta_0$ is a maximal simple type and $r_{\ell}(\delta_0)=\kappa\otimes\rho$. Hence $\delta$ is an extended maximal simple type. We conclude that $r_{\ell}(\delta)=\Lambda$. Therefore, we have the following structure of $\mathcal{P}$:
    \begin{enumerate}
        \item $\overline{\mathcal{P}}=\mathcal{P} \otimes_{\Zl} \Fl$ is indecomposable of finite length, with each irreducible component isomorphic to $\Lambda$.
        \item Let $\widetilde{\mathcal{P}} := \mathcal{P} \otimes_{\Zl} \Ql$ be a $\Ql$-lift of $\mathcal{P}$. Then $\widetilde{\mathcal{P}}$ is isomorphic to the direct sum of extended maximal simple types, which are all the $\Ql$-lifts of $\Lambda$.
    \end{enumerate}

    By (1)  and \eqref{nonvan:chi}, we get that
    \begin{equation}
        \label{equaPQFl}
        \Hom_{\Fl[J]}(\overline{\mathcal{P}},\cInd_{I}^{J}\bar{\chi})\neq 0.
    \end{equation}

    %From the assumption that $\Lambda$ is $H$-distinguished, we have $\omega^{un}\hookrightarrow\Ind_{J\cap H'}^{J}\mathbf{1}_{\Fl}$, from which we deduce that
    %$$\omega_{\Ql}^{un}\hookrightarrow\Ind_{J\cap H'}^{J}\mathbf{1}_{\Ql}.$$
    By \S I,9.3,(viii) of \cite{vigneras}, $\cInd_{I}^{J}\chi$ is a $\Zl[J]$-lattice of $\cInd_{I}^{J}\tilde{\chi}$, and the reduction modulo-$\ell$ gives a surjective morphism from $\cInd_{I}^{J}\chi$ to $\cInd_{I}^{J}\bar{\chi}$. By the projectivity of $\mathcal{P}$, \eqref{equaPQFl} implies that
    $$\Hom_{J}(\mathcal{P},\cInd_{I}^{J}\chi)\neq 0.$$
    Hence by tensoring with $\Ql$ on both sides,
    \begin{equation}
        \label{equaPQl}
        0\neq \Hom_{J}(\mathcal{P}\otimes\Ql,\cInd_{I}^{J}\tilde{\chi})\subset \Hom_{J}(\mathcal{P}\otimes\Ql,\Ind_{J\cap H'}^{J}\mathbf{1}).
    \end{equation}

    %Thus
    % \[
    %\Hom_{\Zl[J]}(\mathcal{P} ,\Ind_{J\cap H'}^{J}\Zl)  \neq 0
    % \]
    %and
    %\[
    %\Hom_{\Ql[J]}(\widetilde{\mathcal{P}},\Ind_{J\cap H'}^{J}\Ql) \simeq \Hom_{\Zl[J]}(\mathcal{P} ,\Ind_{J\cap H'}^{J}\Zl) \otimes_{\Zl} \Ql \neq 0.
    %\]\todo{Problem here!}

    % Let $V=\Ind_{J\cap H'}^{J}\Zl$, thus we get $V \otimes_{\Zl} \Fl=\Ind_{J\cap H'}^{J}\Fl$. Since $\mathcal{P}$ is a projective $\Zl[J]$-module, any morphism in $\Hom_{\Fl[J]}(\mathcal{P} \otimes_{\Zl} \Fl,V \otimes_{\Zl} \Fl)$ can be lifted to $\Hom_{\Fl[J]}(\mathcal{P} ,V )$. Hence
    % \[
    %     \Hom_{\Zl[J]}(\mathcal{P} ,\Ind_{J\cap H'}^{J}\Zl)  \neq 0.
    % \]
    % By Frobenius reciprocity, we get that $\Hom_{\Zl[J\cap H']}(\mathcal{P} ,\Zl)  \neq 0$. Let $f \in \Hom_{\Zl[J\cap H']}(\mathcal{P} ,\Zl)$ be a non zero morphism. Denote by $M$ the image of $f$. Since $\mathcal{P}$ is a projective $\Zl[Q]$-module, it is a free $\Zl$-module. 

    %Since $\Ind_{J\cap H'}^{J}\Zl$ is also a free $\Zl$-module, we have that $\Hom_{\Zl[J]}(\mathcal{P},\Ind_{J\cap H'}^{J}\Zl)$ is a free $\Zl$-module \todo{Problem here!}. From (1) we have
    % \[
    % 0\neq \Hom_{\Fl[J]}(\mathcal{P} \otimes_{\Zl} \Fl,\Ind_{J\cap H'}^{J}\Fl) \simeq \Hom_{\Zl[J]}(\mathcal{P} ,\Ind_{J\cap H'}^{J}\Zl) \otimes_{\Zl} \Fl.
    % \]\todo{Problem here!}

    From (2) and  \eqref{equaPQl} there exists  a $\Ql$-lifts $\widetilde{\Lambda}$ of $\Lambda$ such that $\Hom_{\Ql[J]}(\widetilde{\Lambda},\Ind_{J\cap H'}^{J}\mathbf{1}) \neq 0$. Hence $\Hom_{J\cap H'}(\widetilde{\Lambda},\mathbf{1}) \neq 0$.

    Let $\widetilde{\pi}:=\cInd_{J}^{G}(\widetilde{\Lambda})$, where $\tilde{\Lambda}$ is a $\Ql$-extended maximal simple type (the discussion before about $\delta$ shows that $\tilde{\Lambda}$ is an extended maximal simple type). Hence $\widetilde{\pi}$ is irreducible and supercuspidal. Since the reduction modulo $\ell$ commutes with $\cInd$, we deduce that the representation $\widetilde{\pi}$ is a $\Ql$-lift of $\pi$. We are left to prove that it is $H$-distinguished. As before

    \[
        \Hom_{J^{g}\cap H}(\widetilde{\Lambda}^{g},\mathbf{1}) \simeq \Hom_{J\cap H'}(\widetilde{\Lambda},\mathbf{1}) \neq 0
    \]
    and Mackey formula and Frobenius reciprocity give us that $\widetilde{\pi}$ is $H$-distinguished.
\end{proof}
% \begin{rem}
%     The whole proof of Theorem \ref{thmLiftDistinguished} works for the case when the field $F$ is a a finite field extension of $\mathbb{F}_p((t))$.
% \end{rem}

The following corollary is an immediate consequence of Theorem \ref{thmLiftDistinguished}.
\begin{cor}\label{symplecticvanishing}
    When the characteristic of $\FF$ is different from 2, there is no supercuspidal $\Fl$-representation of $\GL_{2n}(F)$ distinguished by $\Sp_{2n}(F)$.
\end{cor}

\begin{proof}
    There is no complex supercuspidal representation of $\GL_{2n}(F)$ distinguished by $\Sp_{2n}(F)$. When the field $\FF$ has characteristic zero, this follows from \cite[Thm. 3.2.2]{HeumosRallis}; and when the field $F$ is of characteristic $p$ with $p\neq2$, it follows from \cite[Thm. 1]{prasad2019generic} (see \cite[Section 8]{prasad2019generic}). Therefore, the result follows from Theorem \ref{thmLiftDistinguished}.
\end{proof}
% \begin{rem}
%     When the field $F$ is of characteristic $p$ with $p\neq2$, Corollary \ref{symplecticvanishing} follows from Theorem \ref{thmLiftDistinguished} and \cite[Thm. 1]{prasad2019generic}. (See \cite[Section 8]{prasad2019generic}.)
% \end{rem}

\subsection{The Langlands correspondence modulo \texorpdfstring{$\ell$}{l} for \texorpdfstring{$\GL_n$}{GLn}}

In the rest of this paper we will use the modulo $\ell$ local Langlands correspondence to study distinguished representations. We recall in this section some important results about it.

\bigskip

The first step to define a modulo \texorpdfstring{$\ell$}{l} local Langlands correspondence for $\GL_n$ is to define the correspondence for supercuspidal representations. The supercuspidal representations correspond to irreducible representations of the Weil group. This bijection is defined in the modulo $\ell$ case using the modulo $\ell$ reduction. This leads to the semisimple Langlands correspondence using the supercuspidal support of a smooth representation.

\bigskip

Let $R$ be an algebraically closed field of characteristic different from $p$.
We denote by $\Irr_{R}(\GL_n(\FF))$ the set of isomorphism classes of smooth irreducible representations of $\GL_n(\FF)$ over $R$ and by $\Scusp_{R}(\GL_n(\FF))$ the subset of supercuspidal $R$-representations.
Let $\Irr_{R}(\GL_\FF):=\cup_{n\geq 1}\Irr_{R}(\GL_n(\FF))$ and $\Scusp_{R}(\GL_\FF) := \cup_{n\geq 1}\Scusp_{R}(\GL_n(\FF))$.

Let $W_\FF$ be the Weil group of $\FF$. It is the subgroup of $\Gal(\overline{\FF}/\FF)$, where $\overline{\FF}$ is a separable algebraic closure of $\FF$, generated by the inertia subgroup $I_F$ and a geometric Frobenius element $\mathrm{Frob}$. Its topology is such that $I_F$ has the topology induced from $\Gal(\overline{\FF}/\FF)$, $I_F$ is open and the multiplication by $\mathrm{Frob}$ is a homeomorphism. Let $\Irr_{R}(W_{\FF})(n)$ be the set of isomorphism classes of smooth irreducible $R$-representations of $W_\FF$ of dimension $n$.
Finally, we denote by $\Mod_{R}(W_{\FF})$ the set of isomorphism classes of semisimple smooth $R$-representations of $W_\FF$ of finite dimension.

\bigskip

In \cite{VignerasLanglands} Vignéras defines a bijection using the modulo $\ell$ reduction of the local Langlands correspondence over $\Ql$

\begin{thm}[{\cite[Cor. 2.5]{VignerasLanglands}}]
    There exist unique Langlands bijections on $\Fl$
    \[
        \Scusp_{\Fl}(\GL_n(\FF)) \leftrightarrow \Irr_{\Fl}(W_{\FF})(n)
    \]
    which are compatible with modulo $\ell$ reduction.
\end{thm}

Let $\MScusp_{\Fl}(\GL_\FF)$ be the set of formal finite sums $\pi_1+\cdots+\pi_r$ of elements of $\Scusp_{\Fl}(\GL_\FF)$. The previous bijections induce a bijection
\[
    \MScusp_{\Fl}(\GL_\FF) \to \Mod_{\Fl}(W_{\FF}).
\]

\bigskip

Let $\pi \in \Irr_{\Fl}(\GL_{n}(\FF))$. Then $\pi$ appears as a subquotient of the parabolic induction of a representation $\pi_1 \boxtimes \cdots \boxtimes \pi_r$, where $\pi_i \in \Scusp_{\Fl}(\GL_{n_i}(\FF))$ and $\sum n_i =n$. This defines a surjective map with finite fibers, called the \textit{supercuspidal support},
\[
    sc : \Irr_{\Fl}(\GL_\FF) \to \MScusp_{\Fl}(\GL_\FF)
\]
by $sc(\pi)=\pi_1 + \cdots + \pi_r$.

\bigskip

Combining the supercuspidal support and the previous bijections (see \cite[Thm. 1.6]{VignerasLanglands}) we get a map, that we call the semisimple Langlands correspondence modulo $\ell$,
\[
    V_{ss} : \Irr_{\Fl}(\GL_{\FF}) \to \Mod_{\Fl}(W_{\FF}).
\]

\bigskip

The previous map $V_{ss}$ is a surjection but not a bijection. To get a bijection one needs to introduce Weil-Deligne representations.

\bigskip

Let us start by recalling the definition of Weil-Deligne representations. Let $\nu : W_{\FF} \to R^{\times}$ be the unique character trivial on the inertia subgroup of $W_{\FF}$, sending a geometric Frobenius element to $q_{\FF}^{-1}$.

\begin{defi}
    We call a semisimple Weil-Deligne representation of $\GL_n$  over $R$ a couple $(\varphi,N)$ with
    \begin{enumerate}{}
        \item $\varphi : W_{\FF} \to \GL_n(R)$ a smooth representation with image composed of semisimple elements.
        \item $N \in M_n(R)$ is a matrix such for all $w\in W_{\FF}$, $\varphi(w)N = \nu(w)N\varphi(w)$.
    \end{enumerate}
\end{defi}

A morphism of Weil-Deligne representations from $(\varphi,N)$ to $(\varphi',N')$ is a morphism of $W_{\FF}$-representation $f\in \Hom_{W_\FF}(\varphi,\varphi')$ such that $f\circ N = N' \circ f$. It is an isomorphism if $f$ is an isomorphism of $W_{\FF}$-representations. We denote by $\WDRep_{R}(W_{\FF},\GL_n)$ the set of isomorphism classes of semisimple Weil-Deligne representation of $\GL_n$. When $R=\Ql$, every $N$ as above is nilpotent. However, this is not the case when $R=\Fl$. Therefore, we denote by $\Nilp_{R}(W_{\FF},\GL_n)$ the subset composed of the elements $(\phi,N)$ with $N$ nilpotent.

\begin{thm}[{\cite[Thm. 1.8.2]{VignerasLanglands}}]
    \label{thmBijNilp}
    Let $\varphi \in \Mod_{\Fl}(W_{\FF})$. The set consisting of $\pi \in \Irr_{\Fl}(\GL_{\FF})$ such that $V_{ss}(\pi)=\varphi$ is in bijection with the nilpotent endomorphisms $N$ of $\pi$, up to conjugation by an isomorphism of $\varphi$, such that $\varphi(w)N = \nu(w)N\varphi(w)$, for all $w\in W_{\FF}$.
\end{thm}

Hence, we get a bijection $\Irr_R(\GL_{\FF}) \to \Nilp_{R}(W_{\FF},\GL)$. Vignéras defines the ``Langlands'' bijection in \cite[Section 1.8]{VignerasLanglands} using modulo $\ell$ reduction (in a non-naive way), and we call this bijection the $V$ correspondence :

\[
    V : \Irr_R(\GL_{\FF}) \to \Nilp_{R}(W_{\FF},\GL).
\]

\subsection{Supercuspidal representations of $\GL_n(\EE)$ distinguished by a Galois involution}

\label{secSupercuspDist}

Let $\EE/\FF$ be a quadratic extension of locally compact non-archimedean local fields of characteristic different from 2 and residual characteristic $p$. In this section we are interested in $\GL_n(F)$-distinguished supercuspidal $\ell$-modular representations of $\GL_n(\EE)$. A $\GL_n(F)$-distinguished supercuspidal representation is $\sigmaEF$-selfdual, that is $\cont{\pi} \simeq \sdual{\pi}$ (see \cite[Thm. 4.1]{vincent2019supercusp}). Let us assume that $\ell \neq 2$ and let $\pi$ be a $\sigmaEF$-selfdual supercuspidal representation. The main goal of this section is to prove the \emph{Dichotomy Theorem} that $\pi$ is either distinguished or $\omegaEF$-distinguished, and the \emph{Disjunction Theorem} that $\pi$ cannot be both distinguished and $\omegaEF$-distinguished. We call it the Jacquet conjecture in the modular setting. For $p\neq 2$ this is proved in \cite[Thm. 10.8]{vincent2019supercusp}. Using Theorem \ref{thmLiftDistinguished}, when $\ell \neq 2$, we will give a different proof of this result that also works for $p=2$ when $\FF$ has characteristic zero.

\begin{rem}
    In this section we assume $\ell \neq 2$. When $\ell=2$ then $p\neq 2$ and the distinguished supercuspidal representations are studied in \cite{vincent2019supercusp}. Therefore, combining the results of this section and of \cite{vincent2019supercusp} gives a complete result for supercuspidal distinguished representations for all $\ell$ and $p$ with $\ell \neq p$ in characteristic zero.
\end{rem}

\bigskip

With Theorem \ref{thmLiftDistinguished}, we know that we can find a distinguished $\Ql$-lift of a distinguished supercuspidal $\Fl$-representation whose Langlands parameter is conjugate-orthogonal in the sense of \cite[\S7]{GGP}. With the following lemma, we will show that all the $\sigmaEF$-selfdual $\Ql$-lift of a $\sigmaEF$-selfdual supercuspidal representations are either all distinguished or all $\omegaEF$-distinguished.

\begin{lem}\label{lemconj}
    Let $\pi$ be a $\sigma$-selfdual irreducible supercuspidal $\Fl$-representation of $\GL_n(E)$. Suppose that $\tilde{\pi}_1$ and $\tilde{\pi}_2$ are two $\sigma$-selfdual irreducible supercuspidal $\Ql$-representations of $\GL_n(E)$ such that both $\tilde{\pi}_i$ have modulo $\ell$ reduction isomorphic to $\pi$. Then $\tilde{\pi}_1$ and $\tilde{\pi}_2$ share the same sign, i.e., $\tilde{\pi}_1$ and $\tilde{\pi}_2$ are either conjugate-symplectic or conjugate-orthogonal at the same time, but not both.
\end{lem}
\begin{proof}

    Let us denote by $\varphi$ (resp. $\tilde{\varphi}_i$) the Langlands parameter of $\pi$ (resp. $\tilde{\pi}_i$). For $i\in \{1,2\}$, since $\tilde{\pi}_2$ is $\sigma$-selfdual, $\tilde{\varphi}_i$ is either conjugate-symplectic or conjugate-orthogonal, and let us denote by $B_i$ the associated bilinear form. We also denote by $b_i$ the sign of $B_i$. Let us denote by $V_i$ the space of $\tilde{\varphi}_i$ and by $V$ the space of $\varphi$. The non-degenerate bilinear form $B_i:V_i \times V_i \to \Ql$ induces an isomorphism of $W_E$ representation $f_i : V_i^s \to V_i^{\vee}$, where $s \in W_F \setminus W_E$.

    Since $\tilde{\varphi}_i$ is a $\Ql$-lift of $\varphi$, there exists an $\Zl$-lattice $L_i$ of $V_i$ such that $L_i \otimes \Ql \simeq V_i$ and $L_i \otimes \Fl \simeq V$ (as $W_E$ representations). Let $M_i$ be the $\Zl$-lattice of $V_i$ such that $f_i$ induces an isomorphism $L_i^s \to M_i^{\vee}$. The modulo $\ell$ reduction does not depend on the lattice, so we also have $M_i \otimes \Ql \simeq V_i$ and $M_i \otimes \Fl \simeq V$. Let $B_{i,\Zl}$ be the restriction of $B_i$ to $L_i\times M_i$. Its image is a $\Zl$-lattice of $\Ql$. And by definition of $L_i$ and $M_i$ it is non-degenerate as a $\Zl$-bilinear form. Therefore, its modulo $\ell$ reduction gives a non-degenerate bilinear form $\bar{B}_i$ on $V$. As $\ell \neq 2$ the sign $b_i$ is preserved by reduction modulo $\ell$, and is the sign of $\bar{B}_i$. Moreover, $\pi$ is supercuspidal, so $\varphi$ is irreducible. Schur's lemma implies that the sign of a bilinear form on $V$ is unique. Therefore $b_1=b_2$ and this concludes the proof.
\end{proof}

Theorem \ref{thmLiftDistinguished} and Lemma \ref{lemconj} imply the following proposition.

\begin{prop}
    \label{proLiftDistComplete}
    Let $\pi$ be a supercuspidal $\Fl$-representation of $\GL_n(\EE)$. Then the following assertions are equivalent:
    \begin{enumerate}
        \item $\pi$ is distinguished.
        \item There exists a $\Ql$-lift $\tilde{\pi}$ of $\pi$ which is distinguished.
    \end{enumerate}

    Moreover, when these conditions are satisfied, then all the $\sigmaEF$-selfdual $\Ql$-lifts $\tilde{\pi}$ of $\pi$ are distinguished.
\end{prop}

\begin{proof}

    The implication $(1) \Rightarrow (2)$ is given by Theorem \ref{thmLiftDistinguished} applied with $\GL_n(\EE)$ and $H=\GL_n(\FF)$. The other direction $(2) \Rightarrow (1)$ follows from an argument of reduction of invariant linear forms as in \cite[Thm. 3.4]{KuMa}.

    To conclude, when $\FF$ has characteristic zero, if $\pi$ is distinguished, then all the $\sigmaEF$-selfdual $\Ql$-lifts of $\pi$ are distinguished by Lemma \ref{lemconj}. For $\FF$ of positive characteristic, this follows from \cite[Thm. 10.11]{vincent2019supercusp}.
\end{proof}

We can now prove the disjunction theorem.

\begin{prop}
    \label{proDisjonction}
    Let $\pi$ be a supercuspidal $\Fl$-representation of $\GL_n(\EE)$ distinguished by $\GL_n(\FF)$. Then
    \[
        \dim\Hom_{\GL_n(\FF)}(\pi,\omegaEF)=0.
    \]
\end{prop}

\begin{proof}
    We prove the result by contradiction. Let us assume that $\pi$ is $\omegaEF$-distinguished. Let $\chi$ be a character of $\EE^{\times}$ extending $\omegaEF$. Let $\pi':=\pi(\chi^{-1} \circ \det)$. Then $\pi'$ is a supercuspidal representation and is distinguished since $\pi$ is $\omegaEF$-distinguished. Applying Theorem \ref{thmLiftDistinguished} to $\pi'$ we find a lift $\tilde{\pi}$ of $\pi$ which is $\tomegaEF$-distinguished, where $\tomegaEF$ is the canonical $\ell$-adic lift of $\omegaEF$, that is the $\Ql$-character of $\FF^{\times}$ of kernel $\Nm_{\EE/\FF}(\FF^{\times})$.

    With Theorem \ref{thmLiftDistinguished} we also have a distinguished supercuspidal $\Ql$-lift of $\pi$.
    Thanks to  Proposition \ref{proLiftDistComplete}, all $\sigma$-selfdual $\Ql$-lifts are distinguished by $\GL_{n}(\FF)$. In particular, $\tilde{\pi}$ is $\GL_{n}(\FF)$-distinguished, which
    contradicts the fact that a $\sigmaEF$-selfdual supercuspidal $\Ql$-representation cannot be both distinguished and $\omegaEF$-distinguished.
\end{proof}

We are left to prove the dichotomy theorem. To do that, we need a theorem analogous to Theorem \ref{thmLiftDistinguished} but for $\sigmaEF$-selfdual representations.

\begin{prop}
    \label{proLiftSdual}
    Let $\pi$ be a $\sigmaEF$-selfdual supercuspidal representation of $\GL_n(\EE)$ over $\Fl$. Then there exists a $\Ql$-lift $\tilde{\pi}$ of $\pi$ which is $\sigmaEF$-selfdual.
\end{prop}

\begin{proof}

    Let us denote by $\delta$ the action of $\sigmaEF$-duality on $R$-representations, that is $V^\delta = (V^\sigmaEF)^{\vee}$. A representation is $\sigmaEF$-selfdual if and only if it is $\delta$ invariant.

    Let $[\pi]$ be the inertial class of $\pi$. Let $\mathcal{I}$ be the set of inertial classes of supercuspidal $\Ql$-lifts of $\pi$. This is a finite set of cardinal a power of $\ell$ (see \cite[Prop. 1.3]{MinguezSecherreJacquetLanglands}). In particular, since $\ell \neq 2$, this cardinal is odd. The action of $\delta$ induces an action on inertial classes of representations and since $\pi$ is $\delta$-invariant, $\delta$ induces an action on $\mathcal{I}$. But the cardinal of $\mathcal{I}$ is odd and $\delta$ is an involution, therefore there exists an inertial class which is $\delta$-invariant, that is there exists $\tilde{\pi}$ a supercuspidal $\Ql$-lift of $\pi$ such that $[\tilde{\pi}^\delta]=[\tilde{\pi}]$. Let $\tilde{\chi}$ be an unramified character such that $\tilde{\pi}^\sigmaEF=\tilde{\pi}^{\vee} \otimes \tilde{\chi}$.

    Since $\tilde{\chi}$ is unramified, it is trivial on the kernel of the norm map, that is $\tilde{\chi} = \tilde{\chi}^{\sigmaEF}$. Therefore, there exists an unramified character $\tilde{\eta}$ such that $\tilde{\chi} = \tilde{\eta} \tilde{\eta}^{\sigmaEF} = \tilde{\eta}^2$. Since $\tilde{\pi}$ is $\ell$-integral, $\tilde{\chi}$ is also $\ell$-integral and so is $\tilde{\eta}$. Let $\tilde{\pi}':=\tilde{\pi} \otimes \tilde{\eta}^{-1}$ such that $\tilde{\pi}'$ is $\sigmaEF$-selfdual. Let us denote by $\pi'=\pi \otimes \eta^{-1}$ its reduction modulo $\ell$ where $\eta$ is the reduction modulo $\ell$ of $\tilde{\eta}$. Since $\tilde{\pi}'$ is $\sigmaEF$-selfdual so is $\pi'$. The representations $\pi$ and $\pi'=\pi \otimes \eta^{-1}$ are both $\sigmaEF$-selfdual. Thus $\pi = \pi \otimes \eta^{-2}$.

    Let $X_u(\pi)$ be the set of unramified characters $\xi$ such that $\pi = \pi \otimes \xi$. It is contained in the cyclic group of order $n$ composed of the character $\xi$ such that $\xi^n=1$ and is characterized by its order $f(\pi)$. Let $o(\eta)$ be the order of $\eta$. We have $o(\eta) \mid 2f(\pi)$. Let $\tilde{\rho}$ be the canonical $\Ql$-lift of $\eta$. In particular $o(\tilde{\rho})=o(\eta) \mid 2f(\pi)$. Let $\tilde{\pi}'':=\tilde{\pi}' \otimes \tilde{\rho}$. It is a supercuspidal $\Ql$-lift of $\pi$. To finish the proof, we are left to prove that it is $\sigmaEF$-selfdual. The representation $\tilde{\pi}'$ being $\sigmaEF$-selft dual, we just need to prove that $\tilde{\pi}'' = \tilde{\pi}'' \otimes \tilde{\rho}^2$, that is that $o(\tilde{\rho}) \mid 2f(\tilde{\pi}')$. The explicit formulas for $f(\tilde{\pi}')$ and $f(\pi)$ show that $f(\tilde{\pi}')=f(\pi)\ell^k$ for some integer $k$ (see for instance \cite[Rem. 3.21]{MinguezSecherre}). Thus $2f(\pi) \mid 2f(\tilde{\pi}')$ and so $o(\tilde{\rho}) \mid 2f(\tilde{\pi}')$. Therefore $\tilde{\pi}''$ is $\sigmaEF$-selfdual.
\end{proof}

\begin{thm}
    \label{thmDichotomy}
    Let $\pi$ be a supercuspidal $\Fl$-representation of $\GL_n(\EE)$. Then $\pi$ is $\sigmaEF$-selfdual if and only if it is distinguished or $\omegaEF$-distinguished by $\GL_n(F)$. Moreover, $\pi$ cannot be both distinguished and $\omegaEF$-distinguished by $\GL_n(F)$.
\end{thm}

\begin{proof}
    Assume that $\pi$ is $\sigmaEF$-selfdual. Then there exists a $\sigmaEF$-selfdual lift $\tilde{\pi}$ by Proposition \ref{proLiftSdual}. From the dichotomy theorem in the complex case (see \cite{kable} in characteristic zero and \cite[Thm. A.2]{AKMSS} in positive characteristic, or \cite{Jo}), $\tilde{\pi}$ is either distinguished or $\tomegaEF$-distinguished. Since the reduction modulo $\ell$ preserves distinction (see \cite[Thm. 3.4]{KuMa}), it is also true for $\pi$. The disjunction is Proposition \ref{proDisjonction} and the reciprocity follows from Theorem \ref{thmSecherreDist}.
\end{proof}

\begin{prop}
    \label{proDistConj}
    Let $\pi$ be a $\sigmaEF$-selfdual supercuspidal $\Fl$-representation of $\GL_n(\EE)$ and $\varphi_\pi$ its Langlands parameter. Then $\pi$ is distinguished if and only if $\varphi_\pi$ is conjugate-orthogonal and $\pi$ is $\omegaEF$-distinguished if and only if $\varphi_\pi$ is conjugate-symplectic.
\end{prop}

\begin{proof}
    A $\sigmaEF$-selfdual supercuspidal $\Fl$-representation is either distinguished or $\omegaEF$-distinguished by Theorem \ref{thmDichotomy}. By Proposition \ref{proLiftDistComplete} we can lift a distinguished representation to a distinguished representation. Hence the result follows from the complex case in \cite{kable} (the results of \cite{kable} are still valid in positive characteristic by \cite[Appendix A]{AKMSS} or \cite{Jo}).
\end{proof}

\subsection{Supercuspidal representations of $\GL_n(\EE)$ distinguished by unitary groups}

Another example of application of Theorem \ref{thmLiftDistinguished} is for supercuspidal representations distinguished by a unitary involution. Let $G=\GL_n(\EE)$ and assume $\ell \neq 2$. Let $\varepsilon$ be an hermitian matrix in $M_n(\EE)$ ($\sigmaEF({}^{t}\varepsilon)=\varepsilon$). We denote by $\varsigma$ the \emph{unitary involution} on $G$ defined by $\varsigma(g)=\varepsilon \sigma({}^{t}g^{-1}) \varepsilon^{-1}$. In this section, we are interested in supercuspidal $\Fl$-representations of $G$ distinguished by $G^{\varsigma}$. In particular, we are going to prove that a supercuspidal representation $\pi$ of $G$ is distinguished by $G^\varsigma$ if and only if it is $\sigmaEF$-invariant. When $p\neq 2$, this is proved by Zou in \cite{Zou}.

\begin{prop}
    \label{proLiftSigmaInv}
    Let $\pi$ be a $\sigma$-invariant supercuspidal representation of $\GL_n(\EE)$ over $\Fl$. Then there exists a $\Ql$-lift $\tilde{\pi}$ of $\pi$ which is $\sigma$-invariant.
\end{prop}

\begin{proof}

    We start with the same argument as in the proof of Proposition \ref{proLiftSdual}. The number of inertial classes of supercuspidal $\Ql$-lifts of $\pi$ is a power of $\ell$, hence is odd. The involution $\sigma$ permutes its elements, so there exists an inertial class $[\tilde{\pi}]$ which is $\sigma$-invariant. Let $\tilde{\chi}$ be an unramified character such that $\tilde{\pi}^\sigma=\tilde{\pi} \otimes \tilde{\chi}$.

    Let us prove that $\tilde{\pi}$ is in fact $\sigma$-invariant. As in the proof of  Proposition \ref{proLiftSdual}, let $X_u(\tilde{\pi})$ (resp. $X_u(\pi)$) be the set of unramified $\Ql$-characters $\tilde{\xi}$ (resp. unramified $\Fl$-characters $\xi$) such that $\tilde{\pi} = \tilde{\pi} \otimes \tilde{\xi}$ (resp. $\pi = \pi \otimes \xi$). It is characterized by its order $f(\tilde{\pi})$ (resp. $f(\pi)$). Let $o(\tilde{\chi})$ be the order of $\tilde{\chi}$. To finish the proof, we need to prove that $o(\tilde{\chi})$ divides $f(\tilde{\pi})$. Indeed, this implies that $\tilde{\chi} \in X_u(\tilde{\pi})$ and $\tilde{\pi}^\sigma \simeq \tilde{\pi} \otimes \tilde{\chi} \simeq \tilde{\pi}$.

    Let $\chi$ be the modulo $\ell$ reduction of $\tilde{\chi}$ of order $o(\chi)$. By taking the modulo $\ell$ reduction we have that $\pi \simeq \pi \otimes \chi$, thus $\chi \in X_u(\pi)$ and $o(\chi) \mid f(\pi)$. Also there exists an integer $a\in \mathbb{N}$ such that $o(\tilde{\chi}) = o(\chi) \ell^{a}$. The explicit formulas for $f(\tilde{\pi})$ and $f(\pi)$ show that there exists an integer $b \in \mathbb{N}$ such that $f(\tilde{\pi}) = f(\pi) \ell^{b}$ and $f(\pi)$ is prime to $\ell$.

    Since $\tilde{\chi}$ is unramified, it is trivial on the kernel of the norm map and $\tilde{\chi}=\tilde{\chi}^{\sigma}$. In particular, $\tilde{\pi}\simeq (\tilde{\pi} \otimes \tilde{\chi})^{\sigma} \simeq \tilde{\pi} \otimes \tilde{\chi}^2$ and $o(\tilde{\chi}) \mid 2f(\tilde{\pi})$. Since $\ell$ is odd, we get that $a \leq b$. At the end, $o(\chi) \mid f(\pi)$ and $\ell^a \mid \ell^b$ thus $o(\tilde{\chi}) \mid f(\tilde{\pi})$ and this finishes the proof.

\end{proof}

\begin{thm}
    \label{thmDistUnitary}
    Let $\pi$ be a supercuspidal $\Fl$-representation of $G=\GL_n(\EE)$ and $\varsigma$ a unitary involution of $G$. Then $\pi$ is distinguished by $G^{\varsigma}$ if and only if $\pi^{\sigma}\simeq \pi$.
\end{thm}

\begin{proof}
    When $\EE$ has characteristic zero, since the modulo $\ell$ reduction preserves $\sigma$-invariance and the property of being distinguished for supercuspidal representations (\cite[Thm. 3.4]{KuMa}), the result follows from the complex case (see \cite[Section 4, Corollary]{HakimMurnaghan} and \cite[Thm. 0.2]{lapid2012unitary}) by Theorem \ref{thmLiftDistinguished} and Proposition \ref{proLiftSigmaInv}. For a field $\EE$ of positive characteristic this follows from \cite[Thm. 1.1]{Zou}.
\end{proof}

\begin{rem}
    The implication, $\pi$ distinguished by $G^{\varsigma}$ implies $\pi$-invariant, is also proved in \cite[Thm. 4.1]{Zou} similarly.
\end{rem}

\section{The \texorpdfstring{$\GL_2(F)$}{GL2(F)}-distinguished representations}

\label{secGL2}

This section is devoted to the study of $\GL_2(\FF)$-distinguished $\ell$-modular representations of $\GL_2(\EE)$, with $\ell \neq 2$. The case of supercuspidal representations has been dealt with in Section \ref{secSupercuspDist}. Using Mackey theory, we will describe non-supercuspidal distinguished representations, hence giving a complete classification of all $\GL_2(\FF)$-distinguished representations.

\bigskip

Before we start, let us recall a result of Sécherre about $\ell$-modular $\GL_2(\FF)$-distinguished representations of $\GL_2(\EE)$. Denote by $\mathbf{1}_2$ the trivial character of $\GL_2(\FF)$.

\begin{thm}[{\cite[Thm. 4.1]{vincent2019supercusp}}]
    \label{thmSecherreDist}

    Let $\pi$ be a distinguished irreducible $\ell$-modular representation of $\GL_2(\EE)$. Then:
    \begin{enumerate}
        \item The central character of $\pi$ is trivial on $\FF^{\times}$.
        \item The contragredient representation $\cont{\pi}$ is distinguished.
        \item The space $\Hom_{\GL_2(\FF)}(\pi,\mathbf{1}_2)$ has dimension 1.
        \item The representation $\pi$ is $\sigmaEF$-selfdual, that is $\cont{\pi} \simeq \sdual{\pi}$.
    \end{enumerate}
\end{thm}
For the complex supercuspidal representations, it was proved by Hakim \cite{hakim1994} and Prasad \cite{prasad1992}.

\subsection{Non-supercuspidal representations}

\label{secNonSupercuspDist}
In Section \ref{secSupercuspDist}, we studied when supercuspidal representations are distinguished. In this section, we will deal with non-supercuspidal representations including the principal series representations.

\bigskip

Let us start with an easy lemma that {is needed} for the rest of this section.
\begin{lem}
    \label{lemChiSigma}
    Let $\chi$ be a character of $\EE^{\times}$. Then $\chi \chi^{\sigma}=\mathbf{1}$ if and only if $\chi_{|{F^\times}}=\mathbf{1}$ or $\chi_{|{F^\times}}=\omegaEF$.
\end{lem}

\begin{proof}
    The condition $\chi \chi^{\sigma}=\mathbf{1}$ is equivalent to $\chi_{|F^\times}$ being trivial on the norm group $\Nm_{\EE/\FF}(\EE^{\times})$. By the local class field theory the only two characters of $F^\times$ trivial on $\Nm_{\EE/\FF}(\EE^{\times})$ are $\mathbf{1}$ and $\omegaEF$.
\end{proof}

Let $\chi_1,\chi_2$ be two characters of $\EE^{\times}$. Denote by $\pi(\chi_1,\chi_2)$ the principal series representation of $\GL_2(E)$ induced from $(\chi_1,\chi_2)$, that is, $\pi(\chi_1,\chi_2)$ is the normalized parabolic induction of $\chi_1\boxtimes\chi_2$ from the standard Borel subgroup to $\GL_2(E)$.

\begin{lem}
    \label{lemPrincSer}
    Let $\pi=\pi(\chi_1,\chi_2)$ be a principal series representation of $\GL_2(E)$. Then
    $\pi$ is distinguished by $\GL_2(F)$ if and only if  either $\chi_1\chi_2^\sigma=\mathbf{1}$ or $\chi_1|_{F^\times}=\chi_2|_{F^\times}=\mathbf{1}$ with $\chi_1\neq\chi_2$.

    Moreover, when $\pi$ is distinguished by $\GL_2(F)$ we have
    \[\dim\Hom_{\GL_2(F)}(\pi,\mathbf{1}_2)=\begin{cases}
            2, & \mbox{ if } \ell \mid q_F-1, \chi_1=\chi_2 \mbox{ and } \chi_1|_{F^\times}=\mathbf{1} ; \\
            1, & \mbox{ otherwise}.
        \end{cases} \]
\end{lem}

\begin{proof} Recall that $B_2(E)$ is the standard Borel subgroup of $\GL_2(E)$.
    Note that $\GL_2(E)=B_2(E)\GL_2(F)\sqcup B_2(E)\eta\GL_2(F)$ with $\eta=\begin{pmatrix}
            1 & \delta \\1&-\delta
        \end{pmatrix}$ and $E=F[\delta]$. Denote by $S(X)$ the Schwartz space of $X$
    consisting of locally constant functions defined on $X$.
    There is a short exact sequence
    \[0\to S(B_2(E)\eta\GL_2(F))\to S(\GL_2(E))\to S(B_2(E)\GL_2(F))\to0 \]
    of $\GL_2(F)$-modules and $\pi$ can be embedded inside $S(\GL_2(\EE))$.
    Due to the geometric lemma, one has a short exact sequence
    \[0\to \cInd_{E^\times}^{\GL_2(\FF)}\chi_1\chi_2^\sigma\to\pi\to \Ind_{B_2(F)}^{\GL_2(\FF)}\delta_{B_2(E)}^{1/2}\delta_{B_2(F)}^{-1/2}(\chi_1\boxtimes\chi_2)\to0 \]
    with $\delta_{B_2(E)}^{1/2}=\delta_{B_2(F)}$.
    Applying the functor $\Hom_{\GL_2(F)}(-,\mathbf{1}_2)$,
    \[
        0\to \Hom_{F^\times\times F^\times}(\chi_1\boxtimes\chi_2,\mathbf{1})\to \Hom_{\GL_2(F)}(\pi,\mathbf{1}_2)\to \Hom_{E^\times}(\chi_1\chi_2^\sigma,\mathbf{1})\to \Ext^1_{F^\times\times F^\times}(\chi_1\boxtimes\chi_2,\mathbf{1})
    \]
    where $\Ext^1_{F^\times\times F^\times}$ denotes the Ext functor in the category of the finite dimensional representations of $F^\times\times F^\times$. Thanks to \cite[Prop. 8.4]{SecherreDrevon} that $\Ext_{F^\times}^1(\chi,\mathbf{1})\neq0$ if and only if $\chi$ is trivial, one has that $\Hom_{\GL_2(F)}(\pi,\mathbf{1})\neq0$ if and only if $\chi_1\chi_2^\sigma=\mathbf{1}$ or $\chi_1|_{F^\times}=\chi_2|_{F^\times}=\mathbf{1}$. If $\chi_1\chi_2^\sigma=\mathbf{1}$ and $\chi_1|_{F^\times}=\chi_2|_{F^\times}=\mathbf{1}$, then $\chi_1=\chi_2$. Thus $\pi$ is distinguished by $\GL_2(F)$ if and only if  either $\chi_1\chi_2^\sigma=\mathbf{1}$ or $\chi_1|_{F^\times}=\chi_2|_{F^\times}=\mathbf{1}$ with $\chi_1\neq\chi_2$.

    \bigskip

    Now, let us assume that $\pi$ is distinguished by $\GL_2(F)$, that is $\chi_1\chi_2^\sigma=\mathbf{1}$ or $\chi_1|_{F^\times}=\chi_2|_{F^\times}=\mathbf{1}$.

    If $\chi_1\neq\chi_2$, only one of the conditions $\chi_1\chi_2^\sigma=\mathbf{1}$ or $\chi_1|_{F^\times}=\chi_2|_{F^\times}=\mathbf{1}$ can be true. We then see from the previous exact sequence that $\Hom_{\GL_2(F)}(\pi,\mathbf{1}_2) = \Hom_{F^\times\times F^\times}(\chi_1\boxtimes\chi_2,\mathbf{1})$ or $\Hom_{\GL_2(F)}(\pi,\mathbf{1}_2) = \Hom_{E^\times}(\chi_1\chi_2^\sigma,\mathbf{1})$. Therefore
    \[
        \dim\Hom_{\GL_2(F)}(\pi,\mathbf{1}_2)=1.
    \]

    We are left to study the case $\chi_1=\chi_2$ with $\chi_1|_{F^\times}=\mathbf{1}$. Let $\chi:=\chi_1=\chi_2$. Note that if $q_\EE \not\equiv 1 \pmod{\ell}$ then $\pi(\chi,\chi)$ is irreducible. Thus by Theorem \ref{thmSecherreDist}, $\dim\Hom_{\GL_2(F)}(\pi(\chi,\chi),\mathbf{1}_2)=1$. So we assume that $q_\EE \equiv 1 \pmod{\ell}$. We have two cases:
    \begin{enumerate}
        \item Suppose $q_\FF \equiv -1 \pmod{\ell}$. In this case $(\nu^{-1/2})_{|\FF^{\times}}\neq \mathbf{1}$. The representation $\pi(\chi,\chi)$ is a direct sum of two irreducible representations: the character $(\chi \nu^{-1/2}) \circ \det$ and a twisted Steinberg $\St_{\chi \nu^{-1/2}}$ (see for instance \cite[Thm. 3]{vignerasGL2}). As $\dim\Hom_{\GL_2(F)}((\chi \nu^{-1/2}) \circ \det,\mathbf{1}_2)=0$ and
              \[
                  \dim\Hom_{\GL_2(F)}(\St_{\chi \nu^{-1/2}},\mathbf{1}_2) \leq 1
              \]
              by Theorem \ref{thmSecherreDist}, we get that $\dim\Hom_{\GL_2(F)}(\pi(\chi,\chi),\mathbf{1}_2)=1$.
        \item Suppose $q_\FF \equiv 1 \pmod{\ell}$. This time ${(\nu^{-1/2})}_{|\FF^{\times}}= \mathbf{1}$. Therefore,
              \[\dim\Hom_{\GL_2(F)}(\pi(\chi,\chi),\mathbf{1}_2)=\dim\Hom_{\GL_2(F)}(V(\mathbf{1},\mathbf{1}),\mathbf{1}_2)\]
              where $V(\mathbf{1},\mathbf{1})$ denotes the non-normalized induction.
    \end{enumerate}

    Recall that
    \[\GL_2(E)=B_2(E)\GL_2(F)\sqcup B_2(E)\eta\GL_2(F) \]
    where $\eta$ represents the open orbit in double coset $B_2(E)\backslash\GL_2(E)/\GL_2(F) $.
    There is a short exact sequence
    \[0\to \cInd_{E^\times}^{\GL_2(\FF)}\mathbf{1}\to V(\mathbf{1},\mathbf{1})\to \Ind_{B_2(F)}^{\GL_2(\FF)}\delta_{B_2(E)}^{1/2}\delta_{B_2(F)}^{-1/2}\to0. \]
    Let $f_1$ be a linear functional defined on a subset consisting of functions in $V(\mathbf{1},\mathbf{1})$ supported on the closed orbit $B_2(E)\GL_2(F)$ given by
    \[f_1(\varphi)=\int_{B_2(F)\backslash\GL_2(F) }\varphi(x)dx  \]
    for $\varphi\in V(\mathbf{1},\mathbf{1})$ supported on $B_2(E)\GL_2(F)$. There is a natural extension of $f_1$ to the whole space $V(\mathbf{1},\mathbf{1})$, still denoted by $f_1$, due to the embedding
    \[0\to \Hom_{\GL_2(\FF)}(\Ind_{B_2(F)}^{\GL_2(F)}(\delta_{B_2(F)}^{1/2}),\mathbf{1}_2)\to\Hom_{\GL_2(\FF)}(V(\mathbf{1},\mathbf{1}),\mathbf{1}_2) . \]
    Then $f_1$ gives a $\GL_2(F)$-invariant linear functional on $\pi$ with support $B_2(E)\GL_2(\FF)$.
    Note that there is a $\GL_2(E)$-invariant distribution on $P^1(E)=B_2(E)\backslash\GL_2(E)$ if and only if $\delta_{B(E)}=\mathbf{1}$, i.e., $q_\EE \equiv 1 \pmod{\ell}$ (see \cite[Chapter VII.6, Thm 3]{bourbaki}).
    Denote by $d\mu$ the $\GL_2(E)$-invariant measure on $P^1(E)$ and so it is $\GL_2(F)$-invariant. Then the restriction of $d\mu$ on $P^1(F)=B_2(F)\backslash \GL_2(F)$ is zero since $P^1(F)$ has measure zero with respect to $d\mu$. Therefore $d\mu$ and $f_1$ generate two different $\GL_2(F)$-invariant linear functionals on $V(\mathbf{1},\mathbf{1})$.
    Thus $\dim\Hom_{\GL_2(F)}(V(\mathbf{1},\mathbf{1}),\mathbf{1}_2)=2$ (the dimension is bounded by 2 by the exact sequence at the beginning of the proof).
\end{proof}

Let us come back to the criterion for being distinguished for irreducible non\hyph supercuspidal representations. By Lemma \ref{lemPrincSer}, we have the result for all irreducible principal series. We refer to \cite{vignerasGL2} or \cite{vigneras} for the following facts about $\pi(\chi_1,\chi_2)$. Note that $\pi(\chi_1,\chi_2)$ is reducible if and only if $\chi_1=\nu \chi_2$ or $\chi_2=\nu \chi_1$. Hence let $\chi$ be a character of $\EE^{\times}$ and we are left to study the irreducible sub-quotients of $\pi(\chi \nu^{-1/2},\chi \nu^{1/2})$. If $q_{\EE} \not\equiv -1 \pmod{\ell}$, then $\pi(\chi \nu^{-1/2},\chi \nu^{1/2})$ has length 2 with irreducible sub-quotients $\chi$ and $\St_{\chi}$. If $q_{\EE} \equiv -1 \pmod{\ell}$, the length is 3 {and} the sub-quotients are $\chi$, $\chi \nu_2$ and $\Spe_{\chi}$.

\begin{lem}
    \label{lemExt}
    Let $\chi$ be a quadratic character of $\FF^{\times}$. Then $\Ext^1_{\GL_2(F)}(\mathbf{1}_2,\chi \circ \det)=0$ where $\Ext^1_{\GL_2(F)}$ is the Ext functor in the category of the smooth representations of $\GL_2(F)$ with trivial central character.
\end{lem}

\begin{proof}
    Let $M$ be an extension of $\mathbf{1}_2$ by $\chi \circ \det$ in the category of the smooth representations of $\GL_2(F)$ with trivial central character. The representation $M$ has the form
    \[g\mapsto\begin{pmatrix}
            \chi(\det(g)) & \alpha(g) \\0&1
        \end{pmatrix} \]
    for $g\in\GL_2(F)$. If $g_1,g_2 \in\GL_2(F)$, then $\alpha(g_1g_2)=\alpha(g_1)+\chi(\det(g_1))\alpha(g_2)$. Since $M$ has trivial central character, $\alpha(z)=0$ for every element $z$ in the center of $\GL_2(F)$. It is obvious that $\alpha(g)=0$ for all $g$ in $\SL_2(F)$ since $\SL_2(F)$ is a perfect group.  Let $g \in \GL_2(F)$ and $t = \det(g) \in F^\times$. By writing $g=g_1g_2$ with $g_1=\left(\begin{smallmatrix}
                t & 0 \\0&1
            \end{smallmatrix}\right)$ and $g_2=g_1^{-1}g \in \SL_2(F)$ we see that $\alpha(g)=\alpha(g_1)$. Let $\bar{\alpha}:F^\times \to \Fl$ be defined by $\bar{\alpha}(t)=\alpha(\left(\begin{smallmatrix}
                    t & 0 \\0&1
                \end{smallmatrix}\right))$. Then for all $g \in \GL_2(F)$, $\alpha(g)=\bar{\alpha}(\det(g))$. Moreover, $\bar{\alpha}$ is a cocycle satisfying $\bar{\alpha}(t_1t_2)=\bar{\alpha}(t_1)+\chi(t_1)\bar{\alpha}(t_2)$, for $t_1,t_2 \in F^\times$.

    Let $t_1,t_2 \in F^\times$. The equality $\bar{\alpha}(t_1t_2)=\bar{\alpha}(t_2t_1)$ gives us that $\bar{\alpha}(t_1)+\chi(t_1)\bar{\alpha}(t_2)=\bar{\alpha}(t_2)+\chi(t_2)\bar{\alpha}(t_1)$. If $\chi \neq \mathbf{1}$, then there exists $t_1$ such that $\chi(t_1) \neq 1$. Therefore, for all $t_2 \in F^\times$, we have $\bar{\alpha}(t_2)=c(\chi(t_2)-1)$ where $c=\bar{\alpha}(t_1)/(\chi(t_1)-1)$ is a constant. Thus, for all $g \in \GL_2(F)$, $\alpha(g)=c\cdot (\chi(\det(g))-1)$ is a coboundary and the extension $M$ splits.

    We are left with the case $\chi=\mathbf{1}$. In this case $\bar{\alpha}:F^\times \to \Fl$ is a morphism of group. Moreover, $\alpha$ is trivial on the center of $\GL_2(F)$ so the morphism $\bar{\alpha}$ is trivial on $F^{\times 2}$ the subgroup of $F^\times$ of square elements. Since $F^\times / F^{\times 2}$ is a 2-group and $\ell \neq 2$, we get that $\bar{\alpha}$ is trivial, and M also splits.
\end{proof}

\begin{rem}\begin{enumerate}
        \item 	If $\ell=2$, then Lemma \ref{lemExt} does not hold any more.
        \item If there is no restriction for the central character, one can easily find the nontrivial extension of $\mathbf{1}_2$ by $\mathbf{1}_2$, which is of the form
              \[g\mapsto\begin{pmatrix}
                      1 & \ln |\det(g)| \\0&1
                  \end{pmatrix} \]
              where $\ln$ is the natural logarithm function defined on the multiplicative group consisting of positive real numbers.
    \end{enumerate}
\end{rem}

\begin{thm}
    \label{thmCuspDist}
    Let $\chi$ be a character of $\EE^{\times}$.

    \begin{enumerate}
        \item If $q_{\EE}^2\not\equiv 1 \pmod{\ell}$, then $\St_\chi$ is distinguished if and only if $\chi_{|{F^\times}}=\omegaEF$.
        \item If $q_{\EE} \equiv -1 \pmod{\ell}$, then $\Spe_\chi$ is distinguished if and only if $q_F\equiv -1 \pmod{\ell}$ and $\chi_{|{F^\times}}=\omegaEF$ or $\nu^{1/2}_{|{F^\times}}$.
        \item If $q_{\EE}\equiv 1 \pmod{\ell}$, then $\St_\chi$ is distinguished if and only if $\chi_{|{F^\times}}=\omegaEF$ or $\chi_{|{F^\times}}=\mathbf{1}$ with $q_F\equiv 1 \pmod{\ell}$.
    \end{enumerate}

\end{thm}

\begin{proof}
    Suppose that $q_{\EE}\equiv -1 \pmod{\ell}$. If $\Spe_\chi$ is distinguished then by Theorem \ref{thmSecherreDist} (2), $\Spe_\chi$ has trivial central character and so $\chi_{|{F^\times}}^2=1$. Thus we may assume that $\chi_{|{F^\times}}^2=\mathbf{1}$. The principal series $\pi(\chi\nu^{1/2},\chi\nu^{-1/2})$ has length $3$. There are two exact sequences
    \[0\to J_\chi \to \pi(\chi\nu^{1/2},\chi\nu^{-1/2}) \to \chi \circ \det \to 0  \]
    and
    \[0\to (\chi \circ \det)\nu_2\to J_\chi \to \Spe_{\chi} \to 0  \]
    of $\GL_2(E)$-modules. Taking the functor $\Hom_{\GL_2(F)}(-,\mathbf{1}_2)$, we have the following exact sequence
    \[ 0 \to \Hom_{\GL_2(F)}(\chi \circ \det,\mathbf{1}_2) \to \Hom_{\GL_2(F)}(\pi(\chi\nu^{1/2},\chi\nu^{-1/2}),\mathbf{1}_2)\]
    \[     \to \Hom_{\GL_2(F)}(J_\chi,\mathbf{1}_2) \to \Ext^1_{\GL_2(F)}(\chi \circ \det,\mathbf{1}_2) \]
    where $\Ext^1_{\GL_2(F)}$ is the Ext functor in the category of the smooth representations of $\GL_2(F)$ with trivial central character. By Lemma \ref{lemExt}
    \[
        \Ext_{\GL_2(F)}^1(\mathbf{1}_2,\chi \circ \det)=0.
    \]
    If $\Hom_{\GL_2(F)} (\pi(\chi\nu^{1/2},\chi\nu^{-1/2}),\mathbf{1}_2)=0$ then $\Hom_{\GL_2(F)}(J_\chi,\mathbf{1}_2)=0$ and $\Hom_{\GL_2(F)}(\Spe_{\chi},\mathbf{1}_2)=0$. Hence $\Spe_{\chi}$ is not distinguished.

    Therefore we can assume that $\pi(\chi\nu^{1/2},\chi\nu^{-1/2})$ is distinguished. By Lemma \ref{lemPrincSer} we get that $\chi \chi^{\sigma}=\mathbf{1}$ (that is $\chi_{|{F^\times}}=\mathbf{1}$ or $\chi_{|{F^\times}}=\omegaEF$ by Lemma \ref{lemChiSigma}) or $\chi_{|{F^\times}}\nu^{1/2}_{|{F^\times}}=\chi_{|{F^\times}}\nu^{-1/2}_{|{F^\times}}=\mathbf{1}$ (which can only happen if $\nu_{|{F^\times}}=\mathbf{1}$ that is if $q_F \equiv -1 \pmod{\ell}$). Moreover, in this case, Lemma \ref{lemPrincSer} gives us
    \[
        \dim\Hom_{\GL_2(F)} (\pi(\chi\nu^{1/2},\chi\nu^{-1/2}),\mathbf{1}_2)=1.
    \]

    We have three cases to study.
    \begin{itemize}
        \item Suppose $\chi_{|{F^\times}}=\mathbf{1}$. In this case $\dim\Hom_{\GL_2(F)}(J_\chi,\mathbf{1}_2)=0$ and so $\Spe_{\chi}$ is not distinguished.
        \item Suppose $\chi_{|{F^\times}}=\omegaEF$. Now $\Hom_{\GL_2(F)}(\chi \circ \det,\mathbf{1}_2)=0$ and so $\dim\Hom_{\GL_2(F)}(J_\chi,\mathbf{1}_2)=1$. If $q_F\equiv -1 \pmod{\ell}$, then $\nu_2|_{\GL_2(F)}=\mathbf{1}_2$. Furthermore, the exact sequence
              \[0\to \Hom_{\GL_2(F)}(\Spe_\chi,\mathbf{1}_2)\to \Hom_{\GL_2(F)}(J_\chi,\mathbf{1}_2)\to \Hom_{\GL_2(F)}((\chi \circ \det)\nu_2,\mathbf{1}_2)=0  \]
              implies that $\dim\Hom_{\GL_2(F)}(\Spe_\chi,\mathbf{1}_2)=\dim\Hom_{\GL_2(F)}(J_\chi,\mathbf{1}_2)=1$. If $q_F^2\equiv -1 \pmod{\ell}$, i.e., $E/F$ is unramified, then $\nu_2|_{\GL_2(F)}=\omega_{E/F}$. In this case $\dim\Hom_{\GL_2(F)}((\chi \circ \det)\nu_2,\mathbf{1}_2)=1$ and so $\Hom_{\GL_2(F)}(\Spe,\omega_{E/F})=0$.

        \item Suppose $\chi_{|{F^\times}}=\nu^{1/2}_{|{F^\times}}$ and $q_F\equiv -1 \pmod{\ell}$. In the same way, $\Hom_{\GL_2(F)}(\chi \circ \det,\mathbf{1}_2)=0$ and so $\dim\Hom_{\GL_2(F)}(J_\chi,\mathbf{1}_2)=1$. Since $\Hom_{\GL_2(F)}((\chi \circ \det)\nu_2,\mathbf{1}_2)=0$ we get that $\Spe_\chi$ is distinguished.
    \end{itemize}

    \bigskip

    Suppose $q_{\EE}\equiv 1 \pmod{\ell}$. In this case $\pi(\chi \nu^{-1/2},\chi \nu^{-1/2})$ is reducible and semisimple: $\pi(\chi \nu^{-1/2},\chi \nu^{-1/2})=(\chi \circ \det)\oplus \St_\chi$. If $\St_\chi$ is distinguished then so is $\pi(\chi \nu^{-1/2},\chi \nu^{-1/2})$. By Lemma \ref{lemPrincSer}, $\chi \chi^{\sigma}=\mathbf{1}$ so $\chi_{|{F^\times}}=\mathbf{1}$ or $\chi_{|{F^\times}}=\omegaEF$ (Lemma \ref{lemChiSigma}). If $\chi_{|{F^\times}}=\omegaEF$, then $\dim\Hom_{\GL_2(F)}(\chi \circ \det,\mathbf{1}_2)=0$. Thanks to Lemma \ref{lemPrincSer} $\dim\Hom_{\GL_2(F)}(\pi(\chi\nu^{-1/2},\chi\nu^{-1/2}),\mathbf{1}_2)\geq 1$ and so $\St_\chi$ is distinguished. Now if $\chi_{|{F^\times}}=\mathbf{1}$, $\dim\Hom_{\GL_2(F)}(\chi \circ \det,\mathbf{1}_2)=1$ and from Lemma \ref{lemPrincSer}
    \[\dim\Hom_{\GL_2(F)}(\pi(\chi\nu^{-1/2},\chi\nu^{-1/2}),\mathbf{1}_2)=
        \begin{cases}
            2, & \mbox{ if } q_{\FF}\equiv 1 \pmod{\ell};  \\
            1, & \mbox{ if } q_{\FF}\equiv -1 \pmod{\ell}.
        \end{cases}
    \]
    Thus $\St_\chi$ is distinguished if and only if $q_{\FF}\equiv 1 \pmod{\ell}$.

    \bigskip

    The last case is $q_{\EE}^2\not\equiv 1 \pmod{\ell}$. We will do a similar argument as $q_{\EE}\equiv -1 \pmod{\ell}$. We have an exact sequence
    \[0\to \St_\chi \to \pi(\chi\nu^{1/2},\chi\nu^{-1/2}) \to \chi \circ \det \to 0  \]
    of $\GL_2(E)$-modules. If $\St_\chi$ is distinguished then so is $\pi(\chi\nu^{-1/2},\chi\nu^{1/2})$, so $\chi_{|{F^\times}}=\mathbf{1}$ or $\chi_{|{F^\times}}=\omegaEF$. Taking the functor $\Hom_{\GL_2(F)}(-,\mathbf{1}_2)$, in the category of the smooth representations of $\GL_2(F)$ with trivial central character, we have
    \[ 0 \to \Hom_{\GL_2(F)}(\chi \circ \det,\mathbf{1}_2) \to \Hom_{\GL_2(F)}(\pi(\chi\nu^{1/2},\chi\nu^{-1/2}),\mathbf{1}_2)\]
    \[ \to \Hom_{\GL_2(F)}(\St_\chi,\mathbf{1}_2) \to \Ext^1_{\GL_2(F)}(\chi \circ \det,\mathbf{1}_2).\]
    By Lemma \ref{lemExt} $\Ext_{\GL_2(F)}^1(\mathbf{1}_2,\chi \circ \det)=0$. By Lemma \ref{lemPrincSer} $\dim\Hom_{\GL_2(F)} (\pi(\chi\nu^{1/2},\chi\nu^{-1/2}),\mathbf{1}_2)=1$. Therefore, $\St_\chi$ is distinguished if and only if $\chi_{|{F^\times}}=\omegaEF$.
\end{proof}

\begin{rem}
    It can be seen from Theorem \ref{thmCuspDist} that there are in the modular case new phenomena that do not appear in the complex setting.
    \begin{itemize}
        \item When $q_{\FF}\equiv 1 \pmod{\ell}$, the Steinberg representation $\St$ is both $\GL_2(F)$-distinguished and $(\GL_2(F),\omega_{E/F})$-distinguished.
        \item When $q_{\EE}\equiv -1 \pmod{\ell}$ and $\EE/\FF$ is unramified (that is $q_{\FF}^2\equiv -1 \pmod{\ell}$ and so $q_{\FF}\not\equiv -1 \pmod{\ell}$) the special representation $\Spe$ is neither $\GL_2(F)$-distinguished nor $(\GL_2(F),\omega_{E/F})$\hyph distinguished. This has been mentioned by Vincent Sécherre in \cite[Rem. 2.8]{vincent2019supercusp}.
        \item When $q_{\EE}\equiv -1 \pmod{\ell}$ and $\EE/\FF$ is ramified (and so $q_{\FF}\equiv -1 \pmod{\ell}$) the special representation $\Spe$ is $(\GL_2(F),\omega_{E/F})$-distinguished and is also $(\GL_2(F),\nu^{1/2}_{|{F^\times}})$-distinguished.
    \end{itemize}
\end{rem}

\section{The Prasad conjecture for \texorpdfstring{$\ell$}{l}-modular representations of \texorpdfstring{$\PGL_2$}{PGL2}}

\label{secPrasad}

In \cite{prasad2015arelative}, Dipendra Prasad proposed a conjecture for the multiplicity $\dim \Hom_{G(F)}(\pi,\chi_G) $ under the local Langlands conjecture, where $G$ is a quasi-split reductive group defined over $F$, $\pi$ is an irreducible smooth representation of $G(E)$ lying inside a generic $L$-package and $\chi_G$ is a quadratic character depending on $G$ and the quadratic extension $E/F$.
Let us recall briefly the Prasad conjecture for $\GL_n$, which has been verified for the complex representations due to the work of Flicker, Prasad, Kable, Matringe and so on.
\begin{thm}[The Prasad conjecture for $\GL_{n}$]
    Let $\pi$ be a generic irreducible complex representation of $\GL_{n}(\EE)$ with Langlands parameter $\phi_\pi$. Let $\chi_G=\omega_{E/F}^{n+1}$. Let $U_{n,E/F}$ be the quasi-split unitary group. If $\Hom_{\GL_n(\FF)}(\pi,\chi_G)$ is nonzero, then
    \begin{enumerate}
        \item $\pi^\vee\cong\pi^\sigma$;
        \item there exists a parameter $\tilde{\phi}$ of $U_{n,E/F}$ such that $\tilde{\phi}|_{WD_E}=\phi_\pi$;
        \item $\dim_{\mathbb{C}}\Hom_{\GL_{n}(\FF)}(\pi,\chi_G)=|F(\phi_\pi)| $
              where
              \[
                  F(\phi_\pi):=\left\{\tilde{\phi}:WD_\FF\rightarrow{}^LU_{n,E/F}\Big|\tilde{\phi}|_{WD_\EE}=\phi_\pi \right\}
              \]
              and $|F(\phi_\pi)|$ denotes its cardinality.
    \end{enumerate}
    Conversely, if there exists a parameter $\tilde{\phi}$ of $U_{n,E/F}$ such that $\tilde{\phi}|_{WD_E}=\phi_\pi$, then $\Hom_{\GL_n(\FF)}(\pi,\chi_G)$ is nonzero.
\end{thm}

Since the local Langlands correspondence for the $\ell$-modular representations of $G(F)$ has not been set up in general, except for $G=\GL_n$, we are concerned only with the simplest case where $G=\PGL_2$.

\begin{thm}[The Prasad conjecture for $\PGL_2$]
    Let $\pi$ be an irreducible generic complex representation of $\PGL_2(\EE)$ with Langlands parameter $\phi_\pi$. Then $\Hom_{\PGL_2(\FF)}(\pi,\omega_{E/F})$ is nonzero if and only if
    % \begin{enumerate}
    % \item $\pi\cong\pi^\sigma$;
    there exists a parameter $\tilde{\phi}$ of $\PGL_2(\FF)$ such that $\tilde{\phi}|_{WD_F}=\phi_\pi$.
    %  \item \[
    %        \dim_{\mathbb{C}}\Hom_{\PGL_2(\FF)}(\pi,\chi_G)+\dim_{\mathbb{C}}\Hom_{PD^\times}(\pi,%\omega_{E/F}) =|F(\phi_\pi)|
    %     \]
    %     where $PD^\times$ is the pure inner form of $\PGL_2(\FF)$ and
    %   \[
    %      F(\phi_\pi):=\left\{\tilde{\phi}:WD_F\rightarrow\SL_2(\mathbb{C})\Big|\tilde{\phi}|_{WD_E}=\%phi_\pi \right\}.\]
    %\end{enumerate}
    \label{Prasad:PGL}
    %Conversely, if there exists a parameter $\tilde{\phi}$ of $\PGL_2(\FF)$ such that %$\tilde{\phi}|_{WD_F}=\phi_\pi$, then $\Hom_{\PGL_2(\FF)}(\pi,\omega_{E/F})$ is nonzero.
\end{thm}
%\bigskip

\bigskip

First, we will show that the Prasad conjecture is not valid in the $\ell$-modular setting (when $\ell$ is non-banal), i.e. there exists a parameter $\tilde{\phi}$ of $\PGL_2(\FF)$ such that $\tilde{\phi}|_{WD_F}=\phi_\pi$, however the generic representation $\pi$ is not $\omega_{E/F}$-distinguished by $\PGL_2(F)$. Then, we provide a potential solution. To do that we define a non-trivial injection $P$ from nilpotent Weil-Deligne representations $\Nilp_{\Fl}(W_{\EE},\SL_2)$ to equivalence classes of non-nilpotent one $[\WDRep_{\Fl}(W_{\EE},\SL_2)]$. Composing the local Langlands correspondence of Vignéras $PV$ with $P$ gives us a modular version of the Prasad conjecture.  That is, an irreducible generic $\Fl$-representation $\pi$ of $\PGL_2(\EE)$ is $\omegaEF$-distinguished if and only if there exists $\Psi_{F} \in \WDRep_{\Fl}(W_{\FF},\SL_2)$ such that $\Psi_{F|W_{\EE}} \sim P\circ PV(\pi)$.

\subsection{The Langlands correspondence for $\PGL_2$}

Firstly, we recall how to get a Langlands correspondence for $\PGL_2$ using the correspondence for $\GL_2$.

\bigskip

Let $R=\Ql$ or $\Fl$. We denote by $\WDRep_{R}(W_{\EE},\SL_2)$ the subset of $\WDRep_{R}(W_{\EE},\GL_2)$ composed of the elements $(\varphi,N)$ such that $\im(\varphi) \subset \SL_2(R)$ and $\tr(N)=0$. Let $\Nilp_{R}(W_{\EE},\SL_2) := \WDRep_{R}(W_{\EE},\SL_2) \cap \Nilp_{R}(W_{\EE},\GL_2)$.

\begin{rem}
    Let $(\varphi_1,N_1)$ and $(\varphi_2,N_2)$ be two semisimple Weil-Deligne representations with $\im(\varphi_1) \subset \SL_2(R)$ and $\im(\varphi_2) \subset \SL_2(R)$. If $(\varphi_1,N_1)$ and $(\varphi_2,N_2)$ are isomorphic in $\GL_2$ then they are also isomorphic in $\SL_2$. Indeed, let $A \in \GL_2(R)$ such that $A^{-1} \varphi_1 A = \varphi_2$ and $A^{-1} N_1 A= N_2$. Since $R$ is algebraically closed, let $\alpha \in R$ such that $\alpha^2 = \det(A)^{-1}$. Let $B:= \alpha A$. Then $B \in \SL_2(R)$ and $B^{-1} \varphi_1 B = \varphi_2$, $B^{-1} N_1 B= N_2$.
\end{rem}

Let $L$ be any map
\[
    L : \Irr_R(\GL_2(\EE)) \to \Nilp_{R}(W_{\EE},\GL_2)
\]
such that $L$ sends the central character to the determinant (that is, if $\pi \in \Irr_R(\GL_2(\EE))$ and $\omega_\pi$ is the central character of $\pi$, then $\omega_\pi$ corresponds to $\det(L(\pi))$ via local class field theory). Then $L$ induces a map

\[
    PL : \Irr_R(\PGL_2(\EE)) \to \Nilp_{R}(W_{\EE},\SL_2)
\]
making the following diagram commutes

\[
    \xymatrix{
        \Irr_R(\PGL_2(\EE)) \ar[d] \ar[r]^{PL} & \Nilp_{R}(W_{\EE},\SL_2) \ar[d]\\
        \Irr_R(\GL_2(\EE)) \ar[r]^{L} & \Nilp_{R}(W_{\EE},\GL_2)}
\]
where the vertical map on the left-hand side is given by the projection $\GL_2(\EE) \to \PGL_2(\EE)$ and the vertical map on the right-hand side is the inclusion $\Nilp_{R}(W_{\EE},\SL_2) \subset \Nilp_{R}(W_{\EE},\GL_2)$.

In particular, since the $V$ correspondence sends the central character to the determinant (\cite[Lem. 6.4]{MatringeKurinczuk}) we get a map $PV : \Irr_R(\PGL_2(\EE)) \to \Nilp_{R}(W_{\EE},\SL_2)$.

\subsection{Problem with the modular Prasad Conjecture}

\label{secIssuePrasad}

We shall show that the Prasad conjecture does not hold for $\ell$-modular representations with the $V$ correspondence. One may think that it might work with a different choice of bijection in Theorem \ref{thmBijNilp}. We will show in this subsection that such kind of bijection does not exist; see Corollary \ref{corIssuePrasad}.

\bigskip

Let $L$ be a map
\[
    L : \Irr_{\Fl}(\PGL_2(\EE)) \to \Nilp_{\Fl}(W_{\EE},\SL_2)
\]
such that the semisimple part is given by $V_{ss}$.

\bigskip

When $q_E \equiv -1 \pmod{\ell}$, the image under $L$ of the special representation $\Spe$, has $\nu^{-1/2} \oplus \nu^{1/2}$ for the semisimple part.

% Let us consider the parabolic induction of the trivial representation of $\GL_1(\EE)\times\GL_1(\EE)$  to $\GL_2(\EE)$. If $q_E \not\equiv -1 \pmod{\ell}$ then it has length 2 with irreducible quotients $\mathbf{1}$ and $\St$. If $q_E \equiv -1 \pmod{\ell}$ it has length 3 with irreducible quotients $\mathbf{1}$, a quadratic character $\mu$ and $\Spe$. Under the map $L$, all of these representations have $\nu^{-1/2} \oplus \nu^{1/2}$  for the semisimple part.

\begin{lem}
    \label{lemLiftLangSp}
    Let $\EE$ be a quadratic extension of $\FF$. Let us assume that $q_E \equiv -1 \pmod{\ell}$. Then there exists $\Phi_{\FF} \in \Nilp_{\Fl}(W_{\FF},\SL_2)$ such that $\Phi_{\FF|W_{\EE}}=L(\Spe)$.
\end{lem}

\begin{proof}
    There are 3 elements in $\Nilp_{\Fl}(W_{\EE},\SL_2)$ with semisimple part $\nu^{-1/2} \oplus \nu^{1/2}$ : $(\nu^{-1/2} \oplus \nu^{1/2},0)$, $(\nu^{-1/2} \oplus \nu^{1/2},N)$ and $(\nu^{1/2} \oplus \nu^{3/2},N)$, with $N = \begin{pmatrix}
            0 & 1 \\
            0 & 0
        \end{pmatrix}$. Let $\nu_{\FF}$ be a lift of $\nu$. The following  Weil-Deligne representations in $\Nilp_{\Fl}(W_{\FF},\SL_2)$, $(\nu_{\FF}^{-1/2} \oplus \nu_{\FF}^{1/2},0)$, $(\nu_{\FF}^{-1/2} \oplus \nu_{\FF}^{1/2},N)$ and $(\nu_{\FF}^{1/2} \oplus \nu_{\FF}^{3/2},N)$ are respective lifts of the previous 3 elements of $\Nilp_{\Fl}(W_{\EE},\SL_2)$. Hence, one of them is a lift of $L(\Spe)$.
\end{proof}

\begin{cor}
    \label{corIssuePrasad}
    In the non banal case, the Prasad conjecture is not true for any map $L$ as above.
\end{cor}

\begin{proof}
    If $q_\EE \equiv -1 \pmod{\ell}$ and $\EE / \FF$ is unramified, then by Theorem \ref{thmCuspDist}, the special representation $\Spe$ is not $\omegaEF$-distinguished. However, the semisimple part of $L(\Spe)$ is $\nu^{-1/2} \oplus \nu^{1/2}$. By Lemma \ref{lemLiftLangSp} the Langlands parameter $L(\Spe)$ of $\Spe$ can be lifted to $W_\FF$.
\end{proof}

\subsection{Non-nilpotent Weil-Deligne representations}

To fix the issue with the Prasad conjecture in the modular setting discussed in the previous paragraph, we want to modify the $V$ correspondence. In \cite{MatringeKurinczuk}, Kurinczuk and Matringe modify the local Langlands of Vignéras using non-nilpotent Weil-Deligne representations. We will do something similar to solve our problem.

\bigskip

Kurinczuk and Matringe \cite[Def. 4.8]{MatringeKurinczuk} defined an equivalence relation $\sim$ on $\WDRep_{R}(W_{\EE},\GL_2)$. We recall the definition here. Let $(\Phi,U)$ and $(\Phi',U')$  be two Weil-Deligne representations (up to isomorphism) in $\WDRep_{R}(W_{\EE},\GL_2)$. Then
\begin{enumerate}
    \item If $(\Phi,U)$ and $(\Phi',U')$ are indecomposable (as Weil-Deligne representations), we say that $(\Phi,U) \sim (\Phi',U')$ if there exists $\lambda \in R^{\times}$ such that
          \[(\Phi',U') \simeq (\Phi,\lambda U).\]
    \item In the general case, $(\Phi,U) \sim (\Phi',U')$ if one can decompose $(\Phi,U) = \bigoplus_{i=1}^{r} (\Phi_i,U_i)$ and $(\Phi',U') = \bigoplus_{i=1}^{r} (\Phi_i',U_i')$ with indecomposable summands such that $(\Phi_i,U_i) \sim (\Phi'_i,U'_i)$.
\end{enumerate}

We denote by $[(\Phi,U)]$ the equivalence class of $(\Phi,U)$ and by
\[[\WDRep_{R}(W_{\EE},\GL_2)] := \WDRep_{R}(W_{\EE},\GL_2)/{\sim}.\]
We define also an equivalence relation $\sim$ on $\WDRep_{R}(W_{\EE},\SL_2)$ through the inclusion $\WDRep_{R}(W_{\EE},\SL_2) \subset \WDRep_{R}(W_{\EE},\GL_2)$ and we denote by $[\WDRep_{R}(W_{\EE},\SL_2)] := \WDRep_{R}(W_{\EE},\SL_2)/{\sim}$.

\bigskip

Let $\chi$ be a quadratic character of $\EE^\times$. If $\ell \mid q_\EE-1$, we denote by $\Psi_{\St,\chi} \in \Nilp_{\Fl}(W_{\EE},\SL_2)$, the Weil-Deligne representation $\Psi_{\St,\chi} = (\chi\nu^{-1/2} \oplus \chi\nu^{-1/2},N)$ with $N = \begin{pmatrix}
        0 & 1 \\
        0 & 0
    \end{pmatrix}$. And if $\ell \mid q_\EE+1$, we denote by $\Psi_{\Spe,\chi} \in \Nilp_{\Fl}(W_{\EE},\SL_2)$, the Weil-Deligne representation $\Psi_{\Spe,\chi} = (\chi\nu^{-1/2} \oplus \chi\nu^{1/2},0)$. We define an injection

\[P:  \Nilp_{\Fl}(W_{\EE},\SL_2) \hookrightarrow [\WDRep_{\Fl}(W_{\EE},\SL_2)]\]
by
\[
    P(\Psi) = \left\{
    \begin{array}{ll}

        {\left[\chi\nu^{-1/2} \oplus \chi\nu^{-1/2},\begin{pmatrix}
                                                                0 & 1 \\
                                                                1 & 0
                                                            \end{pmatrix}\right]} & \mbox{if } \ell \mid q_\EE-1 \mbox{ and } \Psi = \Psi_{\St,\chi}  \\
        {\left[\chi\nu^{-1/2} \oplus \chi\nu^{1/2},\begin{pmatrix}
                                                           0 & 1 \\
                                                           1 & 0
                                                       \end{pmatrix}\right]}  & \mbox{if } \ell \mid q_\EE+1 \mbox{ and } \Psi = \Psi_{\Spe,\chi} \\

        [\Psi]                                                      & \mbox{otherwise.}
    \end{array}
    \right.
\]
Since $\Nilp_{\Fl}(W_{\EE},\SL_2)=[\Nilp_{\Fl}(W_{\EE},\SL_2)]$ due to \cite[Prop. 4.11]{MatringeKurinczuk}, $P$ is clearly an injection.

\begin{rem}
    \begin{enumerate}
        \item When $\ell$ is banal, that is $\ell \nmid q_\EE^2-1$, $P$ is just the identity, as we have $[\WDRep_{\Fl}(W_{\EE},\SL_2)] = \Nilp_{\Fl}(W_{\EE},\SL_2)$.
        \item  The map $P$ is not exactly the $CV$ map of \cite{MatringeKurinczuk}. The image is different for non banal supercuspidal representation and the Steinberg representation when $\ell \mid q_\EE -1$.
    \end{enumerate}
\end{rem}

\subsection{Lifting for non-irreducible Weil-Deligne representations}

In Section \ref{secNonSupercuspDist} we have described which of the non-supercuspidal irreducible representations are distinguished. To prove the Prasad conjecture, we also need to inquire when Weil-Deligne representations of $W_\EE$ can be lifted to $W_\FF$. This is what we do in this section.

\bigskip

We begin by giving a criterion to lift the semisimple part of a Weil-Deligne representation.

\begin{lem}
    \label{lemLiftSemiSimple}
    Let $\chi$ be a character of $\EE^{\times}$. Let $\Psi : W_{\EE} \to \SL_2(\Fl)$ defined by $\Psi = \chi \nu^{-1/2} \oplus \chi^{-1} \nu^{1/2}$. Then there exists a semisimple $\Psi_F : W_{\FF} \to \SL_2(\Fl)$ such that $\Psi_{F|W_{\EE}} = \Psi$ if and only if $\chi=\chi^{\sigmaEF}$ or $\chi\neq\chi^{\sigmaEF}$ and $\chi_{|\FF^{\times}}=\omegaEF \nu^{1/2}_{|\FF^{\times}}$.

    Moreover, if $\chi=\chi^{\sigmaEF}$ then $\Psi_F = \eta \nu^{-1/2} \oplus \eta^{-1} \nu^{1/2}$, with $\eta$ a character of $\FF^{\times}$ such that $\chi=\eta \circ \Nm_{\EE/\FF}$. And if $\chi\neq\chi^{\sigmaEF}$ and $\chi_{|\FF^{\times}}=\omegaEF \nu^{1/2}_{|\FF^{\times}}$ then $\Psi_F = \Ind_{W_{\EE}}^{W_{\FF}}(\chi \nu^{-1/2})$ (which is irreducible).
\end{lem}

\begin{proof}
    We have two cases if there is a lift to $W_{\FF}$: the two-dimensional representation of $W_{\FF}$ is irreducible or it is the sum of two characters.

    Let us first study when there is a lift which is the sum of two characters. In this case we can lift $\chi$ to $W_{\FF}$. This is equivalent to $\chi=\eta \circ \Nm_{\EE/\FF}$ with $\eta$ a character of $\FF^{\times}$ also equivalent to $\chi=\chi^{\sigmaEF}$. If this condition is satisfied then the lift is $\Psi_F = \eta \nu^{-1/2} \oplus \eta^{-1} \nu^{1/2}$.

    Now we deal with the case where there is an irreducible lift. Then this lift must be $\Psi_F = \Ind_{W_{\EE}}^{W_{\FF}}(\chi \nu^{-1/2})$. Let $\mu = \chi \nu^{-1/2}$. The representation $\Psi_F$ is irreducible if and only if $\mu \neq \mu^{\sigma}$ if and only if $\chi \neq \chi^{\sigma}$. So let us assume that $\chi \neq \chi^{\sigma}$. Also if $\Psi_F$ is a lift of $\Psi$ then $\mu^{\sigma}=\chi^{-1} \nu^{1/2}$ that is $\mu \mu^{\sigma}=\mathbf{1}$ or by Lemma \ref{lemChiSigma} $\mu_{|\FF^{\times}}=\mathbf{1}$ or $\omegaEF$. If these conditions are satisfied then $\Ind_{W_{\EE}}^{W_{\FF}}(\mu)$ is an irreducible lift of $\Psi_F$ in $\GL_2(\Fl)$. We are left to prove at which condition it is in $\SL_2(\Fl)$. By the previous condition we already have $\mu \mu^{\sigma}=1$. Thus for $w\in W_{\EE}$, $\Ind_{W_{\EE}}^{W_{\FF}}(\mu)(w) \in \SL_2(\Fl)$. Let $s\in W_{\FF} \setminus W_{\EE}$. Then
    \[
        \Ind_{W_{\EE}}^{W_{\FF}}(\mu)(s)= \begin{pmatrix}
            0 & \mu(s^2) \\
            1 & 0
        \end{pmatrix}.
    \] Therefore, $\Ind_{W_{\EE}}^{W_{\FF}}(\mu)(s) \in \SL_2(\Fl)$ if and only if $\mu(s^2)=-1$ if and only if $\mu_{|\FF^{\times}}=\omegaEF$. This finishes the proof.
\end{proof}

\begin{rem}
    Note that, when $\chi$ is quadratic, the case where $\Psi_F = \Ind_{W_{\EE}}^{W_{\FF}}(\chi \nu^{-1/2})$ is irreducible can only happen when $q_{\EE}\equiv q_{\FF}\equiv -1 \pmod{\ell}$. Indeed, we have $\chi^{\sigma}=\chi \nu$ and $\chi \neq \chi^{\sigma}$. Therefore $\nu \neq \mathbf{1}$ and $q_{\EE}\not\equiv 1 \pmod{\ell}$. Also since $\chi^{\sigma}=\chi \nu$, we get that $\nu_{|\FF^{\times}}=\mathbf{1}$ and so that $q_{\FF}^2\equiv 1 \pmod{\ell}$. Hence $q_{\EE}\equiv q_{\FF}\equiv -1 \pmod{\ell}$.
\end{rem}

Now we can examine when a Weil-Deligne representation $(\Psi,N)$ can be lifted. The $N$ will be chosen such that these representations correspond under $P\circ PV$ to an irreducible generic representation of $\PGL_2$.

\begin{lem}
    \label{lemLiftWDSp}
    Let $q_{\EE}\equiv -1 \pmod{\ell}$ and $\chi$ be a quadratic character of $\EE^{\times}$. Let $\Psi \in [\WDRep_{\Fl}(W_{\EE},\SL_2)]$ defined by $\Psi = [\chi\nu^{-1/2} \oplus \chi\nu^{1/2},N]$, with $ N=\begin{pmatrix}
            0 & 1 \\
            1 & 0
        \end{pmatrix}$. Then there exists $\Psi_{F} \in \WDRep_{\Fl}(W_{\FF},\SL_2)$ such that $\Psi_{F|W_{\EE}} \sim \Psi$ in $[\WDRep_{\Fl}(W_{\EE},\SL_2)]$ if and only if $q_{\FF}\equiv -1 \pmod{\ell}$ and $\chi_{|\FF^{\times}}=\mathbf{1}$ or $\omegaEF \nu^{1/2}_{|\FF^{\times}}$.
\end{lem}

\begin{proof}

    If we have a lift $\Psi_{F}$ then we also have a lift of the semisimple part of $\Psi$. By Lemma \ref{lemLiftSemiSimple} this is possible only if $\chi=\chi^{\sigmaEF}$ (which is equivalent by Lemma \ref{lemChiSigma} to $\chi_{|\FF^{\times}}=\mathbf{1}$ or $\omegaEF$) or $\chi\neq\chi^{\sigmaEF}$ and $\chi_{|\FF^{\times}}=\omegaEF \nu^{1/2}_{|\FF^{\times}}$.

    If $\chi_{|\FF^{\times}}=\mathbf{1}$, then by Lemma \ref{lemLiftSemiSimple}, the semisimple part of $\Psi_{F}$ should be $\eta \nu^{-1/2} \oplus \eta^{-1} \nu^{1/2}$, with $\eta$ a character of $\FF^{\times}$ such that $\chi=\eta \circ \Nm_{\EE/\FF}$. Note that $\chi_{|\FF^{\times}}=\mathbf{1}$ implies that $\eta^2=\mathbf{1}$.  If $q_{\FF} \equiv -1 \pmod{\ell}$ then we can take $\Psi_{\FF} = [\eta\nu^{-1/2} \oplus \eta\nu^{1/2},N]$. If $q_{\FF} \not\equiv -1 \pmod{\ell}$ then $q_{\FF}^2 \not\equiv 1 \pmod{\ell}$ ($\EE$ is a quadratic extension of $\FF$). In this case, $[\WDRep_{\Fl}(W_{\FF},\SL_2)]=\Nilp_{\Fl}(W_{\FF},\SL_2)$. Since $N$ is not nilpotent, $\Psi$ cannot be the restriction of an element of $\Nilp_{\Fl}(W_{\FF},\SL_2)$.

    If $\chi_{|\FF^{\times}}=\omegaEF$, again by Lemma \ref{lemLiftSemiSimple}, the semisimple part of $\Psi_{F}$ should be $\eta \nu^{-1/2} \oplus \eta^{-1} \nu^{1/2}$. This time $\eta^2=\omegaEF \neq \mathbf{1}$. Thus the only Weil-Deligne representation $[\eta \nu^{-1/2} \oplus \eta^{-1} \nu^{1/2},M]$ is with $M=0$ and is not a lift of $\Psi$.

    Suppose $\chi_{|\FF^{\times}}=\omegaEF \nu^{1/2}_{|\FF^{\times}}$. We get that $\nu_{|\FF^{\times}}=1$ so $q_{\FF}\equiv -1 \pmod{\ell}$. This time Lemma \ref{lemLiftSemiSimple} tells us that the semisimple part of $\Psi_{F}$ should be $\Ind_{W_{\EE}}^{W_{\FF}}(\mu)$, with $\mu =\chi \nu^{-1/2}$. Let $ M=\begin{pmatrix}
            0  & 1 \\
            -1 & 0
        \end{pmatrix}$. First, let us show that $\Psi_F:=(\Ind_{W_{\EE}}^{W_{\FF}}(\mu),M) \in \WDRep_{\Fl}(W_{\FF},\SL_2)$. Let $s\in  W_{\FF} \setminus  W_{\EE}$. For $w \in W_{\EE}$, we get that $\Ind_{W_{\EE}}^{W_{\FF}}(\mu)(w)= \begin{pmatrix}
            \mu(w) & 0          \\
            0      & \mu^{s}(w)
        \end{pmatrix}$. Since $\mu^{\sigma}=\mu \nu$ we get

    \[
        \begin{pmatrix}
            \mu(w) & 0          \\
            0      & \mu^{s}(w)
        \end{pmatrix}\begin{pmatrix}
            0  & 1 \\
            -1 & 0
        \end{pmatrix} = \nu(w) \begin{pmatrix}
            0  & 1 \\
            -1 & 0
        \end{pmatrix} \begin{pmatrix}
            \mu(w) & 0          \\
            0      & \mu^{s}(w)
        \end{pmatrix}.
    \]
    Since $\mu(s^2)=-1$, we also have $\Ind_{W_{\EE}}^{W_{\FF}}(\mu)(s)= \begin{pmatrix}
            0 & -1 \\
            1 & 0
        \end{pmatrix}$. As we are in the case where $q_{\FF}\equiv -1 \pmod{\ell}$, the extension $\EE/\FF$ is ramified. The element $s$ must then be in the inertia subgroup $s\in I_{\FF}$ and therefore $\nu(s)=1$. This gives us:

    \[
        \begin{pmatrix}
            0 & -1 \\
            1 & 0
        \end{pmatrix}\begin{pmatrix}
            0  & 1 \\
            -1 & 0
        \end{pmatrix} = \nu(s) \begin{pmatrix}
            0  & 1 \\
            -1 & 0
        \end{pmatrix} \begin{pmatrix}
            0 & -1 \\
            1 & 0
        \end{pmatrix}.
    \]
    We have checked that $\Psi_F=(\Ind_{W_{\EE}}^{W_{\FF}}(\mu),M) \in \WDRep_{\Fl}(W_{\FF},\SL_2)$. To finish the proof we need to show that $\Psi_{F|W_{\EE}} \sim \Psi$. We have $\Psi_{F|W_{\EE}}=(\chi\nu^{-1/2} \oplus \chi\nu^{1/2},M)$. And $(\chi\nu^{-1/2} \oplus \chi\nu^{1/2},M) \sim \Psi$ by \cite[Lem. 4.23]{MatringeKurinczuk}.
\end{proof}

\begin{lem}
    \label{lemLiftWDSt}
    Let $q_{\EE}\equiv 1 \pmod{\ell}$ and $\chi$ be a quadratic character of $\EE^{\times}$. Let $\Psi \in [\WDRep_{\Fl}(W_{\EE},\SL_2)]$ defined by $\Psi = [\chi \nu^{-1/2} \oplus \chi \nu^{-1/2},N]$, with $ N=\begin{pmatrix}
            0 & 1 \\
            1 & 0
        \end{pmatrix}$. Then there exists $\Psi_{F} \in \WDRep_{\Fl}(W_{\FF},\SL_2)$ such that $\Psi_{F|W_{\EE}} \sim \Psi$ if and only if $\chi_{|\FF^{\times}}=\mathbf{1}$ or $q_{\FF}\equiv 1 \pmod{\ell}$ and $\chi_{|\FF^{\times}}=\omegaEF$.
\end{lem}

\begin{proof}

    By Lemma \ref{lemLiftSemiSimple} the semisimple part of $\Psi$ can be lifted to $W_{\FF}$ if and only if $\chi_{|\FF^{\times}}=\mathbf{1}$ or $\omegaEF$ ($\chi_{|\FF^{\times}}=\omegaEF \nu^{1/2}_{|\FF^{\times}}$ implies that $\chi=\chi^{\sigma}$). Moreover, this lift is $\eta \nu^{-1/2} \oplus \eta^{-1}\nu^{1/2}$, with $\eta$ a character of $\FF^{\times}$ such that $\chi=\eta \circ \Nm_{\EE/\FF}$.

    If $\chi_{|\FF^{\times}}=\mathbf{1}$, then $\eta^2=\mathbf{1}$. We can take $\Psi_{\FF} = (\eta\nu^{-1/2} \oplus \eta\nu^{1/2},N)$ (since $\nu_\FF^2=\mathbf{1}$) and $\Psi_{F|W_{\EE}} \sim \Psi$.

    If $\chi_{|\FF^{\times}}=\omegaEF$, this time $\eta^2=\omegaEF \neq \mathbf{1}$. If $q_{\FF}\equiv -1 \pmod{\ell}$ then $\nu_{\FF}=\omegaEF$. In this case, $\eta^{-1}=\eta \nu_{\FF}$. Thus $\eta \nu^{-1/2} \oplus \eta^{-1}\nu^{1/2} = \eta \nu^{-1/2} \oplus \eta\nu^{-1/2}$. The only  Weil-Deligne representation with this semisimple part is $(\eta \nu^{-1/2} \oplus \eta\nu^{-1/2},0)$ which is not a lift of $\Psi$. If $q_{\FF}\equiv 1 \pmod{\ell}$ let $\Psi_{\FF} := \left(\eta \nu^{-1/2} \oplus \eta^{-1}\nu^{1/2},\begin{pmatrix}
            1 & 0 \\0&-1
        \end{pmatrix}\right)$. We are left to prove that $\Psi_{F|W_{\EE}} \sim \Psi$. Let us remark that $N$ is diagonalizable. Therefore, $(\chi \nu^{-1/2} \oplus \chi \nu^{-1/2},N)$ is isomorphic to $(\chi \nu^{-1/2} \oplus \chi \nu^{-1/2},N')$ with $N'=diag(1,-1)$.

    Hence $\Psi_{F|W_{\EE}} \sim \Psi$ and we get the result.
\end{proof}

\begin{lem}
    \label{lemLiftWDBanal}
    Let $q_{\EE}^2\not\equiv 1 \pmod{\ell}$ and $\chi$ be a quadratic character of $\EE^{\times}$. Let $\Psi \in [\WDRep_{\Fl}(W_{\EE},\SL_2)]$ defined by $\Psi = [\chi\nu^{-1/2} \oplus \chi\nu^{1/2},N]$, with $ N=\begin{pmatrix}
            0 & 1 \\
            0 & 0
        \end{pmatrix}$. Then there exists $\Psi_{F} \in \WDRep_{\Fl}(W_{\FF},\SL_2)$ such that $\Psi_{F|W_{\EE}} \sim \Psi$ if and only if $\chi_{|\FF^{\times}}=\mathbf{1}$.
\end{lem}

\begin{proof}
    By Lemma \ref{lemLiftSemiSimple} we can lift the semisimple part if and only if $\chi_{|\FF^{\times}}=\mathbf{1}$ or $\omegaEF$ and this lift is $\eta\nu^{-1/2} \oplus \eta^{-1}\nu^{1/2}$, with $\eta$ a character of $\FF^{\times}$ such that $\chi=\eta \circ \Nm_{\EE/\FF}$. If $\chi_{|\FF^{\times}}=\mathbf{1}$, then $\eta^2=\mathbf{1}$. We can take $\Psi_{\FF} = (\eta\nu^{-1/2} \oplus \eta\nu^{1/2},N)$ and $\Psi_{F|W_{\EE}} \sim \Psi$. If $\chi_{|\FF^{\times}}=\omegaEF$, $\eta^2=\omegaEF \neq \mathbf{1}$. If $(\eta\nu^{-1/2}  \oplus \eta^{-1}\nu^{1/2} ,M)$ is a Weil-Deligne representation then $M=0$ and this is not a lift of $\Psi$.
\end{proof}

\begin{lem}
    \label{lemLiftWDPrincSeries}
    Let  $\chi$ be a character of $\EE^{\times}$. Let $\Psi \in [\WDRep_{\Fl}(W_{\EE},\SL_2)]$ defined by $\Psi = [\chi\nu^{-1/2} \oplus \chi^{-1}\nu^{1/2},N]$, with $ N=0$. Then there exists $\Psi_{F} \in \WDRep_{\Fl}(W_{\FF},\SL_2)$ such that $\Psi_{F|W_{\EE}} \sim \Psi$ if and only if $\chi=\chi^{\sigma}$ or $\chi\neq\chi^{\sigmaEF}$ and $\chi_{|\FF^{\times}}=\omegaEF \nu^{1/2}_{|\FF^{\times}}$.
\end{lem}

\begin{proof}
    If $\chi=\chi^{\sigma}$, take $\eta$ a character of $\FF^{\times}$ such that $\chi=\eta \circ Nm_{\EE/\FF}$. Then $\Psi_F = [\eta\nu^{-1/2} \oplus \eta^{-1}\nu^{1/2},N]$ is a lift. And if $\chi\neq\chi^{\sigmaEF}$, Lemma \ref{lemLiftSemiSimple} tells us that if there is a lift then $\chi_{|\FF^{\times}}=\omegaEF \nu^{1/2}_{|\FF^{\times}}$. In this case, we can take $\Psi_F = [\Ind_{W_{\EE}}^{W_{\FF}}(\chi \nu^{-1/2}),N]$.
\end{proof}

\subsection{A modulo \texorpdfstring{$\ell$}{l} Prasad conjecture for \texorpdfstring{$\PGL_2$}{PGL2}}

Now we can gather together all the results of the previous sections to prove a ``modified'' Prasad conjecture for $\PGL_2$.

\bigskip

Let $PV$ be the correspondence induced by the Vignéras correspondence
\[
    PV : \Irr_{\Fl}(\PGL_2(\EE)) \to \Nilp_{\Fl}(W_{\EE},\SL_2)
\]

Let us start with the supercuspidal representations.

\begin{prop}
    \label{proPrasadSupercusp}
    Let $\pi$ be an irreducible supercuspidal representation of $\PGL_2(\EE)$ over $\Fl$. Then $\pi$ is $\omegaEF$-distinguished if and only if its Langlands parameter $PV(\pi)$ can be lifted to $W_\FF$.
\end{prop}

\begin{proof}
    Let $\pi$ be an irreducible supercuspidal representation of $\PGL_2(\EE)$ over $\Fl$.
    Let $\varphi:=PV(\pi)$ be the Langlands parameter of $W_\EE$ associated to $\pi$.
    By Proposition \ref{proLiftDistComplete} $\pi$ is $\omegaEF$-distinguished if and only if there exists $\tilde{\pi}$ a $\Ql$-lift of $\pi$ which is supercuspidal and $\tilde{\omega}_{\EE/\FF}$-distinguished. By Theorem \ref{Prasad:PGL}, this happens if and only if  the Langlands parameter $\tilde{\varphi}$ of $\tilde{\pi}$, can be extended to $W_\FF$. From the definition of the modulo $\ell$ Langlands correspondence, $\tilde{\varphi}$ is a $\Ql$-lift of $\varphi$. Thus $\tilde{\varphi}$ can be extended to $W_\FF$ if and only if $\varphi$ can be extended to $W_\FF$.
\end{proof}

Now let us examine the irreducible generic representations of $\PGL_2(\EE)$. Let $\pi$ be such a representation and denote by $PV(\pi)=(\Psi,N)$ its Langlands parameter. We can classify the irreducible generic representations as follows:
\begin{enumerate}
    \item $\pi$ is supercuspidal. In this case, $\Psi$ is irreducible and $N=0$.
    \item $\pi$ is an irreducible principal series. Here $\pi = \pi(\chi \nu^{-1/2},\chi^{-1}\nu^{1/2})$ with $\chi^2\neq \mathbf{1}$. We have $\Psi=\chi\nu^{-1/2} \oplus \chi^{-1}\nu^{1/2}$ and $ N=0$
    \item Suppose that $\pi$ is the unique generic subquotient of a reducible principal series. Let $\chi$ be a quadratic character of $\EE^{\times}$ such that $\pi$ is a subquotient of $\pi(\chi \nu^{-1/2},\chi\nu^{1/2})$.
          \begin{enumerate}
              \item If $q_\EE \not\equiv 1 \pmod{\ell}$, then $\pi = \St_\chi$. In this case, $\Psi=\chi\nu^{-1/2} \oplus \chi \nu^{1/2}$ and $ N=\begin{pmatrix}
                            0 & 1 \\
                            0 & 0
                        \end{pmatrix}$.
              \item If $q_\EE \equiv -1 \pmod{\ell}$, this time $\pi=\Spe_{\chi}$, $\Psi=\chi\nu^{-1/2} \oplus \chi\nu^{1/2}$ and $ N=0$.
          \end{enumerate}

\end{enumerate}

We can summarize the previous results of this article to prove the Prasad conjecture in the modular case.

\begin{thm}
    \label{thmPrasadModular}
    Let $\pi$ be an irreducible generic representation of $\PGL_2(\EE)$ over $\Fl$. Then $\pi$ is $\omegaEF$-distinguished if and only if there exists $\Psi_{F} \in \WDRep_{\Fl}(W_{\FF},\SL_2)$ such that $\Psi_{F|W_{\EE}} \sim P\circ PV(\pi)$.
\end{thm}

\begin{proof}
    For supercuspidal {representations}, the result follows from {Proposition} \ref{proPrasadSupercusp}. For irreducible principal series {representations} it follows from Lemmas \ref{lemPrincSer} and \ref{lemLiftWDPrincSeries}. And for the Steinberg representations or the special representations, it follows from Theorem \ref{thmCuspDist} and Lemmas \ref{lemLiftWDSp}, \ref{lemLiftWDSt} and \ref{lemLiftWDBanal} (depending on the order of $q_\EE$ modulo $\ell$).
\end{proof}

\begin{rem}

    When $q_\EE\equiv -1 \pmod{\ell}$, $q_\FF\equiv -1 \pmod{\ell}$ and $\chi$ is a quadratic character of $\EE^{\times}$ such that $\chi_{|\FF^\times}=\omegaEF \nu_{|\FF^\times}^{1/2}$. We have proved that $P \circ PV(\Spe_{\chi})$ admits a lift to $W_\FF$ which is $\Psi_F=\left(\Ind_{W_{\EE}}^{W_{\FF}}(\chi \nu^{-1/2}),M\right)$ with $ M=\begin{pmatrix}
            0  & 1 \\
            -1 & 0
        \end{pmatrix}$. The semisimple part of $\Psi_F$ is irreducible and $M$ is non-zero, and so this lift is not the Langlands parameter of any representation of $\PGL_2(\FF)$ (nor it is in the image of $P$).

\end{rem}

\section{The \texorpdfstring{$\SL_2(F)$}{SL2(F)}-distinguished representations}

\label{secSL2}

In this section, assuming $\ell\neq2$, we classify all the representations of $\SL_2(\EE)$ distinguished by $\SL_2(F)$. For supercuspidal representations, we will use the restriction method of \cite{anandavardhanan2003distinguished}. Similar to the case $\GL_2(\EE)$, we will deal with principal series representations using Mackey theory.

\subsection{Modulo $\ell$ representations of $\SL_2$}
We start by recalling some general facts about modulo $\ell$ representations of $\SL_2$.

\bigskip

Recall that $E/F$ is a quadratic extension of locally compact non-archimedean local
fields of characteristic different from 2. Let $\mathfrak{o}_E$ be the ring of integers of $E$, and $\mathfrak{p}_{E}$ be the maximal ideal of $\mathfrak{o}_E$.  Let $\pi$ be an irreducible {cuspidal} $\Fl$-representation of $\GL_n(E)$. Thanks to \cite[Prop. 2.35]{C1}, we have that the restricted representation $\pi\vert_{\SL_n(E)}$ is semisimple with finite length {and multiplicity-free}. Denote by $\lg(\pi)$ the length of $\pi\vert_{\SL_n(E)}$. Let
\[
    Y(\pi)=\{\chi: \Fl\text{-character of } E^{\times},\pi\otimes\chi\circ\det\cong\pi\}.
\]
We call $Y(\pi)$ \emph{the twist isomorphism set of $\pi$}. By Corollary 3.8 of \cite{C1}, the cardinality $\vert Y(\pi)\vert$ is an integer prime to $\ell$, and by Proposition  2.37 of \cite{C1}, $\pi\vert_{\SL_n(F)}$ is multiplicity-free. Hence we deduce from part 3 of Corollary 3.8 of \cite{C1} that
\begin{equation}
    \label{equation lg(pi)}
    \vert Y(\pi) \vert=\lg(\pi)_{\ell'},
\end{equation}
where for any positive integer $m$, we denote by $m_{\ell'}$ the largest divisor of $m$ which is coprime to $\ell$.

\begin{lem}
    \label{lema001}
    When $n=2$, we have
    \[
        \vert Y(\pi) \vert=\lg(\pi).
    \]
\end{lem}

\begin{proof}
    By \eqref{equation lg(pi)}, it is sufficient to prove that the length $\lg(\pi)$ is coprime to $\ell$. Let $Z_{E^{\times}}$ be the center of $\GL_2(E)$. The length of $\pi\vert_{Z_{E^{\times}}\SL_2(E)}$ is equal to the length $\lg(\pi)$, which is a divisor of the index $[ \GL_2(E) : Z_{E^{\times}}\SL_2(E)]$, and the latter is equal to $[ E^{\times}: E^{\times 2}]$, where $E^{\times 2}$ consists of the elements of the form $x^2$ for $x\in E^{\times}$. By Corollary 5.8 of \cite{Ne}, when the characteristic of $E$ is different from $2$, the index $[ E^{\times}: E^{\times 2}]$ is equal to $2^{2+a}$ where $a$ is given by $2\mathfrak{o}_E=\mathfrak{p}_E^a$. Hence $\lg(\pi)$ is a power of $2$, and we have the desired identity under our assumption $\ell\neq 2$.
    %When $p\neq 2$, by Hensel's lemma we have $\vert E^{\times}:E^{\times2} \vert$ is equal to $4$. When $p=2$, by the binomial expansion formula, we know that when $N$ is big enough, an element in $1+\mathfrak{p}_{E}^{N}\mathfrak{o}_{E}$ is a square of an element in $\mathfrak{o}_E^{\times}$, which means $\mathfrak{o}_{E}^{\times}\slash\mathfrak{o}_{E}^{\times 2}$ is finite. 
\end{proof}

\begin{defi}
    Define $\GL_2^{+}(E)$ to be a subgroup of $\GL_2(E)$, consisting of matrices whose determinant belongs to $F^{\times}E^{\times2}$, where $E^{\times 2}$ consists of the elements of the form $x^2$ for $x\in E^{\times}$. We have $\GL_2^{+}(E)=Z_{E^{\times}}\SL_2(E)\GL_2(F)$ where $Z_{E^{\times}}$ denotes the center of $\GL_2(E)$.
\end{defi}

\begin{lem}
    \label{corlgY+}
    Let $\lg_{+}(\pi)$ be the length of $\pi\vert_{\GL_2^{+}(E)}$, and
    \[
        Y_+(\pi)=\{\chi:\Fl\text{-character of } E^{\times}, \pi\otimes\chi\circ\det\cong\pi,\ \chi\vert_{F^{\times}}=1\}.
    \]
    Then we have
    \[
        \lg_{+}(\pi)=\vert Y_+(\pi)\vert.
    \]
\end{lem}

\begin{proof}
    %From the proof of Lemma \ref{lema001} the quotient group $\mathrm{GL}_2(E)\slash\GL_2^{+}(E)$ is finite abelian, and the order is a power of $2$ which is prime to $\ell$. We deduce that the dual group $(\mathrm{GL}_2(E)\slash\GL_2^{+}(E))^{\wedge}$ is isomorphic to $Y_+(\pi)$ as abelian groups. %Then the method of Corollary 3.8 of \cite{C1} can be applied. In particular, Corollary 3.8 of \cite{C1} stud{ies} the length of $\pi\vert_{\SL_2(E)}$. 
    %Then we consider the characters of $\GL_2(E)\slash\SL_2(E)\cong E^{\times}$. 
    %By \cite[Corollary 3.8]{C1} and the fact that $\ell$ is odd, the length of $\pi\vert_{\SL_2(E)}$ is equal to the cardinality of the twist isomorphism set $Y(\pi)$. To study the length of $\pi\vert_{\GL_2^{+}(E)}$, we need to consider the characters in $(\mathrm{GL}_2(E)\slash\GL_2^{+}(E))^{\wedge}$. Hence by  \cite[Corollary 3.8]{C1}, we obtain the desired equation.

    Since the direct components of $\pi\vert_{\GL_2^+(E)}$ are $\GL_2(E)$-conjugate, they share the same length after restricted to $\SL_2(E)$. We have that $\lg_+(\pi)$ divides $\lg(\pi)$,  which is coprime to $\ell$ by Lemma \ref{lema001}. Hence the Clifford theory in $\ell$-modular setting (see \cite[Cor. 3.8]{C1}) is the same as the complex setting in our case. Meanwhile, since $\pi\vert_{\SL_2(E)}$ is multiplicity-free, the restriction $\pi\vert_{\GL_2^+(E)}$ is multiplicity-free. Then the Clifford theory gives the desired equation $\lg_+(\pi)=\vert Y_+(\pi)\vert$.

    %In \cite[Corollary 3.8]{C1}, we study the $\lg(\pi)_{\ell}$ so we consider the charactersTo study the length of $\pi\vert_{\GL_2^{+}(E)}$, we need to consider the characters in $(\mathrm{GL}_2(E)\slash\GL_2^{+}(E))^{\wedge}$. By a samilar proof as \cite[Corollary 3.8]{C1}, we have
    %$$\lg_{+}(\pi)_{\ell'}=\vert Y_+(\pi)\vert,$$
    %where $\lg_+(\pi)_{\ell'}$ is the $\ell$-prime part of $\lg_+(\pi)$. 

\end{proof}

To compute the length $\lg(\pi)$ we will need to use the local Langlands correspondence.

\subsection{Local Langlands for supercuspidal representations of \texorpdfstring{$\SL_2$}{SL2}}

In this section, we use the local Langlands correspondence for $\GL_2$ to define a correspondence modulo $\ell$ for supercuspidal representations of $\SL_2$. As in the complex case, for $\SL_2$, this correspondence is not a bijection,  we give a description of the L-packet.

\bigskip

Let $\tau$ be a supercuspidal $\Fl$-representation of $\SL_2(E)$. Let $\pi$ be a \emph{lift} to $\GL_2(E)$ i.e., $\tau\subset \pi|_{\SL_2(E)}$. To $\pi$ we associate by the local Langlands correspondence of Vignéras its Langlands parameter $\varphi_{\pi}: W_{\EE} \to \GL_2(\Fl)$. Let $\gamma$ be the projection $\gamma : \GL_2(\Fl) \to \PGL_2(\Fl)$. Then we define $\varphi_\tau : W_{\EE} \to \PGL_2(\Fl)$ by $\varphi_\tau:= \gamma \circ \varphi_{\pi}$. The parameter $\varphi_\tau$ does not depend on the choice of the lift $\pi$ since two lifts differ by a character and so are their Langlands parameters.

\bigskip

Let  $\varphi : W_{\EE} \to \PGL_2(\Fl)$ be the parameter of a representation of $\SL_2(E)$. Denote by $S_{\varphi}:=C_{\PGL_2(\Fl)}(\varphi(W_\EE))$ the centralizer in $\PGL_2(\Fl)$ of the image of $\varphi$.

\begin{prop}
    \label{proLengthCentralizer}
    Let $\tau$ (resp. $\pi$) be a supercuspidal $\Fl$-representation of $\SL_2(E)$ (resp. $\GL_2(E)$) and $\tau\subset \pi|_{\SL_2(F)}$. Then we have an isomorphism
    \[
        S_{\varphi_\tau} \simeq Y(\pi)
    \]
\end{prop}

\begin{proof}
    We follow the strategy of \cite[Thm. 4.3]{gelbart}. To simplify the notation here, we will simply denote $\varphi_\tau$ by $\varphi$. From the definition of $\varphi$, we have that $\varphi=\gamma \circ \varphi_\pi$. Let $s\in S_{\varphi}$ and $\tilde{s}\in \GL_2(\Fl)$ such that $\gamma(\tilde{s})=s$. We define a function $\chi_s:W_{\EE} \to \GL_2(\Fl)$ by
    \[
        \chi_s(w):=\tilde{s}\varphi_\pi(w) \tilde{s}^{-1}\varphi_\pi(w)^{-1}, \text{ for } w\in W_{\FF}.
    \]
    This definition is independent of the choice of $\tilde{s}$. Moreover, since $\gamma(\chi_s(w))=1$, $\chi_s(w)$ is a scalar times the identity. We will denote this scalar again $\chi_s(w)$. Hence, we have

    \[
        \tilde{s}\varphi_\pi(w) \tilde{s}^{-1} = \chi_s(w)\varphi_\pi(w).
    \]

    Let $w_1,w_2 \in W_{\EE}$. Then
    \begin{align*}
        \chi_s(w_1w_2)\varphi_\pi(w_1w_2) & =\tilde{s}\varphi_\pi(w_1w_2) \tilde{s}^{-1} = \tilde{s}\varphi_\pi(w_1)\varphi_\pi(w_2) \tilde{s}^{-1} \\
                                          & =\tilde{s}\varphi_\pi(w_1)\tilde{s}^{-1}\tilde{s}\varphi_\pi(w_2) \tilde{s}^{-1}                        \\
                                          & =\chi_s(w_1)\varphi_\pi(w_1) \chi_s(w_2)\varphi_\pi(w_2)                                                \\
                                          & =\chi_s(w_1)\chi_s(w_2)\varphi_\pi(w_1)\varphi_\pi(w_2) = \chi_s(w_1)\chi_s(w_2)\varphi_\pi(w_1w_2)
    \end{align*}
    Thus $\chi_s(w_1w_2)=\chi_s(w_1)\chi_s(w_2)$ and $\chi_s$ is a character.

    \medskip

    Since $\varphi_\pi \simeq \chi_s\varphi_\pi$ we have $\pi \simeq \pi\otimes (\chi_s \circ \det)$ and $\chi_s \in Y(\pi)$. This defines a morphism $S_{\varphi} \to Y(\pi)$, $s \mapsto \chi_s$.

    \medskip

    This morphism is surjective. Indeed, if $\omega \in Y(\pi)$ then $\pi \simeq \pi\otimes (\omega \circ \det)$ and thus $\varphi_\pi \simeq \omega \varphi_\pi$. If $\tilde{A}$ implements the equivalence then $\tilde{A} \varphi_\pi(w) \tilde{A}^{-1} = \omega(w) \varphi_\pi(w)$. Let $A:=\gamma(\tilde{A})$. Then $A \in S_{\varphi}$ and $\omega = \chi_A$.

    \medskip

    We are left to prove that $S_{\varphi} \to Y(\pi)$ is injective. Let $s\in S_{\varphi}$ such that $\chi_s=1$. This implies that $\tilde{s}$ centralizes the image of $\varphi_\pi$. Since $\varphi_\pi$ is irreducible, Schur's lemma tells us that $\tilde{s}$ is a scalar. Hence $s=1$ and we have the injectivity.
\end{proof}

\subsection{Explicit computation of the length}

By Proposition \ref{proLengthCentralizer}, we compute the length $\lg(\pi)$ by considering the cardinality of $S_{\varphi_\tau}$. The method in \cite{Shel} can be generalised to the case when $\ell$ is positive, based on which we also obtain the existence of good lift.

\begin{defi}
    Let $\tau$ be an irreducible supercuspidal $\Fl$-representation of $\mathrm{SL}_2(E)$, and $\tilde{\tau}$ an irreducible supercuspidal $\Ql$-representation of $\mathrm{SL}_2(E)$, which is $\ell$-integral. We say
    \begin{itemize}
        \item $\tilde{\tau}$ is a $\Ql$-lift of $\tau$, if  $\tau$ is a subquotient of the reduction modulo $\ell$ of $\tilde{\tau}$;
        \item $\tilde{\tau}$ is a good $\Ql$-lift of $\tau$, if the reduction modulo $\ell$ of $\tilde{\tau}$ is irreducible and isomorphic to $\tau$.
    \end{itemize}
\end{defi}

For the case of $\GL_n(E)$, an irreducible supercuspidal $\Fl$-representation of $\GL_n(E)$ always has a $\Ql$-lift, and every $\Ql$-lift is a good $\Ql$-lift. The latter property is not true for $\SL_2(E)$. However, we will show the existence of a good $\Ql$-lift of an irreducible supercuspidal $\Fl$-representation of $\SL_2(E)$.

\bigskip

Let $\pi$ be an irreducible supercuspidal $\Fl$-representation of $\GL_2(E)$ with Langlands parameter $\varphi_\pi$, and $\tau\subset\pi\vert_{\SL_2(E)}$. We say $\varphi_{\pi}$ is \emph{dihedral}  if the image of $W_E$ is of the form $\Ind_{W_K}^{W_{\EE}}\theta$ in $\mathrm{GL}_2(\Fl)$ where $K/E$ is a quadratic field extension, $W_K$ is the Weil group of $K$ and $\theta$ is a character of $W_K$ which is not invariant under the $\Gal(K/E)$-action. We say $\varphi_\pi$ is \emph{tetrahedral} (resp. \emph{octahedral}) if the image of $W_E$ under the map $W_E\to\GL_2(\Fl)\to \PGL_2(\Fl)$
is the alternating group $A_4$ (resp. the symmetric group $S_4$). (See \cite[Section 42]{BH} for more details.)
\begin{prop}
    \label{cprop 1.3}
    Let $\pi$ be an irreducible supercuspidal $\Fl$-representation of $\GL_2(E)$.
    \begin{enumerate}
        \item When $\varphi_{\pi}$ is dihedral, the length $\lg(\pi)$ is 2 or 4;
        \item When $\varphi_{\pi}$ is not dihedral, then it must be tetrahedral or octahedral, and the length $\lg(\pi)$ is 1.
    \end{enumerate}
\end{prop}

\begin{proof}

    %Now we write $\varphi_{\pi}$ as $\mathrm{ind}_{W_K}^{W_{E}}\chi$, and let $\eta$ be a non-trivial $\Fl$-quasicharacter of $E^{\times}$ such that $\eta\in Y(\pi)$. We deduce that either $\eta=\omega_{K/E}$ or $\eta\vert_{W_K}\cong\chi^{-1}\chi^{s}$, where $s$ is the non-trivial element in $\mathrm{Gal}(K\slash E)$.

    %Suppose that $\eta\neq\omega_{K/E}$. Since $\chi^{-1}\chi^{s}=\eta\circ\Nm_{K\slash E}$ is quadratic, i.e. $\chi^2=(\chi^2)^s$, there exists another $\Fl$-quasicharacter $\eta_{E}$ of $E^{\times}$ such that $\chi^2\cong\eta_{E}\circ \Nm_{K\slash E}$, which implies that $\chi|_{E^\times}$ is isomorphic either to $\eta_E\otimes\eta$ or to $\eta_E\otimes\eta\otimes\omega_{K\slash E}$. We conclude that $Y(\pi)$ consists of $\chi|_{E^\times}\otimes\eta_E^{-1},\chi|_{E^\times}\otimes\eta_E^{-1}\otimes\omega_{K\slash E},\omega_{K\slash E}$ and the trivial character. In other words, the length $\lg(\pi)$ is either $2$ or $4$.
    Let $\tilde{\pi}$ be a $\Ql$-lift of $\pi$. Then the length $\lg(\tilde{\pi})$ divides the length $\lg(\pi)$. Suppose $\vert Y(\pi)\vert\neq 1$. Then there exists an $\Fl$-character $\chi$ of $E^{\times}$ such that $\pi\otimes\chi\circ\det\cong\pi$. By Local Class Field Theory, we identify $\chi$ with an $\Fl$-character of $W_E$ such that $\varphi_{\pi\otimes\chi\circ\det}\cong\chi\varphi_{\pi}\cong\varphi_{\pi}$. Considering the central character, we have that $\chi$ has order $2$. Then there exists a quadratic field extension $K\slash E$ such that $\ker(\chi)=W_K$, which, by Clifford theory implies that $\varphi_{\pi}\vert_{W_K}$ is a direct sum of two characters of $W_K$. Hence $\varphi_\pi=\Ind_{W_K}^{W_{\EE}}\theta$ where $W_K$ is the Weil group of $K$ and $\theta$ is an $\Fl$-character of $W_K$.
    %Hence $\varphi_{\pi}$ factor through $\mathrm{ind}_{K^{\times}}^{W_{K \slash E}}\theta$, where $W_{K \slash E}$ is the Weil group of $K \slash E$, and $\theta$ is a $k$-quasicharacter of $E^{\times}$. 
    Let $s$ be the non-trivial element in $\mathrm{Gal}(K\slash E)$. By a same computation as in \cite[\S11, part (ii)]{Shel}, we deduce that if
    %the image of $W_E$ is a dihedral subgroup. To be more precise, when 
    $\theta\slash\theta^{s}$ has order two, then $\vert S_{\varphi_{\tau}}\vert=4$. When $(\theta\slash\theta^{s})^2\neq\mathbf{1}$, then $\vert S_{\varphi_{\tau}}\vert=2$. When  $\theta\cong\theta^s$ then $\varphi_{\pi}$ is reducible which is not an $L$-parameter of an irreducible supercuspidal representation of $\GL_{2}(E)$. Then we obtain the result by applying Lemma \ref{lema001} and Proposition \ref{proLengthCentralizer}. On the other hand, let $\mu_{K\slash E}$ be the unique non-trivial character on $\mathrm{Gal}(K\slash E)$, then we have $\mu_{K\slash E}\otimes\varphi_{\pi}\cong\varphi_{\pi}$ hence $\vert Y(\pi)\vert\neq 1$. In other words, we show that $\varphi_{\pi}$ is dihedral if and only if $\vert Y(\pi)\vert\neq 1$.

    Suppose that $\varphi_{\pi}$ is not dihedral, then $\vert Y({\pi})\vert=1$, which implies that $p=2$ by \cite[Section 42]{BH}. In this case, the image of a $\Ql$-lift $\tilde{\pi}$ under the projection from $\GL_2(\Ql)$ to $\PGL_2(\Ql)$ is either isomorphic to $S_4$ or $A_4$. For the second case, since any two subgroups of $\PGL_2(\Ql)$ being isomorphic to $A_4$ are conjugate to each other, we can choose $\varphi_{\tilde{\pi}}$ such that its image in $\PGL_2(\Ql)$ will be $N\rtimes C$, where
    \[
        N=\left\{\begin{pmatrix} \pm 1&0\\0&1\end{pmatrix},\begin{pmatrix}0&\pm 1\\ \pm 1&0\end{pmatrix} \right\}, C=\left\{ I,\begin{pmatrix} 1&\sqrt{-1}\\ 1 &-\sqrt{-1}\end{pmatrix},\begin{pmatrix} 1&1\\-\sqrt{-1}&\sqrt{-1} \end{pmatrix}\right\}.
    \]
    Since $\ell\neq 2$, after reduction modulo $\ell$ the image of $\varphi_{\pi}$ is isomorphic to $N\rtimes C\cong S_4$ as well. For the case of $S_4$, it is isomorphic to $\left\langle N\rtimes C,\begin{pmatrix} \sqrt{-1}&0\\ 0&1 \end{pmatrix} \right\rangle$. We repeat the similar argument to the case of $A_4$. This finishes the proof.
\end{proof}

\begin{prop}
    \label{propQoodLifting}
    Let $\pi$ be an irreducible supercuspidal $\Fl$-representation of $\GL_2(E)$ and $\tilde{\pi}$ a $\Ql$-lift of $\pi$. Let $\tau$ be an irreducible component of $\pi\vert_{\SL_2(E)}$ and $\tilde{\tau}$ an irreducible component of $\tilde{\pi}\vert_{\SL_2(E)}$.
    \begin{enumerate}
        \item Suppose that $\varphi_{\pi}$ is tetrahedral or octahedral. Then $\tilde{\pi}\vert_{\SL_2(E)}\cong\tilde{\tau}$ is a good $\Ql$-lift of $\tau$.
        \item Suppose that $\varphi_{\pi}$ is dihedral.
              \begin{enumerate}
                  \item If the cardinality of $S_{\varphi_{\tilde{\tau}}}$ is $4$, then the reduction modulo $\ell$ of $\tilde{\tau}$ is irreducible. In particular, there exists an irreducible component $\tilde{\tau}'\subset\pi\vert_{\SL_2(E)}$ which is a good $\Ql$-lift of $\tau$.

                  \item If the cardinality of $S_{\varphi_{\tilde{\tau}}}$ is $2$, then the reduction modulo $\ell$ of $\tilde{\tau}$ may be reducible. If it is irreducible, there exists an irreducible component $\tilde{\tau}'\subset\pi\vert_{\SL_2(E)}$ which is a good $\Ql$-lift of $\tau$. If it is reducible, then there exists another $\Ql$-lift $\tilde{\pi}'$ of $\pi$, such that the cardinality of $S_{\varphi_{\tilde{\pi}'}}$ is $4$, and there exists an irreducible component $\tilde{\tau}'\subset\tilde{\pi}'\vert_{\SL_2(F)}$ which is a good $\Ql$-lift of $\tau$.
              \end{enumerate}
    \end{enumerate}
\end{prop}

\begin{proof}
    \begin{enumerate}
        \item
              We have that $\varphi_{\tilde{\pi}}$ is tetrahedral or octahedral, hence the reduction modulo $\ell$ of $\tilde{\pi}\vert_{\SL_2(E)}$ is irreducible and isomorphic to $\tau$ by Proposition \ref{cprop 1.3}(2).
        \item \begin{enumerate}
                  \item
                        We first assume that the cardinality of $S_{\varphi_{\tilde{\tau}}}$ is $4$. We switch the order of restriction to $\SL_2(E)$ and reduction modulo $\ell$. Then by Proposition \ref{cprop 1.3}(1), we have $\lg(\pi)=\lg(\tilde{\pi})$. Hence every irreducible component of $\tilde{\pi}\vert_{\SL_2}$ is $\ell$-integral and its reduction modulo $\ell$ is irreducible. The unicity of Jordan-H\"older components implies the existence of a good $\Ql$-lift $\tilde{\tau}$.
                  \item
                        Now we assume that the cardinality of $S_{\varphi_{\tilde{\tau}}}$ is $2$. Then there exists a field extension $K/ E$ and an $\ell$-integral $\Ql$-quasicharacter $\tilde{\theta}$ of $K^{\times}$, such that $\varphi_{\tilde{\pi}}\cong\mathrm{Ind}_{W_K}^{W_{E}}\tilde{\theta}$, and $(\tilde{\theta}^s/{\tilde{\theta}})^2\neq \mathbf{1}$, where $\Gal(K/E)=\langle s\rangle$. Let $\theta$ be the reduction modulo $\ell$ of $\tilde{\theta}$. Then $\varphi_{\pi}\cong\mathrm{Ind}_{W_K}^{W_{E}}\theta$. Suppose $(\theta^s/{\theta})^2\neq \mathbf{1}$. Then the cardinality of $S_{\varphi_{\tau}}$ is $2$ as well. In this case, we apply the same argument as for case (a) in part $(2)$, and deduce the existence of a good lift. If $(\theta^s/{\theta})^2 =\mathbf{1}$, the cardinality of $S_{\varphi_{\tau}}$ is $4$, which implies that the length of the reduction modulo $\ell$ of each irreducible component in $\tilde{\pi}\vert_{\SL_2}$ is $2$, hence none of them is a good lift of $\tau$. Now fix a group embedding from $\Fl^{\times}$ to $\Zl^{\times}$ by sending an element of $\Fl^{\times}$ to its Teichmuller representative, which gives a natural $\Ql$-lift of $\theta$, denoted  by $\tilde{\theta}_{0}$. %In fact, the value of $\tilde{\theta}_{\ell'}$ have finite ordre and prime to $\ell$.
                        We deduce that $(\tilde{\theta}_{0}^s/{\tilde{\theta}}_{0})^2=\mathbf{1}$. Let $\tilde{\pi}'$ be the irreducible supercuspidal $\Ql$-representation of $\GL_2(E)$ corresponding to $\varphi_{0}=\mathrm{Ind}_{W_K}^{W_{E}}\tilde{\theta}_{0}$, which is $\ell$-integral with reduction modulo $\ell$ isomorphic to $\pi$, and the cardinality of $Y(\tilde{\pi}')$ is $4$. We apply the argument for case (a) in part $(2)$ to finish the proof.
              \end{enumerate}
    \end{enumerate}
\end{proof}

\subsection{Representations of $\GL_2(\EE)$ distinguished by $\SL_2(\FF)$}
In the remaining part of this paper, we follow the restriction method of \cite{anandavardhanan2003distinguished} since the main strategy of \cite{anandavardhanan2003distinguished} works for $\Fl$-representations as well. However in the modular setting, some modifications are still needed in the proofs. For convenience, we state the results which are required for further use.

\bigskip

For an irreducible $\Fl$-representation $\pi$ of $\GL_2(E)$, we denote by $X(\pi)$ the following set
\[
    X(\pi)=\{ \chi: \Fl\text{-character of } F^{\times}, \pi \text{ is } (\GL_2(F),\chi)\text{-distinguished}\}
\]

\begin{prop} \cite[Prop. 4.1]{anandavardhanan2003distinguished}
    \label{prop4.1AP}
    Let $\pi$ be an irreducible $\Fl$-representation of $\GL_2(E)$. Then
    \[
        \dim\mathrm{Hom}_{\SL_2(F)}(\pi,\mathbf{1})= \vert X(\pi)\vert.
    \]
\end{prop}

\begin{proof}
    Assume that $\pi$ is $\mathrm{SL}_2(F)$-distinguished. Let $Z_{F^\times}$ be the center of $\mathrm{GL}_2(F)$ and $Z_{F^\times}'$ be the center of $\mathrm{SL}_2(F)$. The group $Z_{F^{\times}}'$ is the kernel of determinant that maps $Z_{F^{\times}}$ into $F^{\times}$. The central character $\omega_\pi$ of $\pi$ is trivial on $Z_{F^\times}'$ and so there exists a character $\chi_F$ of $F^\times$ such that $\chi_F^2=\omega_{\pi}|_{Z_{F^\times}}$. Since each smooth $\Fl$-character of $F^{\times}$ can be extended to a smooth $\Fl$-character of $E^{\times}$, we obtain that after twisting {by} a smooth $\Fl$-character of $E^{\times}$, the central character $\omega_{\pi}$ of $\pi$ is trivial on $Z_{F^{\times}}$. On the other hand, if $\pi$ is $(\GL_2(F),\chi)$-distinguished for an $\Fl$-character $\chi$ of $F^{\times}$, then $\pi$ is $\mathrm{SL}_2(F)$-distinguished, and $\omega_{\pi}$ is trivial on $Z_{F^{\times}}$ after twisting {by} a smooth $\Fl$-character of $E^{\times}$ as explained above. We obtain {the} fact that if $\omega_{\pi}$ is never trivial on $Z_{F^{\times}}$ after twisting {by} a smooth $\Fl$-character of $E^{\times}$, then $\dim\mathrm{Hom}_{\SL_2(F)}(\pi,1)= \vert X(\pi)\vert=0$.

    Now assume that the central character $\omega_{\pi}$ is trivial {on $Z_{F^\times}$}. As in the proof of Proposition 4.1 of \cite{anandavardhanan2003distinguished}, we consider the $\Fl$-space $\mathrm{Hom}_{\mathrm{SL}_2(F)}(\pi,\mathbf{1})$, which has a $\mathrm{GL}_2(F)$-module structure and $F^{\times}\cdot\mathrm{SL}_2(F)$ acts trivially. Since $F^{\times}\slash F^{\times 2}$ is a finite abelian group whose order is a power of $2$, our assumption that $\ell\neq 2$ implies that $\mathrm{Hom}_{\mathrm{SL}_2(F)}(\pi,\mathbf{1})$ can decompose into a direct sum of $\Fl$-characters of $F^{\times}\slash F^{\times 2}$, hence a direct sum of $\Fl$-characters of $\mathrm{GL}_2(F)$. In particular, from the definition, for an $\Fl$-character $\chi$ of $\mathrm{GL}_2(F)$ which appears in the direct sum above, we have that $\pi$ is $\chi$-distinguished with respect to $\mathrm{GL}_2(F)$. Due to Theorem \ref{thmSecherreDist}(3), the direct sum above is multiplicity-free as $\Fl$-representation of $\mathrm{GL}_2(F)$. Hence we obtain the desired equation.
\end{proof}

\subsection{Supercuspidal case}
\label{section supcusp SL2}

This subsection focuses on the distinction problems for supercuspidal representations.

\bigskip

Denote by $\mathrm{U}$ the subgroup of $\GL_2(E)$ consisting of the upper triangular matrices $\begin{pmatrix}
        1 & x \\0&1
    \end{pmatrix}$. Then $\mathrm{U}\cong E$. Fix an $\Fl$-character $\psi_0$ of $E$, which is non-trivial on $\mathfrak{o}_E$ and is trivial on both $\mathfrak{p}_E$ and $F$. Let $\pi$ be an irreducible infinite-dimensional $\Fl$-representation of $\mathrm{GL}_2(E)$. Then $\pi$ is $\psi_0$-generic, and $\pi$ has a unique Whittaker model $\mathcal{W}(\pi,\psi_0)$.

\begin{lem}
    \label{lem 001}
    Let $\pi_0$ be an irreducible $\Fl$-representation of $\mathrm{GL}_2^+(E)$ { which is of infinite dimension and distinguished by $\SL_2(F)$}. Then it has a Whittaker model with respect to $\psi_0$.
\end{lem}

\begin{proof}
    A {similar} method as in \cite[Lem. 3.1]{anandavardhanan2003distinguished} can be applied here. We repeat the proof for the convenience of the reader. We add some preliminaries at first.

    Let $\pi$ be an irreducible $\Fl$-representation of $\mathrm{GL}_2(E)$ such that $\pi|_{\GL_2^+(E)}\supset\pi_0$. Up to twisting {by} an ${\Fl}$-character of $F^{\times}$ on $\pi_0$, we can assume that $\pi_0$ is distinguished with respect to $\mathrm{GL}_2(F)$. So is $\pi$.
        %Note that each $\Fl$-character of $F^{\times}$ can be extended to $E^{\times}$. Hence we can assume that both $\pi_0$ and $\pi$ are distinguished by $\mathrm{GL}_2(F)$. 
        {Recall that} $\psi_0$ is a fixed additive $\Fl$-character of $E$ which is trivial on $F$, and denote by $\mathcal{W}(\pi,\psi_0)$  the $\psi_0$-Whittaker model of $\pi$. Let $W\in\mathcal{W}(\pi,\psi_0)$. By \cite[\S 8.2]{KuMa} we know that the unique $\mathrm{GL}_2(F)$-invariant linear form can be realized as an integration form
    \[
        f(W)=\int_{F^{\times}}W(\begin{pmatrix} a&0\\0&1\end{pmatrix})d^{\times}a,
    \]
    where $d^\times a$ is a Haar measure on $F^\times$,
    which is non-zero (in \cite{KuMa} this integration form is denoted by $\mathcal{P}_{\pi}$) and  $\mathrm{GL}_2(F)$-invariant. There is a unique irreducible component of the restricted representation $\pi\vert_{\mathrm{GL}_2^+(E)}$ that is $\psi_0$-generic. Let $\pi_1$ be a non $\psi_0$-generic irreducible component of $\mathcal{W}(\pi,\psi_0)\vert_{\mathrm{GL}_2^+(E)}$, and $W\in\mathcal{W}(\pi,\psi_0)$ a function belonging to $\pi_1$. Given $g\in\mathrm{GL}_2^+(E)$, $W(g)$ must be zero. Otherwise it {will induce} a non-trivial morphism from $\pi_1$ to the space of $\psi_0$-Whittaker functions on $\mathrm{GL}_2^+(E)$, which contradicts the assumption that $\pi_1$ is not $\psi_0$-generic. Therefore, the $\mathrm{GL}_2(F)$-invariant linear form can be non-zero only on the $\psi_0$-generic part of $\pi\vert_{\mathrm{GL}^+_2(E)}$, and we conclude that $\pi_0$ is $\psi_0$-generic.
\end{proof}

\begin{lem}
    \label{lem3.2AP}
    Let $\pi$ be an irreducible $\Fl$-representation of $\GL_2^{+}(E)$. The restricted representation  $\pi\vert_{\SL_2(E)}$ is semisimple with finite length {and all } irreducible components are conjugate to each other under the action of $\GL_2(F)$. Hence the dimension $\dim\mathrm{Hom}_{\SL_2(F)}(\tau,\mathbf{1})$ is independent of the choice of irreducible component $\tau$ of $\pi\vert_{\SL_2(E)}$.
\end{lem}
\begin{proof}
    The proof of	\cite[Lem. 3.2]{anandavardhanan2003distinguished} works for $\Fl$-representations as well.
\end{proof}

\begin{prop}
    \label{prop4.2AP}
    Let $\pi$ be an irreducible supercuspidal $\Fl$-representation of $\GL_2(E)$, which is distinguished with respect to $\SL_2(F)$. Then we have
    \[
        \vert X(\pi)\vert=\vert Y_+(\pi) \vert.
    \]
    Assume further that $\pi$ is distinguished with respect to $\GL_2(F)$. Then composition with the norm map $N_{E\slash F}$ induces a bijection from $X(\pi)$ to $Y_+(\pi)$.
\end{prop}

\begin{proof}
    The strategy of \cite[Prop. 4.2]{anandavardhanan2003distinguished} can be applied, and we write the proof here for the convenience of the reader. By Proposition \ref{prop4.1AP}, after twisting $\pi$ by an $\Fl$-character of $E^{\times}$, we assume that $\pi$ is distinguished with respect to $\mathrm{GL}_2(F)$. As in the proof of \cite[Prop. 4.2]{anandavardhanan2003distinguished}, we give a bijection from $X(\pi)$ to $Y_+(\pi)$ {given by the composition} with the norm map.

    Recall that $\Nm_{E\slash F}$ is the norm map from $E^{\times}$ to $F^{\times}$ and $\chi\in X(\pi)$. Define a map
    \[
        \mathcal{N}: X(\pi)\rightarrow Y_+(\pi)
    \]
    via $\chi\mapsto\chi\circ \Nm_{E\slash F}$. Indeed, let $\chi\in X(\pi)$ and $\tilde{\chi}$ be a character of $E^{\times}$ such that $\tilde{\chi}|_{F^\times}=\chi$. By Theorem \ref{thmSecherreDist}, $\pi$ and $\pi \otimes \tilde{\chi}^{-1}$ are $\sigma$-selfdual. Therefore, $\pi \simeq \pi \otimes \chi\circ \Nm_{E\slash F}$ and $\chi\circ \Nm_{E\slash F}\in Y_+(\pi)$.

    Conversely, let $\mu\in Y_+(\pi)$. Since $\mu^{2}=\mathbf{1}$ and $\mu\vert_{F^{\times}}=\mathbf{1}$, by Hilbert's Theorem $90$, it is trivial on the kernel of $\Nm_{E\slash F}$. Hence there exists an $\Fl$-character $\eta$ of $F^{\times}$, such that {$\mu\cong\eta\circ \Nm_{E\slash F}$}. Let $\tilde{\eta}$ be an extension of $\eta$ to $E^{\times}$. Then $\mu\cong\tilde{\eta}\cdot\tilde{\eta}^{\sigma}$. Since $\pi$ is distinguished with respect to $\mathrm{GL}_2(F)$, we have
    \[
        (\pi\otimes\tilde{\eta})^{\vee}\cong(\pi\otimes\tilde{\eta})^{\sigma}.
    \]
    On the other hand, let $Z_{F^{\times}}$ be the center of $\GL_2(F)$. Then $\omega_{\pi\otimes\tilde{\eta}}\vert_{Z_{F^{\times}}}$ is trivial where $\omega_{\pi\otimes\tilde{\eta}}$ is the central character of $\pi\otimes\tilde{\eta}$. By Theorem \ref{thmDichotomy}, $\pi\otimes\tilde{\eta}$ is either distinguished or $\omega_{E\slash F}$-distinguished by $\GL_2(F)$ but not both. {We map $\mu$ to $\eta$ or $\eta\otimes\omega_{E\slash F}$ accordingly.} This gives a map from $Y_{+}(\pi)$ to $X(\pi)$ {which} is the inverse of $\mathcal{N}$. Hence we complete the proof.
\end{proof}

\begin{cor}
    \label{cor4.3AP}
    Let $\pi$ be an irreducible supercuspidal $\Fl$-representation of $\mathrm{GL}_2(E)$, distinguished by $\mathrm{SL}_2(F)$. The number of $\mathrm{SL}_2(F)$-invariant linear functionals on $\pi$ is equal to the length of $\pi\vert_{\mathrm{GL}_2^+(E)}$ {and both are} $\lg_{+}(\pi)$.
\end{cor}
\begin{proof}
    It follows from Lemma \ref{corlgY+}, Proposition \ref{prop4.1AP} and Proposition \ref{prop4.2AP}.
\end{proof}

Recall that $\lg(\pi)$ is the length of $\pi\vert_{\SL_2(E)}$ and $\lg_+(\pi)$ is the length of $\pi\vert_{\GL_2^+(E)}$.
\begin{prop}
    Let $\pi$ be an irreducible supercuspidal $\Fl$-representation of $\mathrm{GL}_2(E)$ such that $\lg(\pi)$ is different from $1$.  Suppose that $\pi$ is distinguished by $\SL_2(F)$.  Then $\lg_{+}(\pi)$ is different from $1$, and the only irreducible component of $\pi\vert_{\GL_2^+(E)}$ which is distinguished by $\SL_2(F)$ is the one that is $\psi_0$-generic (see {the beginning of Section \ref{section supcusp SL2}} for the definition of $\psi_0$).
\end{prop}

\begin{proof}
    The proof of Proposition 4.4 of \cite{anandavardhanan2003distinguished} can be applied, and we write the proof here for completeness. Suppose that $\lg_{+}(\pi)=1$ and denote $\pi\vert_{\mathrm{GL}_2^+(E)}$ by $\pi^+$. {Under this} assumption, we have that $\pi^+\vert_{\mathrm{SL}_2(E)}$ is not irreducible. By Lemma \ref{lem3.2AP} we have that the number of $\mathrm{SL}_2(F)$-invariant linear forms is strictly bigger than one. However by Corollary \ref{cor4.3AP}, there is only one $\mathrm{SL}_2(F)$-invariant linear form, which is a contradiction. Hence $\lg_{+}(\pi)\neq 1$, and the result follows from {Lemma \ref{lem 001} and } Lemma \ref{lem3.2AP}.
\end{proof}

\begin{thm}
    \label{thmDistGeneric}
    Let $\pi$ be an irreducible supercuspidal $\Fl$-representation of $\GL_2(E)$ distinguished by $\SL_2(F)$, and $\pi^+$ the unique irreducible component of $\pi\vert_{\GL_2^{+}(E)}$ that is $\psi_0$-generic. Then $\pi^+$ is distinguished by $\SL_2(F)$. Furthermore, let $\tau$ be an irreducible component of $\pi\vert_{\SL_2(E)}$, distinguished by $\SL_2(F)$. Then $\tau$ is an irreducible component of $\pi^+\vert_{\SL_2(E)}$.
\end{thm}

\begin{proof}
    It follows from  Lemma \ref{lem 001} and Lemma \ref{lem3.2AP} directly.
\end{proof}

\begin{thm}
    \label{thmDistSupercuspSL}
    Let $\pi$ be an irreducible supercuspidal $\Fl$-representation of $\GL_2(E)$ and $\tau$ an irreducible component of $\pi\vert_{\SL_2(E)}$. Suppose $\tau$ is distinguished by $\SL_2(F)$. Then,
    \[
        \mathrm{dim}\mathrm{Hom}{_{\SL_2(F)}}(\tau,\mathbf{1})= \left\{
        \begin{array}{ll}

            1, & \mbox{if } \pi\vert_{\SL_2(E)}\cong\tau;           \\
            1, & \mbox{if } \lg_{+}(\pi)=2 \mbox{ and } \lg(\pi)=4; \\
            2, & \mbox{if } \lg_{+}(\pi)=\lg(\pi)=2;                \\
            4, & \mbox{if } \lg_{+}(\pi)=\lg(\pi)=4.
        \end{array}
        \right.
    \]
    The first case and the last case arises only when $p=2$.
\end{thm}

\begin{proof}
    Let $\pi^+\subset \pi|_{\GL_2^+(E)}$ be irreducible and $\psi_0$-generic.
    Let $\lg(\pi)'$ be the length of $\pi^+\vert_{\SL_2(E)}$. {Since the components of $\pi\vert_{\SL_2(E)}$ are $\GL_2(E)$-conjugate, we deduce} that $\lg(\pi)'=\lg(\pi)\slash \lg_+(\pi)$. By Lemma \ref{lem3.2AP}, Corollary \ref{cor4.3AP} and Theorem \ref{thmDistGeneric}, we summarize that $\mathrm{dim}\mathrm{Hom}{_{\SL_2(F)}}(\pi,\mathbf{1})=\lg_{+}(\pi)$, and
    \[
        \mathrm{dim}\mathrm{Hom}{_{\SL_2(F)}}(\tau,\mathbf{1})=\lg_{+}(\pi)\slash \lg(\pi)'=\lg_{+}(\pi)^2\slash \lg(\pi).
    \]
    We obtain the result by a direct computation.

    When $p$ is odd, $E^{\times}\slash F^{\times}E^{\times 2}$ is of order 2. Since $\lg_{+}(\pi)=\vert Y_+(\pi)\vert$ and the latter is a subset of $\Fl$-characters of $E^{\times}\slash F^{\times}E^{\times 2}$, the case $\lg_+(\pi)=4$ can exist only when $p$ is equal to $2$.

        {If $\pi\vert_{\SL_2(E)}$ is irreducible which implies that $\pi$ is primitive, then $p=2$.}
\end{proof}

\subsection{The principal series representations}

In this subsection, we will use Mackey Theory to prove the following theorem, following \cite{lu2018pacific}.

\begin{thm} \label{prin}
    Let $I(\chi)$ be  a principal series representation of $\SL_2(E)$.
    \begin{enumerate}
        \item
              Let $\tau$ be an irreducible principal series representation of $\SL_2(E)$.
              \begin{enumerate}
                  \item If $\tau=I(\chi)$ with $\chi|_{F^\times}=\mathbf{1}$ and $\chi\neq\mathbf{1}$, then $\dim\Hom_{\SL_2(F)}(\tau,\mathbf{1})=1$.
                  \item If $\tau=I(\chi)$ with $\chi^\sigma=\chi$, then $\dim\Hom_{\SL_2(F)}(\tau,\mathbf{1})=2$.
              \end{enumerate}
        \item Let $\tau$ be an irreducible subrepresentation of $I(\chi)$ distinguished by $\SL_2(F)$ with {$q_F\not\equiv 1 \pmod{\ell}$}.
              \begin{enumerate}
                  \item If $\chi=\nu^{\pm1}$ and {$\ell\nmid q_E^2-1$}, then $\dim\Hom_{\SL_2(F)}(\tau,\mathbf{1})=1$.
                  \item If $\chi^2=\mathbf{1},\chi=\chi_F\circ \Nm_{E/F}\neq\mathbf{1}$ and $\ell\nmid q_E^2-1$, then
                        \[
                            \dim\Hom_{\SL_2(F)}(\tau,\mathbf{1})
                            =\begin{cases}
                                3, & \mbox{ if }\chi_F^2=\mathbf{1};   \\
                                1, & \mbox{ if }\chi_F^2=\omega_{E/F}.
                            \end{cases}
                        \]

                  \item If {$\ell\mid q_E+1$} and $\chi=\nu$, then $\tau$ is the trivial character and $\dim\Hom_{\SL_2(F)}(\tau,\mathbf{1})=1$.
                        In this case,
                        there are two cuspidal  (not supercuspidal) representations inside the Jordan-Holder series of $I(\nu)$ which are not distinguished by $\SL_2(F)$ if $E/F$ is unramified. If $E/F$ is ramified, then only one of two cuspidal representations is distinguished by $\SL_2(F)$ with multiplicity two.
                  \item If $\ell\mid q_E-1$ and $\ell\mid q_F+1$, then
                        \[
                            \dim\Hom_{\SL_2(F)}(\tau,\mathbf{1})=\begin{cases}
                                2, & \mbox{ if }\tau\mbox{ is the Steinberg representation; }                              \\
                                3, & \mbox{ if }\chi=\chi_F\circ \Nm_{E/F}\neq\mathbf{1} \mbox{ with }\chi_F^2=\mathbf{1}; \\
                                1, & \mbox{ if }\chi=\chi_F\circ \Nm_{E/F}\mbox{ with }\chi_F^2=\omega_{E/F}.
                            \end{cases}
                        \]
              \end{enumerate}
    \end{enumerate}
\end{thm}

Let $B(E)$ be the standard Borel subgroup of $\SL_2(E)$ and $B(E)\backslash\SL_2(E)\cong P^1(E)$. Recall that there are two $F$-rational $\GL_2(F)$-orbits in $P^1(E)$ which are $P^1(F)$ and $P^1(E)-P^1(F)$. Moreover, the open orbit $P^1(E)-P^1(F)$ is isomorphic to $E^\times\backslash\GL_2(F)$. There is an exact sequence
\[1 \to E^1\backslash\SL_2(F) \to E^\times\backslash\GL_2(F) \to \Nm_{E/F}E^\times\backslash F^\times \to 1  \]
where $E^1=\{e\in E^\times:\Nm_{E/F}(e)=1 \}$.
Thus  $P^1(E)$ decomposes into $3$ $F$-rational $\SL_2(F)$-orbits: one closed orbit $P^1(F)$ and two open orbits corresponding to $F^\times/\Nm_{E/F}E^\times=\{\pm1\}$.
\begin{lem}\label{lem:prin}
    Let $I(\chi)$ be a principal series representation of $\SL_2(E)$. Then $\Hom_{\SL_2(F)}(I(\chi),\mathbf{1})\neq0$ if and only if either $\chi|_{F^\times}=\mathbf{1}$ or $\chi|_{E^1}=\mathbf{1}$.
\end{lem}
\begin{proof}
    Applying Mackey Theory, one has an exact sequence
    \[0 \to \Hom_{F^\times}(\chi,\mathbf{1}) \to \Hom_{\SL_2(F)}(I(\chi),\mathbf{1}) \to \Hom_{E^1}(\chi,\mathbf{1})\oplus \Hom_{E^1}(\chi^{-1},\mathbf{1}) \to  \Ext^1_{F^\times}(\chi,\mathbf{1})  \]
    of $\Fl$-vector spaces. If $\Hom_{\SL_2(F)}(I(\chi),\mathbf{1})$ is nonzero, then either $\chi|_{F^\times}$ or $\chi|_{E^1}$ is trivial. Conversely, it suffices to show that $\chi|_{F^\times}\neq\mathbf{1}$ and $\chi|_{E^1}=\mathbf{1}$ imply $\Hom_{\SL_2(F)}(I(\chi),\mathbf{1})\neq0$. Note that $\Ext^1_{F^\times}(\chi,\mathbf{1})=0$ if and only if  $\chi|_{F^\times}$ is nontrivial; see \cite[Prop. 8.4]{SecherreDrevon}. By the exact sequence, one have
    \[\dim\Hom_{\SL_2(F)}(I(\chi),\mathbf{1})=2\dim \Hom_{E^1}(\chi,\mathbf{1}) \]
    when $\chi|_{F^\times}\neq\mathbf{1}$. This finishes the proof.
\end{proof}

\begin{lem}
    Suppose that $\ell\nmid q_F-1$. Assume that $\pi$ (resp. $\tau$) is a principal series representation of $\GL_2(E)$ distinguished by $\GL_2(F)$
    (resp. $\SL_2(F)$) and $\tau\subset \pi|_{\SL_2(E)}$. Set $\pi|_{\GL_2^+(E)}=\oplus_i\pi_i$.  Then
    $\tau\subset \pi_0|_{\SL_2(E)}$ if and only if $\pi_0$ is $\psi_0$-generic where $\psi_0$ is a non-degenerate character on $E/F$.
\end{lem}
\begin{proof}
    Suppose that $\pi=\pi(\chi_1,\chi_2)$ with $\chi_1\chi_2$ trivial on $F^\times$. Then the $\GL_2(F)$-invariant linear functional on $\pi$, up to a constant,  is given by
    \[f(W)=\int_{F^\times} W(\begin{pmatrix}
            a & 0 \\0&1
        \end{pmatrix})d^\times a  \]
    where $W$ is the unique right $\GL_2(\mathcal{O}_E)$-invariant $\psi_0$-Whittaker function on $\pi$.
    Since there is a unique $\psi_0$-Whittaker functional on the space of $\pi$, exactly one constituent of the restriction of $\pi$ to $\GL_2^+(E)$ is $\psi_0$-generic. Thus  the $\GL_2(F)$-invariant functional can be non-zero only on the $\psi_0$-generic part of the restriction of $\pi$ to $\GL_2^+(E)$. Therefore $\tau\subset \pi_0|_{\SL_2(E)}$ i.e., $\pi_0$ is $\SL_2(F)$-distinguished if and only if  $\pi_0$ is $\psi_0$-generic.
\end{proof}
\begin{rem} Let $\varpi_F$ denote the uniformizer of the ring $\mathfrak{o}_F$ of integers of $F$.
    Nadir Matringe pointed out that when {$q_F\equiv 1 \pmod{\ell}$}, $\pi$ is distinguished by $\GL_2(F)$ but
    \[
        \int_{F^\times}W(\begin{pmatrix}
            a & 0 \\0&1
        \end{pmatrix})da=\frac{Vol(\mathcal{O}_F^\times)}{(1-\chi_1(\varpi_F))(1-\chi_2(\varpi_F))}=0.
    \]
    Thus, {this method does not work} when we try to determine the multiplicity for the irreducible constituent of $I(\chi)$ with $\chi^2=\mathbf{1}$ when {$q_F\equiv 1 \pmod{\ell}$}.
\end{rem}
Now we are ready to prove Theorem \ref{prin}.
\begin{proof}[Proof of Theorem \ref{prin}]
    \begin{enumerate}
        \item Note that if $\chi\neq\nu^{\pm1}$ and $\chi^2\neq\mathbf{1}$, then $I(\chi)$ is irreducible. Then it follows from Lemma \ref{lem:prin} except that $\chi=\mathbf{1}$ and $\ell\nmid q_E-1$. It is enough to show that $\dim\Hom_{\SL_2(F)}(I(\mathbf{1}),\mathbf{1})=2$ when $\ell\nmid q_E-1$. Note that $\pi(\mathbf{1},\mathbf{1})$ is both $\GL_2(F)$-distinguished and $(\GL_2(F),\omega_{E/F})$-distinguished with multiplicity one. Thus $\dim\Hom_{\SL_2(F)}(I(\mathbf{1}),\mathbf{1})=2$ {by Proposition \ref{prop4.1AP}}.
        \item If $\tau$ is a trivial character, then $\dim\Hom_{\SL_2(F)}(\tau,\mathbf{1})=1$. Let $\tau$ be an infinite-dimensional  subrepresentation of $I(\chi)$.
              \begin{enumerate}
                  \item If $\chi=\nu^{-1}$, then $\tau$ is the trivial character.
                        If $\chi=\nu$ and $\tau$ is the Steinberg representation $\St$, then there exists a short exact sequence
                        \[0 \to \St \to I(\nu) \to \mathbf{1} \to 0  \]
                        of $\SL_2(E)$-representations. Taking the functor $\Hom_{\SL_2(F)}(-,\mathbf{1})$, one has an exact sequence
                        \[0 \to \Hom_{\SL_2(F)}(\mathbf{1},\mathbf{1}) \to \Hom_{\SL_2(F)}(I(\nu),\mathbf{1}) \to \Hom_{\SL_2(F)}(\St,\mathbf{1}) \to \Ext^1_{\SL_2(F)}(\mathbf{1},\mathbf{1}).
                        \]
                        Note that $\Ext^1_{\SL_2(F)}(\mathbf{1},\mathbf{1})=0$. Thanks to Lemma \ref{lem:prin}, we have
                        \[\dim\Hom_{\SL_2(F)}(I(\nu),\mathbf{1})=2. \]
                        Therefore, $\dim\Hom_{\SL_2(F)}(\St,\mathbf{1})=2-1=1$.
                  \item If $\chi=\chi_F\circ \Nm_{E/F}$ with $\chi_F^2=\mathbf{1}$, then the irreducible principal series representation $\pi(\mathbf{1},\chi)$ is distinguished by $\GL_2(F)$. Furthermore,
                        \[\dim\Hom_{\GL_2(F)}(\pi(\mathbf{1},\chi),\chi_F)=1=\dim\Hom_{\GL_2(F)}(\pi(\mathbf{1},\chi),\chi_F\omega_{E/F}). \]
                        There is only one subrepresentation in $I(\chi)$ which is $(\mathrm{U},\psi_0)$-generic, which implies that only one constituent  in $I(\chi)$ is distinguished by $\SL_2(F)$. Thus $\dim\Hom_{\SL_2(F)}(\tau,\mathbf{1})=3$.

                        If $\chi=\chi_F\circ \Nm_{E/F}$ with $\chi_F^2=\omega_{E/F}$, then $\pi(\mathbf{1},\chi)$ is both $(\GL_2(F),\chi_F)$-distinguished and $(\GL_2(F),\chi_F^{-1})$-distinguished. Thus
                        \[\dim\Hom_{\SL_2(F)}(I(\chi),\mathbf{1})=2. \]
                        Note that two constituents in $I(\chi)$ are $(\mathrm{U},\psi_0)$-generic. Thus each one in $I(\chi)$ is distinguished by $\SL_2(F)$ with multiplicity one.

                  \item This case is trivial.

                        If $\ell\mid q_E+1$ and $\tau$ is a cuspidal nonsupercuspidal representation of $\SL_2(E)$,
                        then there is a special representation $\Spe$ of $\GL_2(E)$ such that $\tau\subset \Spe|_{\SL_2(E)}$. If $E/F$ is unramified, then $\Spe$ is neither $\GL_2(F)$-distinguished nor $(\GL_2(F),\omega_{E/F})$-distinguished. (See \cite[Rem. 2.8]{vincent2019supercusp}.) Thus $\dim\Hom_{\SL_2(F)}(\Spe|_{\SL_2(E)},\mathbf{1})=0$ {by Proposition \ref{prop4.1AP}}. If $E/F$ is ramified, then $\Spe$ is not $\GL_2(F)$-distinguished but $(\GL_2(F),\omega_{E/F})$-distinguished. It is also
                        $(\GL_2(F),\nu\vert_{F^{\times}}^{1/2})$-distinguished; see Theorem \ref{thmCuspDist}(2). Due to Proposition \ref{prop4.1AP}, $\dim\Hom_{\SL_2(F)}(\Spe|_{\SL_2(E)},\mathbf{1})=2$. Note that there is only one constituent in $\Spe|_{\SL_2(E)}$ which is $(\mathrm{U},\psi_0)$-generic, denoted by $\tau$. Thus
                        \[\dim\Hom_{\SL_2(F)}(\tau,\mathbf{1})=2. \]
                        This {finishes} the proof.
                  \item If $\ell\mid q_E-1$ and $\ell\mid q_F+1$, then $E/F$ is an unramified field extension. In this case, $I(\mathbf{1})=\mathbf{1}\oplus \St$. Thanks to Theorem \ref{thmCuspDist}(3),
                        \[\dim\Hom_{\SL_2(F)}(\St,\mathbf{1})=2.  \]
                        If $\tau\subset I(\chi_F\circ \Nm_{E/F})$ with $\chi_F^2=\omega_{E/F}$ or $\mathbf{1}$, this  case is similar to case (b).
              \end{enumerate}

    \end{enumerate}

\end{proof}
\begin{rem}
    If $\ell\mid q_F-1$, one can still prove that $\St$ is distinguished by $\SL_2(F)$ with multiplicity two due to Theorem \ref{thmCuspDist} and that the constituents in $I(\chi_F\circ \Nm_{E/F})$ with $\chi_F^2=\omega_{E/F}$ are both distinguished by $\SL_2(F)$ with multiplicity one. But we can not determine the situation for $\tau\subset I(\chi_F\circ \Nm_{E/F})$ with $\chi_F^2=\mathbf{1}$.
\end{rem}

\bibliography{sl2}
\bibliographystyle{alpha}

\end{document}